\documentclass[11pt,a4paper,leqno]{amsart}
\usepackage[headings]{fullpage}
\usepackage{amssymb}
\usepackage{amsthm}
\usepackage{mathrsfs}
\usepackage[all]{xypic}
\usepackage{hyperref}

\newcommand{\id}{\operatorname{id}}
\newcommand{\im}{\operatorname{im}}
\newcommand{\gr}{\operatorname{gr}}
\newcommand{\pr}{\operatorname{pr}}
\newcommand{\divi}{\operatorname{div}}
\newcommand{\supp}{\operatorname{supp}}
\newcommand{\Sp}{\operatorname{Sp}}
\newcommand{\Fr}{\operatorname{Fr}}
\newcommand{\Hom}{\operatorname{Hom}}
\newcommand{\End}{\operatorname{End}}
\newcommand{\Ext}{\operatorname{Ext}}
\newcommand{\Fil}{\operatorname{Fil}}
\newcommand{\Mod}{\operatorname{Mod}}
\newcommand{\Lie}{\operatorname{Lie}}

\newcommand{\Tor}{\operatorname{Tor}}
\newcommand{\Ve}{\operatorname{Vec}}

\newcommand{\Aut}{\operatorname{Aut}}

\newcommand{\Spec}{\operatorname{Spec}}

\newcommand{\GL}{\operatorname{GL}}

\newcommand{\Gal}{\operatorname{Gal}}

\newcommand{\co}{\widehat{\otimes}}
\newcommand{\la}{\mathrm{la}}

\newcommand{\Qp}{\mathbf{Q}_p}

\newcommand{\Cp}{\mathbf{C}_p}
\newcommand{\Zp}{\mathbf{Z}_p}
\newcommand{\ZZ}{\mathbf{Z}}
\newcommand{\QQ}{\mathbf{Q}}

\theoremstyle{plain}

\newtheorem{theorem}{Theorem}[subsection]
\newtheorem{corollary}[theorem]{Corollary}
\newtheorem{lemma}[theorem]{Lemma}
\newtheorem{proposition}[theorem]{Proposition}

\newtheorem{fact}[theorem]{Fact}

\newtheorem*{theoa}{Theorem A}
\newtheorem*{theob}{Theorem B}
\newtheorem*{theoc}{Theorem C}
\newtheorem*{theorec}{Theorem}

\theoremstyle{definition}

\newtheorem{definition}[theorem]{Definition}

\theoremstyle{remark}

\newtheorem{remark}[theorem]{Remark}

\begin{document}

\title{Rigid character groups, Lubin-Tate theory, and $(\varphi,\Gamma)$-modules}

\author{Laurent Berger}
\address{Laurent Berger \\ UMPA de l'ENS de Lyon \\
UMR 5669 du CNRS \\ IUF \\ Lyon \\ France}
\email{laurent.berger@ens-lyon.fr}
\urladdr{http://perso.ens-lyon.fr/laurent.berger/}

\author{Peter Schneider}
\address{Peter Schneider \\ Universit\"at M\"unster \\
Mathematisches Institut \\ M\"unster \\ Germany}
\email{pschnei@uni-muenster.de}
\urladdr{http://wwwmath.uni-muenster.de/u/schneider/}

\author{Bingyong Xie}
\address{Bingyong Xie \\ÊDepartment of Mathematics \\
East China Normal University \\ Shanghai \\ PR China}
\email{byxie@math.ecnu.edu.cn}
\urladdr{http://math.ecnu.edu.cn/~byxie}

\thanks{We acknowledge support by the DFG Sonderforschungsbereich 878 at M\"unster}

\date{\today}

\subjclass[2010]{11F; 11S; 14G; 22E; 46S}

\begin{abstract}
The construction of the $p$-adic local Langlands correspondence for $\GL_2(\Qp)$ uses in an essential way Fontaine's theory of cyclotomic $(\varphi,\Gamma)$-modules. Here \emph{cyclotomic} means that $\Gamma = \Gal(\Qp(\mu_{p^\infty})/\Qp)$ is the Galois group of the cyclotomic extension of $\Qp$. In order to generalize the $p$-adic local Langlands correspondence to $\GL_2(L)$, where $L$ is a finite extension of $\Qp$, it seems necessary to have at our disposal a theory of Lubin-Tate $(\varphi,\Gamma)$-modules. Such a generalization has been carried out to some extent, by working over the $p$-adic open unit disk, endowed with the action of the endomorphisms of a Lubin-Tate group. The main idea of our article is to carry out a Lubin-Tate generalization of the theory of cyclotomic $(\varphi,\Gamma)$-modules in a different fashion. Instead of the $p$-adic open unit disk, we work over a character variety, that parameterizes the locally $L$-analytic characters on $o_L$. We study $(\varphi,\Gamma)$-modules in this setting, and relate some of them to what was known previously.
\end{abstract}

\maketitle

\setlength{\baselineskip}{16pt}

\tableofcontents

\section*{Introduction}

The construction of the $p$-adic local Langlands correspondence for $\GL_2(\Qp)$ (see \cite{ICBM}, \cite{CGL2}, and \cite{LBGL}) uses in an essential way Fontaine's theory \cite{Fon} of cyclotomic $(\varphi,\Gamma)$-modules. Here \emph{cyclotomic} means that $\Gamma = \Gal(\Qp(\mu_{p^\infty})/\Qp)$ is the Galois group of the cyclotomic extension of $\Qp$. These $(\varphi,\Gamma)$-modules are modules over various rings of power series (denoted by $\mathscr{E}$, $\mathscr{E}^\dagger$, and $\mathscr{R}$) that are constructed, by localizing and completing, from the ring $\Qp \otimes_{\Zp} \Zp[[X]]$ of bounded functions on the $p$-adic open unit disk $\mathfrak{B}$. The Frobenius map $\varphi$ and the action of $\Gamma$ on these rings come from the action of the multiplicative monoid $\Zp \setminus \{0\}$ on $\mathfrak{B}$. The relevance of these objects comes from the following theorem (which combines results from \cite{Fon}, \cite{CC98}, and \cite{KSF}).

\begin{theorec}
\label{cyclorecall}
There is an equivalence of categories between the category of $p$-adic representations of $G_{\Qp}$ and the categories of \'etale $(\varphi,\Gamma)$-modules over either $\mathscr{E}$, $\mathscr{E}^\dagger$, or $\mathscr{R}$.
\end{theorec}

In order to generalize the $p$-adic local Langlands correspondence to $\GL_2(L)$, where $L$ is a finite extension of $\Qp$, it seems necessary to have at our disposal a theory of Lubin-Tate $(\varphi,\Gamma)$-modules, where \emph{Lubin-Tate} means that $\Gamma$ is now the Galois group of the field generated over $L$ by the torsion points of a Lubin-Tate group $LT$ attached to a uniformizer of $L$. Such a generalization has been carried out to some extent (see \cite{Fon}, \cite{KR}, \cite{FX} and \cite{PGMLAV}). The resulting $(\varphi,\Gamma)$-modules are modules over rings $\mathscr{E}_L(\mathfrak{B})$, $\mathscr{E}_L^\dagger(\mathfrak{B})$, and $\mathscr{R}_L(\mathfrak{B})$ that are constructed from the ring of bounded analytic functions on the rigid analytic open unit disk $\mathfrak{B}_{/L}$ over $L$, with an $o_L \setminus \{0\}$-action given by the endomorphisms of $LT$. If $M$ is a $(\varphi,\Gamma)$-module over $\mathscr{R}_L(\mathfrak{B})$, the action of $\Gamma$ is differentiable, and we say that $M$ is $L$-analytic if the derived action of $\Lie(\Gamma)$ is $L$-bilinear. We can also define the notion of an $L$-analytic representation of $G_L$. In this setting, the following results are known (see \cite{KR} for (i) and \cite{PGMLAV} for (ii)).

\begin{theorec}
\label{introrecall}
(i) There is an equivalence of categories between the category of $L$-linear continuous representations of $G_L$ and the category of \'etale $(\varphi,\Gamma)$-modules over $\mathscr{E}_L(\mathfrak{B})$.

(ii) There is an equivalence of categories between the category of $L$-linear $L$-analytic representations of $G_L$ and the category of \'etale $L$-analytic $(\varphi,\Gamma)$-modules over $\mathscr{R}_L(\mathfrak{B})$.
\end{theorec}

The main idea of our article is to carry out a Lubin-Tate generalization of the theory of cyclotomic $(\varphi,\Gamma)$-modules in a different fashion. The open unit disk $\mathfrak{B}_{/\Qp}$, with its $\Zp \setminus \{0\}$-action, is naturally isomorphic to the space of locally $\Qp$-analytic characters on $\Zp$. Indeed, if $K$ is an extension of $\Qp$, a point $z \in \mathfrak{B}(K)$ corresponds to the character $\kappa_z : a \mapsto (1+z)^a$ and all $K$-valued continuous characters are of this form. In the Lubin-Tate setting, there exists (\cite{ST}) a rigid analytic group variety $\mathfrak{X}$ over $L$, whose closed points in an extension $K/L$ parameterize locally $L$-analytic characters $o_L \to K^\times$. The ring $\mathcal{O}^b_L(\mathfrak{X})$ of bounded analytic functions on $\mathfrak{X}$ is endowed with an action of the multiplicative monoid $o_L \setminus \{0\}$, and can also be localized and completed. This way we get new analogs $\mathscr{E}_L(\mathfrak{X})$, $\mathscr{E}_L^\dagger(\mathfrak{X})$, and $\mathscr{R}_L(\mathfrak{X})$ of the rings on which the cyclotomic $(\varphi,\Gamma)$-modules are defined, and a corresponding theory of $(\varphi,\Gamma)$-modules.

Note that the varieties $\mathfrak{B}_{/L}$ and $\mathfrak{X}$ are quite different. For instance, the ring $\mathcal{O}_L(\mathfrak{B})$ of rigid analytic functions on $\mathfrak{B}_{/L}$ is isomorphic to the ring of power series in one variable with coefficients in $L$ converging on $\mathfrak{B}(\Cp)$, and is hence a Bezout ring (it is the same for an ideal of $\mathcal{O}_L(\mathfrak{B})$ to be closed, finitely generated, or principal). If $L \neq \Qp$, then in the ring $\mathcal{O}_L(\mathfrak{X})$ there is a finitely generated ideal that is not principal. It is still true, however, that $\mathcal{O}_L(\mathfrak{X})$ is a Pr\"ufer ring (it is the same for an ideal of $\mathcal{O}_L(\mathfrak{X})$ to be closed, finitely generated, or invertible). The beginning of our paper is therefore devoted to establishing geometric properties of $\mathfrak{X}$, and the corresponding properties of $\mathscr{E}_L(\mathfrak{X})$, $\mathscr{E}_L^\dagger(\mathfrak{X})$, and $\mathscr{R}_L(\mathfrak{X})$.

Although the varieties $\mathfrak{B}_{/L}$ and $\mathfrak{X}$ are not isomorphic, they become isomorphic over $\Cp$. This gives rise to an isomorphism $\mathscr{R}_{\Cp}(\mathfrak{B})  = \mathscr{R}_{\Cp}(\mathfrak{X})$ (and likewise for $\mathscr{E}$ and $\mathscr{E}^\dagger$; the subscript ${-}_{\Cp}$ denotes the extension of scalars from $L$ to $\Cp$). In addition, there is an action of $G_L$ on those rings and  $\mathscr{R}_L(\mathfrak{B})  = \mathscr{R}_{\Cp}(\mathfrak{X})^{G_L}$. There is also a ``twisted'' action of $G_L$ and $\mathscr{R}_L(\mathfrak{X})  = \mathscr{R}_{\Cp}(\mathfrak{B})^{G_L,*}$. These isomorphisms make it possible to compare the theories of $(\varphi,\Gamma)$-modules over $\mathscr{R}_L(\mathfrak{B})$ and $\mathscr{R}_L(\mathfrak{X})$, by extending scalars to $\Cp$ and descending. We construct two functors $M \mapsto M_{\mathfrak{X}}$ and $N \mapsto N_{\mathfrak{B}}$ from the category of $(\varphi,\Gamma)$-modules over $\mathscr{R}_L(\mathfrak{B})$ to the category of $(\varphi,\Gamma)$-modules over $\mathscr{R}_L(\mathfrak{X})$ and vice versa.

\begin{theoa}
\label{theoA}
The functors $M \mapsto M_{\mathfrak{X}}$ and $N \mapsto N_{\mathfrak{B}}$ from the category of $L$-analytic $(\varphi,\Gamma)$-modules over $\mathscr{R}_L(\mathfrak{B})$ to the category of $L$-analytic $(\varphi,\Gamma)$-modules over $\mathscr{R}_L(\mathfrak{X})$ and vice versa, are mutually inverse and give rise to equivalences of categories.
\end{theoa}

\begin{theob}
\label{theoB}
The functors $M \mapsto M_{\mathfrak{X}}$ and $N \mapsto N_{\mathfrak{B}}$ preserve degrees, and give equivalences of categories between the categories of \'etale objects on both sides. There is an equivalence of categories between the category of $L$-linear $L$-analytic representations of $G_L$ and the category of \'etale $L$-analytic $(\varphi,\Gamma)$-modules over $\mathscr{R}_L(\mathfrak{X})$.
\end{theob}

One way of constructing cyclotomic $(\varphi,\Gamma)$-modules over $\mathscr{R}$ is to start from a filtered $\varphi$-module $D$, and to perform a modification of $\mathscr{R} \otimes_{\Qp} D$ according to the filtration on $D$ (see \cite{BEQ}). The generalization of this construction to $\mathscr{R}_L(\mathfrak{B})$ has been done in \cite{KR} and \cite{KFC}: They attach to every filtered $\varphi_L$-module $D$ over $L$ a $(\varphi_L,\Gamma_L)$-module $\mathcal{M}_{\mathfrak{B}}(D)$ over $\mathcal{O}_L(\mathfrak{B})$ (we can then extend scalars to $\mathscr{R}_L(\mathfrak{B})$). We carry out the corresponding construction over $\mathcal{O}_L(\mathfrak{X})$, and show the following.

\begin{theoc}
\label{theoC}
(i) The functor $D \mapsto \mathcal{M}_{\mathfrak{X}}(D)$ provides an equivalence of categories between the category of filtered $\varphi_L$-modules over $L$ and the category $\Mod^{\varphi_L,\Gamma_L,\mathrm{an}}_{/\mathfrak{X}}$ of $(\varphi_L,\Gamma_L)$-modules $M$ over $\mathcal{O}_L(\mathfrak{X})$ defined in \ref{defmodw}.

(ii) The functors $D \mapsto \mathscr{R}_L(\mathfrak{B}) \otimes_{\mathcal{O}_L(\mathfrak{B})} \mathcal{M}_{\mathfrak{B}}(D)$ and $D \mapsto \mathscr{R}_L(\mathfrak{X}) \otimes_{\mathcal{O}_L(\mathfrak{X})} \mathcal{M}_{\mathfrak{X}}(D)$ are compatible with the equivalence of categories of Theorem A.
\end{theoc}

\subsection*{Notations and preliminaries}

Let $\Qp \subseteq L \subseteq K \subseteq \Cp$ be fields with $L/\Qp$ finite Galois and $K$ complete. Let $d := [L:\Qp]$, $o_L$ the ring of integers of $L$, and $\pi_L \in o_L$ a fixed prime element. Let $q$ be the cardinality of $k_L := o_L/\pi_L$ and let $e$ be the ramification index of $L$. We always use the absolute value $|\ |$ on $\Cp$ which is normalized by $|p| = p^{-1}$.

For any locally $L$-analytic manifold $M$ we let $M_0$ denote $M$ but viewed as a locally $\Qp$-analytic manifold.

For any rigid analytic variety $\mathfrak{Y}$ over $L$, let $\mathcal{O}_K(\mathfrak{Y})$ denote the ring of $K$-valued global holomorphic functions on $\mathfrak{Y}$. Suppose that $\mathfrak{Y}$ is reduced. Then a function $f \in \mathcal{O}_K(\mathfrak{Y})$ is called bounded if there is a real constant $C > 0$ such that $|f(y)| < C$ for any rigid point $y \in \mathfrak{Y}(\Cp)$.

\begin{remark}
Let $A$ be an affinoid algebra over $K$ whose base extension $A_{\Cp} := A \widehat{\otimes}_K \Cp$ is a reduced affinoid algebra over $\Cp$. Then $A$ is reduced as well, and the corresponding supremum norms $\|\ \|_{\sup,K}$ on $A$ and $\|\ \|_{\sup, \Cp}$ on $A_{\Cp}$ are characterized as being the only power-multiplicative complete algebra norms on $A$ and $A_{\Cp}$, respectively (\cite{BGR} Lemma 3.8.3/3 and Thm.\ 6.2.4/1). It follows that $\|\ \|_{\sup, \Cp} \mid_A = \|\ \|_{\sup,K}$.
\end{remark}

On the subring
\begin{equation*}
    \mathcal{O}_K^b(\mathfrak{Y}) := \{f \in \mathcal{O}_K(\mathfrak{Y}) : f\ \text{is bounded}\}
\end{equation*}
we have the supremum norm
\begin{equation*}
    \|f\|_\mathfrak{Y} = \sup_{y \in \mathfrak{Y}(\Cp)} |f(y)|.
\end{equation*}
One checks that $(\mathcal{O}_K^b(\mathfrak{Y}), \|\ \|_\mathfrak{Y})$ is a Banach $K$-algebra.

\section{Lubin-Tate theory and the character variety}

\subsection{The ring of global functions on a one dimensional smooth quasi-Stein space}\label{sec:prufer}

Let $\mathfrak{Y}$ be a rigid analytic variety over $K$ which is smooth, one dimensional, and quasi-Stein. In this section we will establish the basic ring theoretic properties of the ring $\mathcal{O}_K(\mathfrak{Y})$. For simplicity we always will assume that $\mathfrak{Y}$ is connected in the sense that $\mathcal{O}_K(\mathfrak{Y})$ is an integral domain. We recall that the quasi-Stein property means that $\mathfrak{Y}$ has an admissible covering $\mathfrak{Y} = \bigcup_{n} \mathfrak{Y}_n$ by an increasing sequence $\mathfrak{Y}_1 \subseteq \ldots \subseteq \mathfrak{Y}_n \subseteq \ldots$ of affinoid subdomains such that the restriction maps $\mathcal{O}_K(\mathfrak{Y}_{n+1}) \longrightarrow \mathcal{O}_K(\mathfrak{Y}_n)$, for any $n \geq 1$, have dense image.

An effective divisor on $\mathfrak{Y}$ is a function $\Delta : \mathfrak{Y} \longrightarrow \ZZ_{\geq 0}$ such that, for any affinoid subdomain $\mathfrak{Z} \subseteq \mathfrak{Y}$, the set $\{x \in \mathfrak{Z} : \Delta(x) > 0\}$ is finite. The support $\supp(\Delta) := \{x \in \mathfrak{Y} : \Delta(x) > 0\}$ of a divisor is always a countable subset of $\mathfrak{Y}$. The set of effective divisors is partially ordered by $\Delta \leq \Delta'$ if and only if $\Delta(x) \leq \Delta'(x)$ for any $x \in \mathfrak{Y}$.

For simplicity let $\mathcal{O}$ denote the structure sheaf of the rigid variety $\mathfrak{Y}$ and $\mathcal{O}_x$ its stalk in a point $x \in \mathfrak{Y}$ with maximal ideal $\mathfrak{m}_x$. Since $\mathfrak{Y}$ is smooth and one dimensional every $\mathcal{O}_x$ is a discrete valuation ring. For any element $0 \neq f \in \mathcal{O}_K(\mathfrak{Y})$ we define the function $\divi(f) : \mathcal{O}_K(\mathfrak{Y}) \longrightarrow \ZZ_{\geq 0}$ by $\divi(f)(x) = n$ if and only if the germ of $f$ in $x$ is contained in $\mathfrak{m}_x^n \setminus \mathfrak{m}_x^{n+1}$.

\begin{lemma}\label{principal-divisor}
$\divi(f)$ is a divisor.
\end{lemma}
\begin{proof}
Let $\Sp(A) \subseteq \mathfrak{Y}$ be any connected affinoid subdomain. Then $\Sp(A)$ is smooth and one dimensional. Hence the affinoid algebra $A$ is a Dedekind domain. But in a Dedekind domain any nonzero element is contained in at most finitely many maximal ideals.
\end{proof}

For any effective divisor $\Delta$ we have the ideal $I_\Delta := \{f \in \mathcal{O}_K(\mathfrak{Y}) \setminus \{0\} : \divi(f) \geq \Delta\} \cup \{0\}$ in $\mathcal{O}_K(\mathfrak{Y})$. It also follows that for any nonzero ideal $I \subseteq \mathcal{O}_K(\mathfrak{Y})$ we have the effective divisor $\Delta(I)(x) := \min_{0 \neq f \in I} \divi(f)(x)$.

As the projective limit $\mathcal{O}_K(\mathfrak{Y}) = \varprojlim_n \mathcal{O}_K(\mathfrak{Y}_n)$ of the Banach algebras $\mathcal{O}_K(\mathfrak{Y}_n)$ the ring $\mathcal{O}_K(\mathfrak{Y})$ has a natural structure of a $K$-Fr\'echet algebra (which is independent of the choice of the covering $(\mathfrak{Y}_n)_n$). In fact,  $\mathcal{O}_K(\mathfrak{Y})$ is a Fr\'echet-Stein algebra in the sense of \cite{ST0}. Any closed ideal $I \subseteq \mathcal{O}_K(\mathfrak{Y})$ is the space of global sections of a unique coherent ideal sheaf $\widetilde{I} \subseteq \mathcal{O}$. Let $\widetilde{I}_x \subseteq \mathcal{O}_x$ denote the stalk of $\widetilde{I}$ in the point $x$.

\begin{lemma}\label{germ}
For any closed ideal $I \subseteq \mathcal{O}_K(\mathfrak{Y})$ and any point $x \in \mathfrak{Y}$ the map $I \longrightarrow \widetilde{I}_x/\mathfrak{m}_x \widetilde{I}_x$ induced by passing to the germ in $x$ is surjective, continuous, and open, if we equip the finite dimensional $K$-vector space $\widetilde{I}_x/\mathfrak{m}_x \widetilde{I}_x$ with its unique Hausdorff vector space topology.
\end{lemma}
\begin{proof}
Let $\Sp(A) = \mathfrak{Y}_n \subseteq \mathfrak{Y}$ be such that it contains the point $x$, and let $\mathfrak{m} \subseteq A$ be the maximal ideal of functions vanishing in $x$. By \cite{BGR} Prop.\ 7.3.2/3 we have $AI/\mathfrak{m}AI = \widetilde{I}_x/\mathfrak{m}_x \widetilde{I}_x$. The continuity of the map in question follows since the restriction map $I \longrightarrow AI$ is continuous and any ideal in the Banach algebra $A$ is closed (\cite{BGR} Prop.\ 6.1.1/3). Moreover, Theorem A for the quasi-Stein space $\mathfrak{Y}$ says that the map $I \longrightarrow AI$ has dense image (cf.\ \cite{ST0} Thm.\ in \S3, Cor.\ 3.1, and Remark 3.2) which implies the asserted surjectivity. The openness then follows from the uniqueness of the Hausdorff topology on a finite dimensional vector space.
\end{proof}

\begin{lemma}\label{prescribe-divisor}
Let $I \subseteq \mathcal{O}_K(\mathfrak{Y})$ be any nonzero closed ideal and let $Z \subseteq \mathfrak{Y}$ be any countable subset disjoint from $\supp(\Delta(I))$; then there exists a function $f \in I$ such that
\begin{equation*}
    \divi(f)(x) =
    \begin{cases}
    \Delta(I)(x) & \text{if $x \in \supp(\Delta(I))$}, \\
    0 & \text{if $x \in Z$}.
    \end{cases}
\end{equation*}
\end{lemma}
\begin{proof}
It follows from Lemma \ref{germ} that $V(x) := I \setminus \mathfrak{m}_x \widetilde{I}_x$, for any point $x \in \mathfrak{Y}$, is an open subset of $I$. It is even dense in $I$: On the one hand, the complement $I \setminus \overline{V(x)}$ of its closure, being contained in $\mathfrak{m}_x \widetilde{I}_x$, projects into the subset $\{0\} \subseteq \widetilde{I}_x/\mathfrak{m}_x \widetilde{I}_x$. On the other hand, by Lemma \ref{germ}, the image of the open set $I \setminus \overline{V(x)}$ is open in $\widetilde{I}_x/\mathfrak{m}_x \widetilde{I}_x$. But $I$ being nonzero the vector space $\widetilde{I}_x/\mathfrak{m}_x \widetilde{I}_x$ is one dimensional so that the subset $\{0\}$ is not open. Hence we must have $\overline{V(x)} = I$.

The topology of $I$ being complete and metrizable the Baire category theorem implies that, for any countable number of points $x$ the intersection of the corresponding $V(x)$ still is dense in $I$ and, in particular, is nonempty. Since $V(x) = \{f \in I : \divi(f)(x) = \Delta(I)(x)\}$ any $f \in \bigcap_{x \in \supp(\Delta(I)) \cup Z} V(x)$ satisfies the assertion.
\end{proof}

\begin{proposition}\label{divisor-theory}
The map $\Delta \longmapsto I_\Delta$ is an order reversing bijection between the set of all effective divisors on $\mathfrak{Y}$ and the set of all nonzero closed ideals in $\mathcal{O}_K(\mathfrak{Y})$. The inverse is given by $I \longmapsto \Delta(I)$.
\end{proposition}
\begin{proof}
The map $I \longmapsto \widetilde{I}$ is a bijection between the set of all nonzero closed ideals and the set of all nonzero coherent ideal sheaves in $\mathcal{O}$. Let $\mathfrak{Y}_n = \Sp(A_n)$. Then a coherent ideal sheaf $\mathcal{I}$ is the same as a sequence $(I_n)_{n \geq 1}$ where $I_n$ is an ideal in $A_n$ such that $I_{n+1} A_n = I_n$ for any $n \geq 1$ the bijection being given by $\mathcal{I} \longmapsto (\mathcal{I}(\Sp(A_n))_n$. On the other hand, an effective divisor on $\mathfrak{Y}$ is the same as a sequence $(\Delta_n)_{n \geq 1}$ where $\Delta_n$ is an effective divisor supported on $\Sp(A_n)$ such that $\Delta_{n+1} | \Sp(A_n) = \Delta_n$ for any $n \geq 1$. Since $I_\Delta = \varprojlim_n I_{\Delta|\Sp(A_n)}$ this reduces the asserted bijection to the case of the affinoid varieties $\Sp(A_n)$. But each $A_n$ is a Dedekind ring for which the bijection between nonzero ideals and effective divisors is well known. The description of the inverse bijection follows from Lemma \ref{germ} which implies that $\min_{0 \neq f \in I} \divi(f)(x) = \min_{0 \neq f \in A_n I} \divi(f)(x)$ provided $x \in \Sp(A_n)$.
\end{proof}

\begin{lemma}\label{fingensub}
Any finitely generated submodule in a finitely generated free $\mathcal{O}_K(\mathfrak{Y})$-module is closed.
\end{lemma}
\begin{proof}
(Any finitely generated free $\mathcal{O}_K(\mathfrak{Y})$-module carries the product topology which is easily seen to be independent of the choice of a basis.) Our assertion is a general fact about Fr\'echet-Stein algebras (cf.\ \cite{ST0} Cor.\ 3.4.iv and Lemma 3.6).
\end{proof}

\begin{proposition}\label{closed-fingen}
An ideal $I \subseteq \mathcal{O}_K(\mathfrak{Y})$ is closed if and only it is finitely generated.
\end{proposition}
\begin{proof}
If $I$ is finitely generated then it is closed by the previous Lemma \ref{fingensub}. Suppose therefore that $I$ is closed. We apply Lemma \ref{prescribe-divisor} twice, first to obtain a function $f \in I$ such that $\divi(f) = \Delta(I) + \Delta_1$ with $\supp(\Delta(I)) \cap \supp(\Delta_1) = \emptyset$ and then to obtain another function $g \in I$ such that $\divi(g) = \Delta(I) + \Delta_2$ with $\big( \supp(\Delta(I)) \cup \supp(\Delta_1) \big) \cap \supp(\Delta_2) = \emptyset$. By the first part of the assertion the ideal $(f,g)$ is closed. Its divisor, by construction, must be equal to $\Delta(I)$. It follows that $I = (f,g)$.
\end{proof}

\begin{proposition}\label{closed-invert}
A nonzero ideal $I \subseteq \mathcal{O}_K(\mathfrak{Y})$ is closed if and only it is invertible.
\end{proposition}
\begin{proof}
In any ring invertible ideals are finitely generated. Hence $I$ is closed if it is invertible by Prop.\ \ref{closed-fingen}. Now assume that $I \neq 0$ is closed. Again by \ref{closed-fingen} we have $I = (g_1, \ldots, g_m)$ for appropriate functions $g_i$. Fix any nonzero function $f \in I$. Then $\divi(f) - \Delta(I)$ is an effective divisor and the closed ideal $J := I_{\divi(f) - \Delta(I)}$ is defined. Then $IJ = \sum_i g_i J$. Using once more \ref{closed-fingen} we see that $J$ and therefore all the ideals $g_iJ$ as well as $\sum_i g_i J$ are finitely generated and hence closed. We have $\Delta(g_iJ) = \divi(f) - \Delta(I) + \divi(g_i)$, and we conclude that
\begin{align*}
    \Delta(\sum_i g_iJ) & = \min_i \Delta(g_iJ) = \min_i (\divi(f) - \Delta(I) + \divi(g_i)) \\
    & = \divi(f) - \Delta(I) + \min_i \divi(g_i) = \divi(f) - \Delta(I) + \Delta(I) \\
    & = \divi(f) \ .
\end{align*}
This implies $IJ = (f)$ and then $I (f^{-1}J) = (1)$ which shows that $I$ is invertible.
\end{proof}

\begin{corollary}\label{pruefer}
$\mathcal{O}_K(\mathfrak{Y})$ is a Pr\"ufer domain; in particular, we have:
\begin{itemize}
  \item[--] $\mathcal{O}_K(\mathfrak{Y})$ is a coherent ring;
  \item[--] an $\mathcal{O}_K(\mathfrak{Y})$-module is flat if and only if it is torsionfree;
  \item[--] any finitely generated torsionfree $\mathcal{O}_K(\mathfrak{Y})$-module is projective;
  \item[--] any finitely generated projective $\mathcal{O}_K(\mathfrak{Y})$-module is isomorphic to a direct sum of (finitely generated) ideals (\cite{CE} Prop.\ I.6.1).
\end{itemize}
\end{corollary}

\begin{lemma}\label{closed-proj}
Any closed submodule of a finitely generated projective $\mathcal{O}_K(\mathfrak{Y})$-module is finitely generated projective.
\end{lemma}
\begin{proof}
Because of Lemma \ref{fingensub} we may assume that $P$ is a closed submodule of some free module $\mathcal{O}_K(\mathfrak{Y})^m$. We now prove the assertion by induction with respect to $m$. For $m=1$ this is Prop.\ \ref{closed-fingen} and Cor.\ \ref{pruefer}. In general we consider the exact sequence
\begin{equation*}
    0 \longrightarrow P \cap \mathcal{O}_K(\mathfrak{Y})^{m-1} \longrightarrow P \xrightarrow{\; \pr_m \;} \mathcal{O}_K(\mathfrak{Y}) \ .
\end{equation*}
By the induction hypothesis the left term is finitely generated projective. On the other hand, by \cite{ST0} Cor.\ 3.4.ii and Lemma 3.6 the image of $\pr_m$ is closed and therefore finitely generated projective as well (by the case $m=1$).
\end{proof}

In fact, $\mathcal{O}_K(\mathfrak{Y})$ is a Pr\"ufer domain of a special kind. We recall that a Pr\"ufer domain $R$ is called a $1 \frac{1}{2}$ generator Pr\"ufer domain if for any nonzero finitely generated ideal $I$ of $R$ and any $0 \neq f \in I$, there exists another element $g \in I$ such that $f$ and $g$ generate $I$.

\begin{proposition}\label{1 1/2}
The ring $\mathcal{O}_K(\mathfrak{Y})$ is a $1 \frac{1}{2}$ generator Pr\"ufer domain.
\end{proposition}
\begin{proof}
Let $I \subseteq \mathcal{O}_K(\mathfrak{Y})$ be an arbitrary nonzero finitely generated ideal and $0 \neq f \in I$ be any element. Our assertion, by definition, amounts to the claim that there exists another element $g \in I$ such that $f$ and $g$ generate $I$.

As before, let $\widetilde{I} \subseteq \mathcal{O}$ be the coherent ideal sheaf corresponding to $I$. By Thm.\ B we have $(\mathcal{O}/\widetilde{I})(\mathfrak{Y}) = \mathcal{O}_K(\mathfrak{Y})/I$. According to \cite{BGR} Prop.\ 9.5.3/3 the quotient sheaf $\mathcal{O}/\widetilde{I}$ is (the direct image of) the structure sheaf of a (in general nonreduced) structure of a rigid analytic variety on the analytic subset $\supp(\Delta(I))$ of $\mathfrak{Y}$. But topologically $\supp(\Delta(I))$ is a discrete set. We deduce that
\begin{equation*}
  \mathcal{O}_K(\mathfrak{Y})/I = (\mathcal{O}/\widetilde{I})(\mathfrak{Y}) = \prod_{x \in \mathfrak{Y}} \mathcal{O}_x/\mathfrak{m}_x^{\Delta(I)(x)} \ .
\end{equation*}
We now repeat the corresponding observation for the principal ideal $J := f \mathcal{O}_K(\mathfrak{Y})$. By comparing the two computations we obtain
\begin{equation*}
  I/f \mathcal{O}_K(\mathfrak{Y}) = \prod_{x \in \mathfrak{Y}} \mathfrak{m}_x^{\Delta(I)(x)}/\mathfrak{m}_x^{\divi(f)(x)} \subseteq \prod_{x \in \mathfrak{Y}} \mathcal{O}_x/\mathfrak{m}_x^{\divi(f)(x)} = \mathcal{O}_K(\mathfrak{Y})/ f \mathcal{O}_K(\mathfrak{Y}) \ .
\end{equation*}
This shows that for $g$ we may take any function in $I$ whose germ generates $\mathfrak{m}_x^{\Delta(I)(x)}/\mathfrak{m}_x^{\divi(f)(x)}$ for any $x \in \mathcal{O}_K(\mathfrak{Y})$.
\end{proof}

\begin{corollary}\phantomsection\label{free-invertible}
\begin{itemize}
  \item[i.] If $I \subseteq \mathcal{O}_K(\mathfrak{Y})$ is any nonzero finitely generated ideal, then in the factor ring $\mathcal{O}_K(\mathfrak{Y})/I$ every finitely generated ideal is principal.
  \item[ii.]  For any two nonzero finitely generated ideals $I$ and $J$ in $\mathcal{O}_K(\mathfrak{Y})$ we have an isomorphism of $\mathcal{O}_K(\mathfrak{Y})$-modules $I \oplus J \cong \mathcal{O}_K(\mathfrak{Y}) \oplus IJ$.
\end{itemize}
\end{corollary}
\begin{proof}
i. This is immediate from the $1 \frac{1}{2}$ generator property. ii. Because of i. this is \cite{Kap} Thm.\ 2(a).
\end{proof}

We also recall the following facts about coherent modules on a quasi-Stein space, which will be used, sometimes implicitly, over and over again.

\begin{fact}\label{coherent}
Let $\mathfrak{M}$ be a coherent $\mathcal{O}$-module sheaf on $\mathfrak{Y}$ and set $M := \mathfrak{M}(\mathfrak{Y})$; then:
\begin{itemize}
  \item[i.] For any open affinoid subvariety $\mathfrak{Z} \subseteq \mathfrak{Y}$ we have $\mathfrak{M}(\mathfrak{Z}) = \mathcal{O}(\mathfrak{Z}) \otimes_{\mathcal{O}_K(\mathfrak{Y})} M$;
  \item[ii.] for any point $x \in \mathfrak{Y}$ the stalk of $\mathfrak{M}$ in $x$ is $\mathfrak{M}_x = \mathcal{O}_x \otimes _{\mathcal{O}_K(\mathfrak{Y})} M$;
  \item[iii.] $\mathfrak{M} = 0$ if and only if $M = 0$ if and only if $\mathfrak{M}_x = 0$ for any point $x$;
  \item[iv.] if $M$ is a finitely generated projective $\mathcal{O}_K(\mathfrak{Y})$-module then $\mathfrak{M}(\mathfrak{Z}) = \mathcal{O}(\mathfrak{Z}) \otimes_{\mathcal{O}_K(\mathfrak{Y})} M$ for any admissible open subvariety $\mathfrak{Z} \subseteq \mathfrak{Y}$.
\end{itemize}
\end{fact}
\begin{proof}
i. By the quasi-Stein property the claim holds true for any $\mathfrak{Z} = \mathfrak{Y}_n$ (cf.\ \cite{ST0} Cor.\ 3.1). A general open affinoid $\mathfrak{Z}$ is contained in some $\mathfrak{Y}_n$, and we have
\begin{align*}
    \mathfrak{M}(\mathfrak{Z}) & = \mathcal{O}(\mathfrak{Z}) \otimes_{\mathcal{O}_K(\mathfrak{Y}_n)} \mathfrak{M}(\mathfrak{Y}_n) = \mathcal{O}(\mathfrak{Z}) \otimes_{\mathcal{O}_K(\mathfrak{Y}_n)} (\mathcal{O}_K(\mathfrak{Y}_n) \otimes_{\mathcal{O}_K(\mathfrak{Y})} M) \\
    & = \mathcal{O}(\mathfrak{Z}) \otimes_{\mathcal{O}_K(\mathfrak{Y})} M
\end{align*}
(using \cite{BGR} Prop.\ 9.4.2/1 for the first identity). ii. This follows immediately from i. by passing to the direct limit. iii. Use \cite{BGR} Cor.\ 9.4.2/7. iv. Choose admissible affinoid coverings $\mathfrak{Z} = \bigcup_i \mathfrak{Z}_i$ and $\mathfrak{Z}_i \cap \mathfrak{Z}_j = \bigcup_\ell \mathfrak{Z}_{i,j,\ell}$. The sheaf axiom gives the upper exact sequence
\begin{equation*}
    \xymatrix{
       0 \ar[r] & \mathfrak{M}(\mathfrak{Z}) \ar[r] & \prod_i \mathfrak{M}(\mathfrak{Z}_i)  \ar[r] & \prod_{i,j,\ell} \mathfrak{M}(\mathfrak{Z}_{i,j,\ell})  \\
       0 \ar[r] & \mathcal{O}(\mathfrak{Z}) \otimes_{\mathcal{O}_K(\mathfrak{Y})} M  \ar[u] \ar[r] & \prod_i (\mathcal{O}(\mathfrak{Z}_i) \otimes_{\mathcal{O}_K(\mathfrak{Y})} M) \ar[u]^{\cong} \ar[r] & \prod_{i,j,\ell} (\mathcal{O}(\mathfrak{Z}_{i,j,\ell}) \otimes_{\mathcal{O}_K(\mathfrak{Y})} M) \ar[u]^{\cong}. }
\end{equation*}
Since the tensor product by a finitely generated projective module is exact and commutes with arbitrary direct products the lower sequence is exact as well. The middle and the right arrow are bijective by i. So the first arrow must be bijective as well.
\end{proof}

\begin{proposition}\label{gruson}
The global section functor $\mathfrak{M} \longmapsto \mathfrak{M}(\mathfrak{Y})$ induces an equivalence of categories between the category of locally free coherent $\mathcal{O}$-module sheaves on $\mathfrak{Y}$ and the category of finitely generated projective $\mathcal{O}_K(\mathfrak{Y})$-modules.
\end{proposition}
\begin{proof}
This is a special case of \cite{Gru} Thm.\ V.1 and subsequent Remark.
\end{proof}

\subsection{The variety $\mathfrak{X}$ and the ring $\mathcal{O}_K(\mathfrak{X})$}

Let $G := o_L$ denote the additive group $o_L$ viewed as a locally $L$-analytic group. The group of $K$-valued locally analytic characters of $G$ is denoted by $\widehat{G}(K)$.

We have the bijection
\begin{align}\label{f:variety}
    \mathfrak{B}_1(K) \otimes_{\Zp} \Hom_{\Zp}(o_L,\Zp) & \xrightarrow{\; \sim \;} \widehat{G}_0(K) \\
    z \otimes \beta & \longmapsto \chi_{z \otimes \beta} (g) := z^{\beta(g)} \nonumber
\end{align}
where $\mathfrak{B}_1$ is the rigid $\Qp$-analytic open disk of radius one around the point $1 \in \Qp$. The rigid analytic group variety
\begin{equation*}
    \mathfrak{X}_0 := \mathfrak{B}_1 \otimes_{\Zp} \Hom_{\Zp}(o_L,\Zp)
\end{equation*}
over $\Qp$ (noncanonically a $d$-dimensional open unit polydisk) satisfies $\mathfrak{X}_0(K) = \mathfrak{X}_{0/L}(K) = \widehat{G}_0(K)$ and therefore ``represents the character group $\widehat{G}_0$''. It is shown in \cite{ST} \S2 that, if $t_1, \ldots, t_d$ is a $\Zp$-basis of $o_L$, then the equations
\begin{equation*}
    (\beta(t_i) - t_i \cdot \beta(1)) \cdot \log(z) = 0 \qquad\text{for $1 \leq i \leq d$}
\end{equation*}
define a one dimensional rigid analytic subgroup $\mathfrak{X}$ in $\mathfrak{X}_{0/L}$ which ``represents the character group $\widehat{G}$''.\footnote{Even more explicitly, if we use the basis $\beta_1, \ldots, \beta_d$ dual to $t_i$ in order to identify $\mathfrak{X}_0$ with $\mathfrak{B}_1^d$, then $\mathfrak{X}$ identifies with
\begin{multline*}
    \{(z_1, \ldots, z_d) \in \mathfrak{B}_{1/L}^d : \sum_j \beta_j(1) \log(z_j) = \frac{1}{t_i} \log(z_i)\ \text{for $1 \leq i \leq d$}\} \\ = \{(z_1, \ldots, z_d) \in \mathfrak{B}_{1/L}^d : \log(z_i) = \frac{t_i}{t_1} \log(z_1)\ \text{for $1 \leq i \leq d$}\} \ .
\end{multline*}}
In particular, restriction of functions induces a surjective homomorphism of rings
\begin{equation*}
    \mathcal{O}_K(\mathfrak{X}_0) \longrightarrow \mathcal{O}_K(\mathfrak{X})
\end{equation*}
which further restricts to a ring homomorphism
\begin{equation*}
    \mathcal{O}_K^b(\mathfrak{X}_0) \longrightarrow \mathcal{O}_K^b(\mathfrak{X}) \ .
\end{equation*}
There is the following simple, but remarkable fact.

\begin{lemma}\label{injective}
The restriction map $\mathcal{O}_K^b(\mathfrak{X}_0) \longrightarrow \mathcal{O}_K^b(\mathfrak{X})$ is injective.
\end{lemma}
\begin{proof}
By Fourier theory (cf.\ \cite{ST} and \cite{Schi} App.\ A.6) the vector space $\mathcal{O}_K(\mathfrak{X})$ is the continuous dual of the locally convex vector space of locally $L$-analytic functions $C^{an}(G,K)$ whereas $\mathcal{O}_K^b(\mathfrak{X}_0)$ is the continuous dual of the $K$-Banach space $C^{cont}(G,K)$ of all continuous functions on $G$. It therefore suffices to show that $C^{an}(G,K)$ is dense in $C^{cont}(G,K)$. In fact, already the vector space $C^\infty(G,K)$ of all locally constant functions, which obviously are locally $L$-analytic, is dense in $C^{cont}(G,K)$. This is a well known fact about functions on compact totally disconnected groups. Let $\epsilon > 0$ be any positive constant and $f \in C^{cont}(G,K)$ be any function. By continuity of $f$ we find for any point $x \in G$ an open neighbourhood $U_x \subseteq G$ such that $|f(y) - f(x)| \leq \epsilon$ for any $y \in U_x$. The covering $(U_x)_x$ of $G$ can be refined into a finite disjoint open covering $G = V_1 \dot{\cup} \ldots \dot{\cup} V_m$. Pick points $x_i \in V_i$ and define the locally constant function $f_\epsilon$ by requiring that $f_\epsilon | V_i$ has constant value $f(x_i)$. Then $\|f - f_\epsilon\| \leq \epsilon$ in the sup-norm of the Banach space $C^{cont}(G,K)$.
\end{proof}

For any $r \in (0,1) \cap p^\QQ$ we let $\mathfrak{B}_1(r)$, resp.\ $\mathfrak{B}(r)$, denote the $\Qp$-affinoid disk of radius $r$ around $1$, resp.\ around $0$, and we put $\mathfrak{X}(r) := \mathfrak{X} \cap (\mathfrak{B}_1(r) \otimes_{\Zp} \Hom_{\Zp}(o_L,\Zp))_{/L}$. In fact, each $\mathfrak{X}(r)$ is an affinoid subgroup of $\mathfrak{X}$. For small $r$ its structure is rather simple. By \cite{ST} Lemma 2.1 we have the cartesian diagram of rigid $L$-analytic varieties
\begin{equation*}
    \xymatrix{
      \mathfrak{X}_{\vphantom{\ZZ_i}} \ar[d]_{d} \ar[r]^-{\subseteq} & \mathfrak{B}_1 \otimes \Hom_{\Zp} (o_L,\Zp) \ar[d]^{\log \otimes \id} \\
      \mathbb{A}^1_{\vphantom{\ZZ}} \ar[r]^-{\subseteq} & \mathbb{A}^1 \otimes \Hom_{\Zp} (o_L,\Zp)   }
\end{equation*}
where $d$ is the morphism which sends a locally analytic character $\chi$ of $o_L$ to its derivative $d\chi(1) = \frac{d}{dt}\chi(t) {\mid}_{t=0}$ and where the lower horizontal arrow is the map $a \longmapsto \sum_{i=1}^d at_i \otimes \beta_i$ with $\beta_1, \ldots, \beta_d$ being the basis dual to $t_1, \ldots, t_d$. Consider any $r < p^{-\frac{1}{p-1}}$. Then the logarithm restricts to an isomorphism $\mathfrak{B}_1(r) \xrightarrow{\cong} \mathfrak{B}(r)$ with inverse the exponential map. The above diagram restricts to the cartesian diagram with vertical isomorphisms
\begin{equation*}
    \xymatrix{
      \mathfrak{X}(r)_{\vphantom{\ZZ_i}} \ar[d]_{d}^{\cong} \ar[r]^-{\subseteq} & \mathfrak{B}_1(r) \otimes \Hom_{\Zp} (o_L,\Zp) \ar[d]^{\log \otimes \id}_{\cong} \\
      \mathfrak{B}(r)_{\vphantom{\ZZ}} \ar[r]^-{\subseteq} & \mathfrak{B}(r) \otimes \Hom_{\Zp} (o_L,\Zp).   }
\end{equation*}

\begin{lemma}\label{small-disk}
For any $r \in (0,p^{-\frac{1}{p-1}}) \cap p^\QQ$ the map
\begin{align*}
    \mathfrak{B}(r) & \xrightarrow{\;\cong\;} \mathfrak{X}(r) \\
    y & \longmapsto \chi_y(g) := \exp(gy)
\end{align*}
is an isomorphism of $L$-affinoid groups.
\end{lemma}

\begin{proof}
Obviously we have $\chi_y \in \mathfrak{X}(r)$ and $d\chi_y(1) = y$. Hence the map under consideration is the inverse of the isomorphism $d$ in the above diagram.
\end{proof}

\subsection{The LT-isomorphism}\label{sec:LT}

Let $LT = LT_{\pi_L}$ be the Lubin-Tate formal $o_L$-module over $o_L$ corresponding to the prime element $\pi_L$. If $\mathfrak{B}$ denotes the rigid $\Qp$-analytic open disk of radius one around $0 \in \Qp$, then we always identify $LT$ with $\mathfrak{B}_{/L}$, which makes the rigid variety $\mathfrak{B}_{/L}$ into an $o_L$-module object and which gives us a global coordinate $Z$ on $LT$. The resulting $o_L$-action on $\mathfrak{B}_{/L}$ denoted by $(a,z) \longmapsto [a](z)$ is given by formal power series $[a](Z) \in o_L[[Z]]$. In this way, in particular, the multiplicative monoid $o_L \setminus \{0\}$ acts on the rigid group variety $\mathfrak{B}_{/L}$. For any $r \in (0,1) \cap p^\QQ$ the $L$-affinoid disk $\mathfrak{B}(r)_{/L}$ of radius $r$ around zero is an $o_L$-submodule object of $\mathfrak{B}_{/L}$.

Let $T$ be the Tate module of $LT$. Then $T$ is
a free $o_L$-module of rank one, and the action of
$\Gal(\overline{L}/L)$ on $T$ is given by a continuous character $\chi_{LT} :
\Gal(\overline{L}/L) \longrightarrow o_L^\times$. Let $T'$ denote the Tate module of the $p$-divisible group dual to $LT$, which again is a free $o_L$-module of rank one. The Galois action on $T'$ is given by the continuous character $\tau := \chi_{cyc}\cdot\chi_{LT}^{-1}$, where
$\chi_{cyc}$ is the cyclotomic character. \footnote{We always normalize the isomorphism of local class field theory by letting a prime element correspond to a geometric Frobenius. Then, by \cite{Ser} III.A4 Prop.\ 4, the character $\chi_{LT}$ coincides with the character
\begin{equation*}
    \Gal(\overline{L}/L) \twoheadrightarrow \Gal(\overline{L}/L)^{ab} \cong \widehat{L}^\times = o_L^\times \times \pi_L^{\widehat{\ZZ}} \xrightarrow{\pr} o_L^\times \ .
\end{equation*}
}

The ring $\mathcal{O}_K(\mathfrak{B})$ is the ring of all formal power series in $Z$ with coefficients in $K$ which converge on $\mathfrak{B}(\Cp)$. In terms of formal power series the induced $o_L \setminus \{0\}$-action on $\mathcal{O}_K(\mathfrak{B})$ is given by $(a,F) \longmapsto F \circ [a]$. We define
\begin{equation*}
    \mathscr{R}_K(\mathfrak{B}) := \bigcup_r \, \mathcal{O}_K(\mathfrak{B} \setminus \mathfrak{B}(r))
\end{equation*}
which is the ring of all formal series
$F(Z) = \sum_{n\in \ZZ} a_n Z^n$, $a_n \in K$, which are convergent in $(\mathfrak{B} \setminus \mathfrak{B}(r))(\Cp)$ for some $r<1$ depending on $F$. It follows from \cite{ST} Lemma 3.2 that the $o_L \setminus \{0\}$-action on $\mathcal{O}_K(\mathfrak{B})$ extends to $\mathscr{R}_K(\mathfrak{B})$.

By the maximum modulus principle (cf.\ \cite{Schi} Thm.\ 42.3(i)) $\mathcal{O}_K^b(\mathfrak{B})$ is the ring of all formal power series $F(Z) = \sum_{n \geq 0} a_n Z^n$, $a_n \in K$, such that $\sup_{n \geq 0} |a_n| < \infty$, and the supremum norm satisfies
\begin{equation*}
    \|F\|_{\mathfrak{B}} = \sup_{z \in \mathfrak{B}(\Cp)} |F(z)| = \sup_{n \geq 0} |a_n| \ .
\end{equation*}
The norm $\|\ \|_{\mathfrak{B}}$ is known to be multiplicative  (cf.\ \cite{vRo} Lemma 6.40). We define
\begin{equation*}
    \mathscr{E}_K^\dagger(\mathfrak{B}) := \bigcup_r \, \mathcal{O}_K^b(\mathfrak{B} \setminus \mathfrak{B}(r))
\end{equation*}
and
\begin{equation*}
    \|F\|_1 := \lim_{r \rightarrow 1} \|F\|_{\mathfrak{B} \setminus \mathfrak{B}(r)}
\end{equation*}
for $F \in \mathscr{E}_K^\dagger(\mathfrak{B})$. Using the maximum modulus principle for affinoid annuli one shows that an element in $\mathscr{R}_K(\mathfrak{B})$ written as a formal Laurent series $F(Z) = \sum_{n \in \ZZ} a_nZ^n$ lies in $\mathscr{E}_K^\dagger(\mathfrak{B})$ if and only if $\sup_{n \in \ZZ} |a_n| < \infty$ and that, in this case, $\|F\|_1 = \sup_{n \in \ZZ} |a_n| < \infty$. Since $\|F\|_1 = \lim_{r \rightarrow 1} \|F\|_r$ is the limit of the multiplicative norms $\|F\|_r := \sup_{n \in \ZZ} |a_n|r^n$ we see that $\|\ \|_1$ is a multiplicative norm. We now let
\begin{equation*}
    \mathscr{E}_K(\mathfrak{B}) := \text{completion of $\mathscr{E}^\dagger_K(\mathfrak{B})$ with respect to $\|\ \|_1$}
\end{equation*}
and
\begin{equation*}
    \mathscr{E}_K^{\leq 1}(\mathfrak{B}) := \{ F \in \mathscr{E}_K(\mathfrak{B}) : \|F\|_1 \leq 1\} \ .
\end{equation*}
Again by \cite{ST} Lemma 3.2 the $o_L \setminus \{0\}$-action on $\mathcal{O}_K^b(\mathfrak{B})$ extends to a $\|\ \|_1$-isometric action on $\mathscr{E}_K(\mathfrak{B})$ which respects the subrings $\mathscr{E}_K^\dagger(\mathfrak{B})$ and $\mathscr{E}_K^{\leq 1}(\mathfrak{B})$.

\begin{lemma}\label{E}
$\mathscr{E}_K(\mathfrak{B})$ is the ring of formal series $F(Z) = \sum_{n\in \ZZ} a_n Z^n$, $a_n \in K$, such that $\sup_{n \in \ZZ} |a_n| < \infty$ and $\lim_{n \rightarrow -\infty} a_n = 0$, and $\|F\|_1 = \sup_{n \in \ZZ} |a_n|$ is a multiplicative norm. \footnote{Suppose that $K$ is discretely valued. Then  $\mathscr{E}_K^\dagger(\mathfrak{B})$ and $\mathscr{E}_K(\mathfrak{B})$ are fields. (Cf.\ \cite{TdA} \S\S9-10 for detailed proofs.) This is no longer the case in general.}
\end{lemma}
\begin{proof}
This is well known. See a formal argument, for example, in the proof of \cite{TdA} Lemma 10.4.
\end{proof}

By Cartier duality, $T'$ is the group of homomorphisms of formal
groups over $o_{\Cp}$ from $LT$ to the formal multiplicative group. This gives rise to a Galois equivariant and $o_L$-invariant pairing
\begin{equation*}
    \langle\ ,\ \rangle : T' \otimes_{o_L} \mathfrak{B}(\Cp) \longrightarrow \mathfrak{B}_1(\Cp) \ .
\end{equation*}
We fix a generator $t'_0$ of the $o_L$-module $T'$. Thm.\ 3.6 in \cite{ST} constructs an isomorphism
\begin{equation*}
    \kappa : \mathfrak{B}_{/\Cp} \xrightarrow{\;\cong\;} \mathfrak{X}_{/\Cp}
\end{equation*}
of rigid group varieties over $\Cp$ which on $\Cp$-points is given by
\begin{align*}
    \mathfrak{B}(\Cp) & \xrightarrow{\;\cong\;} \mathfrak{X}(\Cp) = \widehat{G}(\Cp) \\
    z & \longmapsto \kappa_z(g) := \langle t'_0, [g](z) \rangle \ .
\end{align*}
In fact, we need a more precise statement. We define
\begin{equation*}
  R_n := p^{\QQ} \cap [p^{-q/e(q-1)}, p^{-1/e(q-1)})^{1/q^{en}} \qquad\text{for $n \geq 0$}.
\end{equation*}
Note that these sets are pairwise disjoint; any sequence $(r_n)_{n \geq 0}$ with $r_n \in R_n$ converges to $1$. We also put $\omega := p^{1/e(q-1) - 1/(p-1)}$,
\begin{equation*}
  S_0 := R_0 \omega = p^{\QQ} \cap [p^{-1/e - 1/(p-1)}, p^{-1/(p-1)}) \subseteq p^{\QQ} \cap [p^{- p/(p-1)}, p^{-1/(p-1)}) \ ,
\end{equation*}
and $S_n := S_0^{1/p^n}$ for $n \geq 0$. Again these latter sets are pairwise disjoint such that any sequence $(s_n)_{n \geq 0}$ with $s_n \in S_n$ converges to $1$. The map
\begin{align*}
  S_n & \xrightarrow{\;\simeq\;} R_n \\
  s & \longmapsto s^{1/p^{(d-1)n}} \omega^{-1/p^{dn}} \ ,
\end{align*}
for any $n \geq 0$, is an order preserving bijection.

\begin{proposition}\label{LT-iso}
For any $n \geq 0$ and any $s \in S_n$ the isomorphism $\kappa$ restricts to an isomorphism of affinoid group varieties
\begin{equation*}
    \kappa : \mathfrak{B}(s^{1/p^{(d-1)n}} \omega^{-1/p^{dn}})_{/\Cp} \xrightarrow{\;\cong\;} \mathfrak{X}(s)_{/\Cp} \ .
\end{equation*}
\end{proposition}
\begin{proof}
Although formally not stated there in this generality, the proof is completely contained in \cite{ST} Thm.\ 3.6 and App.\ Thm.\ part (c). We briefly recall the argument. We have:
\begin{itemize}
  \item[--] For any $r \in R_0$ the map $[p^n] : \mathfrak{B}(r^{1/q^{en}}) = [p^n]^{-1}(\mathfrak{B}(r)) \longrightarrow \mathfrak{B}(r)$ is a finite etale affinoid map (\cite{ST} Lemma 3.2).
  \item[--] For any $s \in S_0$ the map $p^n : \mathfrak{X}(s^{1/p^n}) = (p^n)^{-1}(\mathfrak{X}(s)) \longrightarrow \mathfrak{X}(s)$ is a finite etale affinoid map (\cite{ST} Lemma 3.3).
  \item[--] For any $r \in p^{\QQ} \cap (0,p^{-1/e(q-1)})$ the map $\kappa$ restricts to a rigid isomorphism
\begin{equation*}
  \mathfrak{B}(r)_{/\Cp} \xrightarrow{\;\cong\;} \mathfrak{X}(r\omega)_{/\Cp}
\end{equation*}
      (\cite{ST} Lemma 3.4 and proof of Lemma 3.5).
\end{itemize}
Exactly as in the proof of \cite{ST} Thm.\ 3.6 it follows from these three facts that the horizontal arrows in the commutative diagram
\begin{equation*}
   \xymatrix{
     \mathfrak{B}(1^{1/q^{en}})_{/\Cp} \ar[d]_{[p^n]} \ar[r]^-{\kappa} & \mathfrak{X}((r\omega)^{1/p^n})_{/\Cp} \ar[d]^{p^n} \\
     \mathfrak{B}(r)_{/\Cp} \ar[r]^-{\kappa} & \mathfrak{X}(r\omega)_{/\Cp},   }
\end{equation*}
are rigid isomorphisms for any $n \geq 0$ and $r \in R_0$.
\end{proof}

\begin{remark}
If $L/\Qp$ is unramified then $\bigcup_{n \geq 0} R_n = p^{\QQ} \cap [p^{-q/(q-1)},1)$ and $\bigcup_{n \geq 0} S_n = p^{\QQ} \cap [p^{-p/(p-1)},1)$. Hence in this case $\mathfrak{X}(s)_{/\Cp}$, for any radius $s \in p^{\QQ} \cap (0,1)$, is isomorphic via $\kappa$ to an affinoid disk.
\end{remark}

In the following we abbreviate $\mathfrak{B}_n := \mathfrak{B}(p^{-1/e(q-1)q^{en-1}})$ and $\mathfrak{X}_n := \mathfrak{X}(p^{-(1+e/(p-1))/ep^n})$ for any $n \geq 1$
\footnote{Everything which follows also works for $n=0$. We avoid this case only since the symbol $\mathfrak{X}_0$ already has a different meaning.}; the two radii are the left boundary points of the sets $R_n$ and $S_n$, respectively, and they correspond to each other under the above bijection. From now on we treat $\kappa$ as an identification and view both $\mathcal{O}_K(\mathfrak{X})$ and $\mathcal{O}_K(\mathfrak{B})$ as subalgebras of $\mathcal{O}_{\Cp}(\mathfrak{B})$. The standard action of the Galois group $G_K := \Gal(\overline{K}/K)$ on $\mathcal{O}_{\Cp}(\mathfrak{B})$ is given by $(\sigma,F) \longmapsto {^\sigma F} := \sigma \circ F \circ \sigma^{-1}$ (in terms of power series $F = \sum_{n \geq 0} a_n Z^n$ we have ${^\sigma F} = \sum_{n \geq 0} \sigma(a_n) Z^n$), and $\mathcal{O}_K(\mathfrak{B})$ is the corresponding ring of Galois fixed elements
\begin{equation*}
    \mathcal{O}_K(\mathfrak{B}) = \mathcal{O}_{\Cp}(\mathfrak{B})^{G_K} \ .
\end{equation*}
The latter is a special case of the following general principle.

\begin{remark}\label{Galois-fixed}
For any quasi-separated rigid analytic variety $\mathfrak{Y}$ over $K$ we have $\mathcal{O}_K(\mathfrak{Y}) = \mathcal{O}_{\Cp}(\mathfrak{Y})^{G_K}$.
\end{remark}
\begin{proof}
See \cite{ST} p.\ 463 observing that $\Cp^{G_K} = K$ by the Ax-Sen-Tate theorem (\cite{Ax} or \cite{Tat}).
\end{proof}

The twisted Galois action on $\mathcal{O}_{\Cp}(\mathfrak{B})$ is defined by $(\sigma,F) \longmapsto {^{\sigma *}F} := ({^\sigma F})([\tau(\sigma^{-1}](\cdot))$, it commutes with the $o_L \setminus \{0\}$-action, and according to \cite{ST} Cor.\ 3.8 we have
\begin{equation}\label{f:twisted-O}
     \mathcal{O}_K(\mathfrak{X}) = \mathcal{O}_{\Cp}(\mathfrak{B})^{G_K,*} \ .
\end{equation}
Obviously we also have $\mathcal{O}_{\Cp}^b(\mathfrak{X}) = \mathcal{O}_{\Cp}^b(\mathfrak{B})$ (isometrically). The twisted Galois action on $\mathcal{O}_{\Cp}^b(\mathfrak{B})$ is by isometries, and we have
\begin{equation}\label{f:twisted-Ob}
    \mathcal{O}_K^b(\mathfrak{X}) = \mathcal{O}_{\Cp}^b(\mathfrak{B})^{G_K,*} \ .
\end{equation}

\begin{corollary}\label{multiplic}
The norms $\|\ \|_\mathfrak{X}$ on $\mathcal{O}_K^b(\mathfrak{X})$ and $\|\ \|_{\mathfrak{X}_n}$ on $\mathcal{O}_K(\mathfrak{X}_n)$, for any $n \geq 1$, are multiplicative.
\end{corollary}
\begin{proof}
Because of \eqref{f:twisted-Ob} the multiplicativity of $\|\ \|_\mathfrak{B}$ implies the multiplicativity of $\|\ \|_\mathfrak{X}$. In view of Prop.\ \ref{LT-iso} the argument for the $\|\ \|_{\mathfrak{X}_n}$ is exactly analogous.
\end{proof}

\begin{corollary}\label{units}
We have $\mathcal{O}_K(\mathfrak{X})^\times = \mathcal{O}_K^b(\mathfrak{X})^\times$.
\end{corollary}

\begin{proof}
For any fixed $f \in \mathcal{O}_K(\mathfrak{X})$ the supremum norms $\|f\|_{\mathfrak{X}(r)}$ on the affinoids $\mathfrak{X}(r)$ form a function in $r$ which is monotonously  increasing. If $f$ is a unit this applies to $f^{-1}$ as well. In Cor.\ \ref{multiplic} we have seen that there is a sequence $r_1 < \ldots < r_n < \ldots$ converging to $1$ such that the norms $\|\ \|_{\mathfrak{X}(r_n)}$ are multiplicative. It follows that the sequence $(\|f\|_{\mathfrak{X}(r_n)})_n$ at the same time is monotonously increasing and decreasing. Hence it is constant which shows that $f$ and $f^{-1}$ are bounded.
\end{proof}

\begin{lemma}\label{twist-Xn}
For any $n \geq 1$ we have $\mathcal{O}_K(\mathfrak{X} \setminus \mathfrak{X}_n) = \mathcal{O}_{\Cp}(\mathfrak{B} \setminus \mathfrak{B}_n)^{G_K,*}$ and $\mathcal{O}_K^b(\mathfrak{X} \setminus \mathfrak{X}_n) = \mathcal{O}_{\Cp}^b(\mathfrak{B} \setminus \mathfrak{B}_n)^{G_K,*}$.
\end{lemma}
\begin{proof}
As a consequence of Prop.\ \ref{LT-iso} the isomorphism $\kappa$ restricts to isomorphisms
\begin{equation*}
    (\mathfrak{B} \setminus \mathfrak{B}_n)_{/\Cp} \xrightarrow{\;\cong\;} (\mathfrak{X} \setminus \mathfrak{X}_n)_{/\Cp} \ .
\end{equation*}
On the other hand every unit $a \in o_L^\times$ with $\mathfrak{B}_{n/L}$ also preserves the admissible open subset $(\mathfrak{B} \setminus \mathfrak{B}_n)_{/L}$. Hence the twisted Galois action on $\mathcal{O}_{\Cp}(\mathfrak{B} \setminus \mathfrak{B}_n)$ is well defined. The assertion now follows in the same way as \eqref{f:twisted-O}.
\end{proof}

\subsection{Properties of $\mathcal{O}_K(\mathfrak{X})$}\label{sec:globalring}

The rigid variety $\mathfrak{X}$ is smooth and one dimensional by \cite{ST} paragraph before Lemma 2.4. As a closed subvariety of an open polydisk $\mathfrak{X}_{/K}$ is quasi-Stein. By \cite{ST} Cor.\ 3.7 the ring $\mathcal{O}_K(\mathfrak{X})$ is an integral domain. Therefore $\mathfrak{X}_{/K}$ satisfies the assumptions of section \ref{sec:prufer}. Hence the ring $\mathcal{O}_K(\mathfrak{X})$ has all the properties which we have established in section \ref{sec:prufer} (and which were listed, without proof, at the end of section 3 in \cite{ST}). We need a few further properties specific to $\mathfrak{X}$.

\begin{lemma}\phantomsection\label{unramified}
\begin{itemize}
  \item[i.] The ring homomorphism $\varphi_L : \mathcal{O}_K(\mathfrak{X}) \longrightarrow \mathcal{O}_K(\mathfrak{X})$ makes $\mathcal{O}_K(\mathfrak{X})$ a free module over itself of rank equal to the cardinality of the residue field $o_L/\pi_L o_L$.
  \item[ii.] The ring homomorphism $\varphi_L : \mathcal{O}_K(\mathfrak{X}) \longrightarrow \mathcal{O}_K(\mathfrak{X})$ is unramified at every point of $\mathfrak{X}_{/K}$.\footnote{For a torsion point $x$ the subsequent Lemma \ref{zeros}.i, which says that $\mathfrak{m}_x = \log_\mathfrak{X} \mathcal{O}_x$, allows the following elementary argument. Using assertion ii. of the same lemma we have
\begin{equation*}
    \varphi_L(\mathfrak{m}_{\pi_L^*(x)}) \mathcal{O}_x = \varphi_L(\log_\mathfrak{X}) \mathcal{O}_x = \pi_L \log_\mathfrak{X} \mathcal{O}_x = \mathfrak{m}_x \ .
\end{equation*}}
  \item[iii.] For any $0 \neq f \in \mathcal{O}_K(\mathfrak{X})$ and any point $x \in \mathfrak{X}_{/K}$ we have
\begin{equation*}
    \divi(\varphi_L(f))(x) = \divi(f)(\pi_L^*(x)) \ .
\end{equation*}
\end{itemize}
\end{lemma}
\begin{proof}
i. This is most easily seen by using the Fourier isomorphism which reduces the claim to the corresponding statement about the distribution algebra $D(o_L,K)$. But here the ring homomorphism $\varphi_L$ visibly induces an isomorphism between $D(o_L,K)$ and the subalgebra $D(\pi_L o_L,K)$ of $D(o_L,K)$. Let $R \subseteq o_L$ denote a set of representatives for the cosets in $o_L/\pi_L o_L$. Then the Dirac distributions $\{\delta_g\}_{g \in R}$ form a basis of $D(o_L,K)$ as a $D(\pi_L o_L,K)$-module.

ii. Since a power of $\varphi_L$ times an appropriate automorphism is equal to $p_*$ it suffices to show that the latter homomorphism is everywhere unramified. But $\mathfrak{X}$ is a closed subvariety of $\mathfrak{X}_0 \cong \mathfrak{B}_1^d$. Hence it further suffices to observe that the endomorphism of $\mathcal{O}_K(\mathfrak{X}_0)$ induced by the $p$th power map $(z_1,\ldots,z_d) \longmapsto (z_1^p,\ldots,z_d^p)$ on $\mathfrak{B}_1^d$ is everywhere unramified, which is clear.

iii. This follows immediately from the second assertion.
\end{proof}

The Lie algebra $\mathfrak{g} = L$ of the locally $L$-analytic group $G = o_L$ embeds $L$-linearly into the distribution algebra $D(G,K)$ via
\begin{align*}
    \mathfrak{g} & \longrightarrow D(G,K) \\
    \mathfrak{x} & \longmapsto \delta_\mathfrak{x}(f) := (-\mathfrak{x}(f))(0) \ .
\end{align*}
Composing this with the Fourier isomorphism in \cite{ST} Thm.\ 2.3 we obtain the embedding
\begin{align*}
    \mathfrak{g} & \longrightarrow \mathcal{O}_K(\mathfrak{X}) \\
    \mathfrak{x} & \longmapsto [x \mapsto \delta_\mathfrak{x}(\chi_x) = d\chi_x(\mathfrak{x})]
\end{align*}
where we denote by $\chi_x$ the locally $L$-analytic character of $G$ corresponding to the point $x$ (cf. \cite{ST} \S2).

\begin{definition}
$\log_\mathfrak{X} \in \mathcal{O}_L(\mathfrak{X})$ denotes the holomorphic function which is the image of $1 \in \mathfrak{g} = L$ under the above embedding.
\end{definition}

Using \eqref{f:variety} we compute $d\chi_{z \otimes \beta}(1) = \log(z) \cdot \beta(1) = \log(z^{\beta(1)}) = \log(\chi_{z \otimes \beta}(1))$, and we obtain the formula
\begin{equation*}
    \log_\mathfrak{X}(x) = \log(\chi_x(1)) \ .
\end{equation*}
We see that the set of zeros of $\log_\mathfrak{X}$ coincides with the torsion subgroup of $\mathfrak{X}$.

\begin{lemma}\phantomsection\label{zeros}
\begin{itemize}
  \item[i.] All zeros of $\log_\mathfrak{X}$ are simple.
  \item[ii.] For any $a \in o_L$ we have $a_*(\log_\mathfrak{X}) = a \cdot \log_\mathfrak{X}$.
\end{itemize}
\end{lemma}
\begin{proof}
i. By the commutative diagram after Lemma 3.4 in \cite{ST} the function $\log_\mathfrak{X}$ corresponds under the LT-isomorphism $\kappa$ to the function $\Omega_{t'_0} \log_{LT}$. The simplicity of the zeros of the Lubin-Tate logarithm $\log_{LT}$ is well known.

ii. The locally analytic endomorphism $g \longmapsto ag$ of $G = o_L$ induces on the Lie algebra $\mathfrak{g}$ the multiplication by the scalar $a$. On the other hand, the map $\mathfrak{g} \longrightarrow D(G,L)$ is functorial in $G$. Hence we have $\delta_{a\mathfrak{x}} = a_*(\delta_\mathfrak{x})$.
\end{proof}

\section{The boundary of $\mathfrak{X}$ and $(\varphi_L,\Gamma_L)$-modules}

\subsection{The boundary of $\mathfrak{X}$}

Recall that the complement of an affinoid domain in an affinoid space is an admissible open subset (compare \cite{Sch} \S3 Prop.\ 3(ii)). Thus $\mathfrak{X}(r) \setminus \mathfrak{X}(r_0)$, for any pair of
$r > r_0$ in $(0,1) \cap p^\QQ$, is an admissible open subset of $\mathfrak{X}(r)$. As $\{\mathfrak{X}(r)\}_{r}$ is an admissible covering of $\mathfrak{X}$, a subset $S$ of $\mathfrak{X}$ is admissible open if and
only if $S\cap \mathfrak{X}(r)$ is admissible open in $\mathfrak{X}(r)$ for any $r$. Hence $\mathfrak{X} \setminus \mathfrak{X}(r)$ is an admissible open subset of $\mathfrak{X}$ and the rings $\mathcal{O}_K^b(\mathfrak{X} \setminus \mathfrak{X}(r)) \subseteq \mathcal{O}_K(\mathfrak{X} \setminus \mathfrak{X}(r))$ are defined.

The ring
\begin{equation*}
    \mathscr{R}_K(\mathfrak{X}) := \bigcup_r \, \mathcal{O}_K(\mathfrak{X} \setminus \mathfrak{X}(r))
\end{equation*}
is called the \textit{Robba ring} for $L$ (and $K$). In the case of $L=\Qp$ this definition coincides with the usual one. Observe that every affinoid subdomain of $\mathfrak{X}$ is contained in some
$\mathfrak{X}(r)$, so $\mathscr{R}_K(\mathfrak{X})$ is isomorphic to
$\varinjlim_\mathfrak{Y} \mathcal{O}_K(\mathfrak{X} \setminus
\mathfrak{Y})$, where $\mathfrak{Y}$ runs through all affinoid subdomains of $\mathfrak{X}$.

Next we define
\begin{equation*}
    \mathscr{E}^\dagger_K(\mathfrak{X}) :=
\bigcup_r \, \mathcal{O}_K^b(\mathfrak{X} \setminus \mathfrak{X}(r)) \ .
\end{equation*}
We obviously have $\|\ \|_{\mathfrak{X} \setminus \mathfrak{X}(r')} \leq \|\ \|_{\mathfrak{X} \setminus \mathfrak{X}(r)}$, for any $r' \geq r$ in $(0,1) \cap p^\QQ$, and therefore may define
\begin{equation*}
    \|f\|_1 := \lim_{r \rightarrow 1} \|f\|_{\mathfrak{X} \setminus \mathfrak{X}(r)}
\end{equation*}
for any $f \in \mathscr{E}^\dagger_K(\mathfrak{X})$. Later on, before Prop.\ \ref{twisted-R-E}, we will see that $\|\ \|_1$ is a multiplicative norm on $\mathscr{E}^\dagger_K(\mathfrak{X})$. We finally put
\begin{equation*}
    \mathscr{E}_K(\mathfrak{X}) := \text{completion of $\mathscr{E}^\dagger_K(\mathfrak{X})$ with respect to $\|\ \|_1$}
\end{equation*}
as well as
\begin{equation*}
    \mathscr{E}_K^{\leq 1}(\mathfrak{X}) := \{ f \in \mathscr{E}_K(\mathfrak{X}) : \|f\|_1 \leq 1\} \ .
\end{equation*}

There are a natural topology on $\mathscr{R}_K(\mathfrak{X})$ as well as the so called weak topology on $\mathscr{E}_K(\mathfrak{X})$, which is weaker than the norm topology.

In order to discuss $\mathscr{R}_K(\mathfrak{X})$ we first need to collect a few facts about the rigid topology of the varieties $\mathfrak{X}_0$ and $\mathfrak{X}$. In the following all radii like $r$, $r'$, $r_0$, and $r_1$ will be understood to lie in $(0,1) \cap p^\QQ$. Besides $\mathfrak{B}_1(r)$ we also need the open $\Qp$-disk $\mathfrak{B}_1^-(r)$ of radius $r$ around $1$. We introduce the subsets
\begin{align*}
  \mathfrak{X}_0(r) & := \mathfrak{B}_1(r) \otimes_{\Zp} \Hom_{\Zp}(o_L,\Zp), \\
  \mathfrak{X}_0^-(r) & := \mathfrak{B}_1^-(r) \otimes_{\Zp} \Hom_{\Zp}(o_L,\Zp), \\
  \mathfrak{X}_0(r,r') & := \mathfrak{X}_0(r') \setminus \mathfrak{X}_0^-(r) \quad\text{for $r \leq r'$}
\end{align*}
of $\mathfrak{X}_0$. The first two obviously are admissible open. We also noted already that $\mathfrak{X}_0 \setminus \mathfrak{X}_0(r)$ is admissible open in $\mathfrak{X}_0$. In order to understand $\mathfrak{X}_0(r,r')$ we list the following facts.
\begin{itemize}
  \item[--] If $z_1, \ldots, z_d$ are coordinate functions on $\mathfrak{X}_0$ then $\mathfrak{X}_0(r,r')$ is the union of the $d$ affinoid subdomains
\begin{equation*}
  \mathfrak{X}_0^{(i)}(r,r') := \{ x \in \mathfrak{X}_0(r') : |z_i(x)| \geq r\}
\end{equation*}
      of $\mathfrak{X}_0(r')$. (\cite{BGR} Cor.\ 9.1.4/4)
  \item[--] In particular, $\mathfrak{X}_0(r,r')$ is admissible open in $\mathfrak{X}_0(r')$ and hence in $\mathfrak{X}_0$.
  \item[--] If $\mathfrak{Y} \longrightarrow \mathfrak{X}_0 \setminus \mathfrak{X}_0(r_0)$ is any morphism from a $\Qp$-affinoid variety into $\mathfrak{X}_0 \setminus \mathfrak{X}_0(r_0)$ then its image is contained in $\mathfrak{X}_0(r,r')$ for some $r_0 < r \leq r'$. We apply the maximum modulus principle in the following way. Let $\alpha$ denote the morphism in question. First by applying the maximum modulus principle to the functions $\alpha^*(z_i)$ we find an $r' > r_0$ such that $\alpha(\mathfrak{Y})$ is contained in $\mathfrak{X}_0(r')$. Next we observe that $\mathfrak{X}_0(r') \setminus \mathfrak{X}_0(r_0) = \bigcup_{i=1}^d \mathfrak{U}_i$ with $\mathfrak{U}_i := \{ x \in \mathfrak{X}_0(r') : |z_i(x)| > r_0\}$ is an admissible covering (\cite{BGR} Prop.\ 9.1.4/5). Then $\mathfrak{Y} = \bigcup_{i=1}^d \alpha^{-1}(\mathfrak{U}_i)$ is an admissible covering and therefore, necessarily, can be refined into a finite affinoid covering $\mathfrak{Y} = \mathfrak{V}_1 \cup \ldots \cup \mathfrak{V}_m$. For any $1 \leq j \leq m$ let $1 \leq i(j) \leq d$ be such that $\alpha(\mathfrak{V}_j) \subseteq \mathfrak{U}_{i(j)}$. By applying the maximum modulus principle to the function $z_{i(j)}^{-1}$ pulled back to $\mathfrak{V}_j$ we obtain that
\begin{equation*}
  \alpha(\mathfrak{V}_j) \subseteq \mathfrak{X}_0^{(i(j))}(r_j,r') \qquad\text{for some $r_0 < r_j < r'$}.
\end{equation*}
      We deduce that $\alpha(\mathfrak{Y}) \subseteq \mathfrak{X}_0(r,r')$ with $r := \min_j r_j$.
  \item[--] In particular, $\mathfrak{X}_0 \setminus \mathfrak{X}_0(r_0) = \bigcup_{r_0 < r \leq r' < 1} \mathfrak{X}_0(r,r')$ is an admissible covering.
\end{itemize}

With $\mathfrak{X}(r)$ being defined already we also put
\begin{equation*}
  \mathfrak{X}^-(r) := \mathfrak{X} \cap \mathfrak{X}_0^-(r)_{/L} \qquad\text{and}\qquad \mathfrak{X}(r,r') := \mathfrak{X} \cap \mathfrak{X}_0(r,r')_{/L} = \mathfrak{X}(r') \setminus \mathfrak{X}^-(r) \ ,
\end{equation*}
and we have a corresponding list of properties:
\begin{itemize}
  \item[--] $\mathfrak{X}(r,r')$ is a finite union of affinoid subdomains and is admissible open in the $L$-affinoid variety $\mathfrak{X}(r')$.
  \item[--] If $\mathfrak{Y} \longrightarrow \mathfrak{X} \setminus \mathfrak{X}(r_0)$ is any morphism from an $L$-affinoid variety $\mathfrak{Y}$ then its image is contained in $\mathfrak{X}(r,r')$ for some $r_0 < r \leq r'$. In particular,
\end{itemize}
\begin{equation}\label{f:projlim}
  \mathcal{O}_K(\mathfrak{X} \setminus \mathfrak{X}(r_0)) = \varprojlim_{r_0 < r \leq r' < 1} \mathcal{O}_K(\mathfrak{X}(r,r')) \ .
\end{equation}

For simplicity we now use the fact that $\mathfrak{X}$ and hence each $\mathfrak{X}(r')$ is a connected, smooth, and one dimensional rigid variety (cf.\ \cite{ST}). Therefore (\cite{Fie} Satz 2.1) any finite union of affinoid subdomains in $\mathfrak{X}(r')$ again is an affinoid subdomain. It follows that each $\mathcal{O}_K(\mathfrak{X}(r,r'))$ is a $K$-affinoid algebra which is a Banach algebra with respect to the supremum norm. This together with \eqref{f:projlim} permits us to equip $\mathcal{O}_K(\mathfrak{X} \setminus \mathfrak{X}(r_0))$ with the structure of a $K$-Fr\'echet algebra by viewing it as the topological projective limit of these affinoid algebras. By construction the restriction maps $\mathcal{O}_K(\mathfrak{X} \setminus \mathfrak{X}(r_0)) \longrightarrow \mathcal{O}_K(\mathfrak{X} \setminus \mathfrak{X}(r_1))$, for $r_0 \leq r_1$, are continuous.

\begin{proposition}\label{quasi-Stein}
For any $n \geq 1$, the rigid variety $\mathfrak{X} \setminus \mathfrak{X}_n$ is a quasi-Stein space (w.r.t. the admissible covering $\{\mathfrak{X}(s,s')\}$ where $p^{-(1+e/(p-1))/ep^n} < s \leq s' < 1$, $s \in S_n$, and $s' \in \bigcup_{m \geq n} S_m$).
\end{proposition}
\begin{proof}
(The sets $S_n$ were defined in section \ref{sec:LT}.) We have to show that the restriction map $\mathcal{O}_K(\mathfrak{X}(s,s')) \longrightarrow \mathcal{O}_K(\mathfrak{X}(r,r'))$, for any $p^{-(1+e/(p-1))/ep^n} < s \leq r \leq r' \leq s' < 1$ with $r,s \in S_n$ and $r',s' \in \bigcup_{m \geq n} S_m$, has dense image. First of all note that affinoid algebras are Banach spaces of countable type.

In a first step we check that we may assume that $K = \Cp$. Quite generally, let $\beta : B_1 \longrightarrow B_2$ be a continuous linear map between $K$-Banach spaces such that $B_2$ is of countable type. By \cite{NFA} Prop.\ 10.5 there is a closed vector subspace $C \subseteq B_2$ such that $B_2 = \overline{\im(\beta)} \oplus C$ topologically. It follows that
\begin{equation*}
  \overline{\im(\id \widehat{\otimes} \beta)} = \Cp\, \widehat{\otimes}_K\, \overline{\im(\beta)} \subseteq \Cp\, \widehat{\otimes}_K\, B_2 = (\Cp\, \widehat{\otimes}_K\, \overline{\im(\beta)}) \oplus (\Cp\, \widehat{\otimes}_K\, C) \ .
\end{equation*}
Moreover, \cite{NFA} Cor.\ 17.5.iii implies that $\Cp\, \widehat{\otimes}_K\, C$ is nonzero if $C$ was. We see that $\beta$ has dense image if $\id \widehat{\otimes}\, \beta$ has.

So for the rest of the proof we let $K = \Cp$. We first observe that, quite generally, $\mathfrak{X}(r_0) = \bigcup_{r_1 < r_0} \mathfrak{X}(r_1)$. Due to the conditions we have imposed on the radii $r, r', s, s'$ we may apply Prop.\ \ref{LT-iso} and we see that $\mathfrak{X}(s,s')_{/\Cp}$ is isomorphic to a one dimensional affinoid annulus in such a way that $\mathfrak{X}(r,r')_{/\Cp}$ becomes isomorphic to a subannulus. It follows that $\mathfrak{X}(r,r')_{/\Cp}$ is a Weierstra{\ss} domain in $\mathfrak{X}(s,s')_{/\Cp}$, and the density statement holds by \cite{BGR} Prop.\ 7.3.4/2.
\end{proof}

The rigid variety $\mathfrak{X} \setminus \mathfrak{X}_n$ is smooth and one dimensional since it is admissible open in $\mathfrak{X}$. It is quasi-Stein by the above Prop.\ \ref{quasi-Stein}. The isomorphism $\mathcal{O}_{\Cp} (\mathfrak{X} \setminus \mathfrak{X}_n) \cong \mathcal{O}_{\Cp} (\mathfrak{B} \setminus \mathfrak{B}_n)$, which is a consequence of Prop.\ \ref{LT-iso}, implies that $\mathcal{O}_K (\mathfrak{X} \setminus \mathfrak{X}_n)$ is an integral domain. Therefore $(\mathfrak{X}\setminus \mathfrak{X}_n)_{/K}$ satisfies the assumptions of section \ref{sec:prufer}. Hence the ring $\mathcal{O}_K(\mathfrak{X} \setminus \mathfrak{X}_n)$ has all the properties which we have established in section \ref{sec:prufer}.

\begin{corollary}\phantomsection\label{pruefer2}
\begin{itemize}
  \item[i.] $\mathcal{O}_K(\mathfrak{X} \setminus \mathfrak{X}_n)$, for any $n \geq 1$, is a $1 \frac{1}{2}$ generator Pr\"ufer domain.
  \item[ii.] $\mathscr{R}_K(\mathfrak{X})$ is a $1 \frac{1}{2}$ generator Pr\"ufer domain. In particular, the assertions of Cor.\ \ref{free-invertible} analogously hold over $\mathscr{R}_K(\mathfrak{X})$.
\end{itemize}
\end{corollary}
\begin{proof}
i. Prop.\ \ref{1 1/2}. ii. This follows by a direct limit argument from i.
\end{proof}

We will view $\mathscr{R}_K(\mathfrak{X})$ as the locally convex inductive limit of the Fr\'echet algebras $\mathcal{O}_K(\mathfrak{X} \setminus \mathfrak{X}_n)$. We note that the multiplication in $\mathscr{R}_K(\mathfrak{X})$ is only separately continuous.

In order to analyze the functional analytic nature of the $\mathcal{O}_K(\mathfrak{X} \setminus \mathfrak{X}_n)$ and of $\mathscr{R}_K(\mathfrak{X})$ we need a few preliminary facts.

\begin{lemma}\label{descent-compactoid}
Let $C \subseteq B$ be a bounded subset of a $K$-Banach space $B$; if the image $\iota(C)$ of $C$ under the canonical map $\iota : B \longrightarrow \Cp\, \widehat{\otimes}_K\, B$ is compactoid, then $C$ is compactoid.
\end{lemma}
\begin{proof}
We fix a defining norm on $B$. Note (\cite{NFA} Prop.\ 17.4) that the map $\iota$ is norm preserving. As a consequence of \cite{PGS} Thm.\ 3.9.6 it suffices to show that, for any $t \in (0,1]$ and any sequence $(c_n)_{n \in \ZZ_{\geq 0}}$ in $C$ which is $t$-orthogonal (\cite{PGS} Def.\ 2.2.14), there exists a $t' \in (0,1]$ such that the sequence $(\iota(c_n))_n$ is $t'$-orthogonal in $\Cp \widehat{\otimes}_K B$.

Let $B_0 \subseteq B$ be the closed subspace generated by $(c_n)_n$, and let $c_0(K)$ denote the standard $K$-Banach space of all $0$-sequences in $K$. By the boundedness of $(c_n)_n$ the linear map
\begin{align*}
  \alpha : \qquad c_0(K) & \longrightarrow B_0 \\
  (\lambda_1, \lambda_2, \ldots) & \longmapsto \sum_{i=n}^\infty \lambda_n c_n
\end{align*}
is well defined and continuous. The $t$-orthogonality then implies that $\alpha$, in fact, is a homeomorphism (cf.\ the proof of \cite{PGS} Thm.\ 2.3.6). Hence
\begin{equation*}
  \id \widehat{\otimes}\, \alpha : \Cp\, \widehat{\otimes}_K\, c_0(K)  \xrightarrow{\;\cong\;} \Cp\, \widehat{\otimes}_K\, B_0
\end{equation*}
is a homeomorphism. By \cite{BGR} Prop.\ 2.1.7/8 the left hand side is equal to $c_0(\Cp)$. Hence $(\iota(c_n))_n$ is a basis of $\Cp\, \widehat{\otimes}_K\, B_0$ in the sense of \cite{PGS} Def.\ 2.3.10 and hence, by \cite{PGS} Thm.\ 2.3.11, is $t'$-orthogonal in $\Cp\, \widehat{\otimes}_K\, B_0$ for some $t' \in (0,1]$. Using \cite{NFA} Prop.\ 17.4.iii we finally obtain that $(\iota(c_n))_n$ is $t'$-orthogonal in $\Cp\, \widehat{\otimes}_K\, B$.
\end{proof}

For two locally convex $K$-vector spaces $V$ and $W$ we always view the tensor product $V \otimes_K W$ as a locally convex $K$-vector space with respect to the projective tensor product topology (cf.\ \cite{NFA} \S17), and we let $V \widehat{\otimes}_K W$ denote the completion of $V \otimes_K W$.

\begin{lemma}\label{tensor-projlim}
Let $\ldots \rightarrow V_n \rightarrow \ldots \rightarrow V_1$ and $\ldots \rightarrow W_n \rightarrow \ldots \rightarrow W_1$ be two sequences of locally convex $K$-vector spaces, $V := \varprojlim_n V_n$ and $W := \varprojlim_n W_n$ their projective limits, and $\alpha_n : V \rightarrow V_n$ and $\beta_n : W \rightarrow W_n$ the corresponding canonical maps.  We then have:
\begin{itemize}
  \item[i.] The projective tensor product topology on $V \otimes_K W$ is the initial topology with respect to the maps $\alpha_n \otimes \beta_n : V \otimes_K W \rightarrow V_n \otimes_K W_n$;
  \item[ii.] the canonical map $V \otimes_K W \longrightarrow \varprojlim_n V_n \otimes_K W_n$ is a topological embedding;
  \item[iii.] suppose that, for any $n \geq 1$, the spaces $V_n$ and $W_n$ are Hausdorff and the maps $\alpha_n$ and $\beta_n$ have dense image; then $V \widehat{\otimes}_K W = \varprojlim_n V_n \widehat{\otimes}_K W_n$.
\end{itemize}
\end{lemma}
\begin{proof}
i. Because of \cite{NFA} Cor.\ 17.5.ii we may assume, by replacing $V_n$, resp.\ $W_n$, by $\alpha_n(V)$, resp.\ $\beta_n(W)$, equipped with the subspace topology, that $V_n = \alpha_n(V)$ and $W_n = \beta_n(W)$ for any $n \geq 1$.

Let $\{L_{n,i}\}_{i \in I}$ and $\{M_{n,j}\}_{j \in J}$ be the families of all open lattices in $V_n$ and $W_n$, respectively (cf.\ \cite{NFA} \S4). Then $\{\alpha_n^{-1}(L_{n,i})\}_{n,i}$ and $\{\beta_n^{-1}(M_{n,j})\}_{n,j}$ are defining families of open lattices in $V$ and $W$, respectively, and $\{\alpha_n^{-1}(L_{n,i}) \otimes_o \beta_n^{-1}(M_{n,j})\}_{n,i,j}$ is a defining family of open lattices in $V \otimes_K W$. We therefore have to show that $\alpha_n^{-1}(L_{n,i}) \otimes_o \beta_n^{-1}(M_{n,j})$ is open for the initial topology of the assertion. Since, by construction, $(\alpha_n \otimes \beta_n)^{-1}(L_{n,i} \otimes_o M_{n,j})$ is open for this initial topology it suffices to prove that
\begin{equation*}
  (\alpha_n \otimes \beta_n)^{-1}(L_{n,i} \otimes_o M_{n,j}) \subseteq \alpha_n^{-1}(L_{n,i}) \otimes_o \beta_n^{-1}(M_{n,j})
\end{equation*}
(the opposite inclusion being trivial we then, in fact, have equality). Let $x \in V \times_K W$ such that
\begin{equation*}
  (\alpha_n \otimes \beta_n)(x) = \sum_{\rho = 1}^r v_\rho \otimes w_\rho \qquad\text{with $v_\rho \in L_{n,i}$ and $w_\rho \in M_{n,j}$}.
\end{equation*}
By our additional assumption we find $v'_\rho \in V$ and $w_\rho' \in W$ such that $\alpha_n(v'_\rho) = v_\rho$ and $\beta_n(w_\rho') = w_\rho$. Hence $x' := \sum_\rho v_\rho' \otimes w_\rho' \in \alpha_n^{-1}(L_{n,i}) \otimes_o \beta_n^{-1}(M_{n,j})$ and
\begin{equation*}
  x - x' \in \ker(\alpha_n \otimes \beta_n) = \ker(\alpha_n) \otimes_K V + W \otimes_K \ker(\beta_n) \ .
\end{equation*}
But the right hand side (and therefore $x$) is contained in $\alpha_n^{-1}(L_{n,i}) \otimes_o \beta_n^{-1}(M_{n,j})$. This follows from the general fact that for any $K$-subspace $V_0 \subseteq V$ and any lattice $L \subseteq W$ we have
\begin{equation*}
  V_0 \otimes_K W = V_0 \otimes_o W = V_0 \otimes_o L \ .
\end{equation*}

ii. It remains to check that the map in question is injective. Let $x \in V \otimes_K W$ be a nonzero element. We fix a $K$-basis $(e_j)_{j \in J}$ of $W$ and write $x = \sum_{j \in J_0} v_j \otimes e_j$ with an appropriate finite subset $J_0 \subseteq J$ and nonzero elements $v_j \in V$ for $j \in J_0$. We may choose $n$ large enough so that $\alpha_n(v_j) \neq 0$ for any $j \in J_0$. Then $y := \alpha_n \otimes \id_W (x) \neq 0$. We now repeat the argument with $y$ and a $K$-basis of $V_n$ in order to find an $m \geq n$ such that $\id_{V_n} \otimes \beta_m (y) \neq 0$. It follows that $\alpha_n \otimes \beta_m (x) \neq 0$ and hence that $\alpha_m \otimes \beta_m (x) \neq 0$.

iii. With $V_n$ and $W_n$ also $V$ and $W$ as well as $V_n \otimes_K W_n$ (cf.\ \cite{NFA} Cor.\ 17.5.i) and $\varprojlim_n V_n \otimes_K W_n$ are Hausdorff. In particular, $V_n \otimes_K W_n$ is a subspace of $V_n \widehat{\otimes}_K W_n$ and, consequently, $\varprojlim_n V_n \otimes_K W_n$ is a subspace of $\varprojlim_n V_n \widehat{\otimes}_K W_n$. Hence, by ii., the composite map
\begin{equation*}
  V \otimes_K W \longrightarrow \varprojlim_n V_n \otimes_K W_n \longrightarrow \varprojlim_n V_n \widehat{\otimes}_K W_n
\end{equation*}
is a topological embedding. Since the right hand space is complete it suffices to check that $V \otimes_K W$ maps to a dense subspace. But our assumption on $\alpha_n$ and $\beta_n$ implies that each map
\begin{equation*}
  V \otimes_K W \xrightarrow{\alpha_n \otimes \beta_n} V_n \otimes_K W_n \xrightarrow{\;\subseteq\;} V_n \widehat{\otimes}_K W_n
\end{equation*}
has dense image.
\end{proof}

As a special case of this lemma (together with the Mittag-Leffler theorem \cite{B-TG} II.3.5 Thm.\ 1) we obtain that the functor $\Cp\, \widehat{\otimes}_K .$ commutes with projective limits of sequences of $K$-Fr\'echet spaces such that the transition maps have dense image. Note (\cite{PGS} \S10.6) that the scalar extension $\Cp\, \widehat{\otimes}_K\, V$ of any locally convex $K$-vector space $V$ is a locally convex $\Cp$-vector space.

\begin{proposition}\label{compactoid}
For any $n \geq 1$ we have:
\begin{itemize}
  \item[i.] The Fr\'echet space $\mathcal{O}_K(\mathfrak{X} \setminus \mathfrak{X}_n)$ is nuclear; in particular, it is Montel, hence reflexive, and in $\mathcal{O}_K(\mathfrak{X} \setminus \mathfrak{X}_n)$ the class of compactoid subsets coincides with the class of bounded subsets;
  \item[ii.] $\mathcal{O}_{\Cp}(\mathfrak{X} \setminus \mathfrak{X}_n) =  \Cp\, \widehat{\otimes}_K\, \mathcal{O}_K(\mathfrak{X} \setminus \mathfrak{X}_n)$, and the inclusion $\mathcal{O}_K(\mathfrak{X} \setminus \mathfrak{X}_n) \subseteq \mathcal{O}_{\Cp}(\mathfrak{X} \setminus \mathfrak{X}_n)$ is topological.
\end{itemize}
\end{proposition}
\begin{proof}
The Fr\'echet space $\mathcal{O}_K(\mathfrak{X} \setminus \mathfrak{X}_n)$ is the projective limit of the affinoid $K$-Banach spaces $\mathcal{O}_K(\mathfrak{X}(s,s'))$ with the radii $s \leq s'$ as in Prop.\ \ref{quasi-Stein}. In the proof of Prop.\ \ref{quasi-Stein} we have seen that the transition maps $\mathcal{O}_K(\mathfrak{X}(s,s')) \longrightarrow \mathcal{O}_K(\mathfrak{X}(r,r'))$, for $s \leq r \leq r' \leq s'$, in this projective system have dense image. For a $K$-affinoid variety $\mathfrak{Y}$ we have $\mathcal{O}_{\Cp}(\mathfrak{Y}) = \Cp\, \widehat{\otimes}_K\, \mathcal{O}_K(\mathfrak{Y})$ by construction of the base field extension (cf.\ \cite{BGR} 9.3.6). Hence we may apply Lemma \ref{tensor-projlim} and obtain the equality in part ii. of the assertion. The second half of ii. then follows by direction inspection or by using \cite{NFA} Cor.\ 17.5.iii.

We now assume that the radii satisfy $s < r \leq r' < s'$ and show that then the restriction map $\mathcal{O}_K(\mathfrak{X}(s,s')) \longrightarrow \mathcal{O}_K(\mathfrak{X}(r,r'))$ is compactoid. Recall that this means (cf.\ \cite{PGS} Thm.\ 8.1.3(vii)) that the unit ball in $\mathcal{O}_K(\mathfrak{X}(s,s'))$ for the supremum norm is mapped to a compactoid subset in $\mathcal{O}_K(\mathfrak{X}(r,r'))$. Because of Lemma \ref{descent-compactoid} we may assume that $K = \Cp$. But then, according to Prop.\ \ref{LT-iso}, $\mathfrak{X}(s,s')$ is isomorphic to a one dimensional affinoid annulus. This reduces us to the following claim. Let $a < b \leq b' < a'$ be any radii in $p^{\QQ}$ and let $\mathfrak{B}(a,a') := \{z \in \Cp : a \leq |z| \leq a'\}$; then the restriction map $\mathcal{O}_K(\mathfrak{B}(a,a')) \longrightarrow \mathcal{O}_K(\mathfrak{B}(b,b'))$ is compactoid. Let $a = |u|$ for some $u \in \Cp$. The Mittag-Leffler decomposition (which is directly visible in terms of Laurent series) gives the topological decomposition
\begin{align*}
  \mathcal{O}_K(\mathfrak{B}(a^{-1})) \oplus \mathcal{O}_K(\mathfrak{B}(a')) & \xrightarrow{\; \cong \;} \mathcal{O}_K(\mathfrak{B}(a,a')) \\
  (F_1,F_2) & \longmapsto uZ^{-1}F_1(Z^{-1}) + F_2(Z) \ .
\end{align*}
It further reduces us to showing that the restriction map $\mathcal{O}_K(\mathfrak{B}(a')) \longrightarrow \mathcal{O}_K(\mathfrak{B}(b'))$ is compactoid. But this is a well known fact (cf.\ \cite{PGS} Thm.\ 11.4.2).

We have established that $\mathcal{O}_K(\mathfrak{X} \setminus \mathfrak{X}_n) = \varprojlim_n B_n$ is the projective limit of a sequence of $K$-Banach spaces $B_n$ with compactoid transition maps. In order to show that $\mathcal{O}_K(\mathfrak{X} \setminus \mathfrak{X}_n)$ is nuclear we have to check that any continuous linear map $\alpha : \mathcal{O}_K(\mathfrak{X} \setminus \mathfrak{X}_n) \longrightarrow V$ into a normed $K$-vector space $B$ is compactoid (cf.\ \cite{PGS} Def.\ 8.4.1(ii)). But there is an $n \geq 1$ such that $\alpha$ factorizes through $B_n$:
\begin{equation*}
   \xymatrix{
 \mathcal{O}_K(\mathfrak{X} \setminus \mathfrak{X}_n)  \ar[rr]^-{\alpha} \ar[d] \ar[dr]
                &  &    V     \\
          B_{n+1} \ar[r]      & B_n \ar[ur]                }
\end{equation*}
Since the transition map $B_{n+1} \rightarrow B_n$ is compactoid it follows that $\alpha$ is compactoid as well. The remaining assertions in i. now follow by \cite{PGS} Cor.\ 8.5.3 and Thm.\ 8.4.5.
\end{proof}

Next we investigate the inductive system of Fr\'echet spaces
\begin{equation*}
  \mathcal{O}_K(\mathfrak{X} \setminus \mathfrak{X}_1) \longrightarrow \ldots \rightarrow \mathcal{O}_K(\mathfrak{X} \setminus \mathfrak{X}_n) \rightarrow \mathcal{O}_K(\mathfrak{X} \setminus \mathfrak{X}_{n+1}) \rightarrow \ldots
\end{equation*}
with locally convex inductive limit $\mathscr{R}_K(\mathfrak{X})$. Note that all transition maps in this system are injective.

\begin{proposition}\phantomsection\label{regular}
\begin{itemize}
  \item[i.] $\mathscr{R}_K(\mathfrak{X}) = \varinjlim_n \mathcal{O}_K(\mathfrak{X} \setminus \mathfrak{X}_n)$ is a regular inductive limit, i.e., for each bounded subset $B \subseteq \mathscr{R}_K(\mathfrak{X})$ there exists an $n \geq 1$ such that $B$ is contained in $\mathcal{O}_K(\mathfrak{X} \setminus \mathfrak{X}_n)$ and is bounded as a subset of the Fr\'echet space $\mathcal{O}_K(\mathfrak{X} \setminus \mathfrak{X}_n)$.
  \item[ii.] $\mathscr{R}_{\Cp}(\mathfrak{X})$ is complete.
  \item[iii.] $\mathscr{R}_K(\mathfrak{X})$ is Hausdorff, nuclear, and reflexive.
\end{itemize}
\end{proposition}

\begin{proof}
i. and ii. We first reduce the assertion to the case $K = \Cp$. For this we consider the commutative diagram
\begin{equation*}
  \xymatrix{
    \mathcal{O}_K(\mathfrak{X} \setminus \mathfrak{X}_1) \ar@{^{(}->}[d] \ar[r] & \ldots \ar[r] & \mathcal{O}_K(\mathfrak{X} \setminus \mathfrak{X}_n) \ar@{^{(}->}[d] \ar[r] & \ldots \ar[r] & \mathscr{R}_K(\mathfrak{X}) \ar[d] \\
    \mathcal{O}_{\Cp}(\mathfrak{X} \setminus \mathfrak{X}_1) \ar[r] & \ldots \ar[r] & \mathcal{O}_{\Cp}(\mathfrak{X} \setminus \mathfrak{X}_n) \ar[r] &  \ldots \ar[r] & \mathscr{R}_{\Cp}(\mathfrak{X})   }
\end{equation*}
in which all maps are injective and continuous and the two left vertical maps are topological embeddings. Let us suppose that the lower inductive limit is regular, and let $B \subseteq \mathscr{R}_K(\mathfrak{X})$ be a bounded subset. Then $B$ is bounded in $\mathscr{R}_{\Cp}(\mathfrak{X})$ and hence, by assumption, is bounded in some $\mathcal{O}_{\Cp}(\mathfrak{X} \setminus \mathfrak{X}_n)$. Using Remark \ref{Galois-fixed}  we see that $B \subseteq \mathscr{R}_K(\mathfrak{X}) = \mathscr{R}_{\Cp}(\mathfrak{X})^{G_K}$ and hence $B \subseteq \mathcal{O}_{\Cp}(\mathfrak{X} \setminus \mathfrak{X}_n)^{G_K} = \mathcal{O}_K(\mathfrak{X} \setminus \mathfrak{X}_n)$.

In the following we therefore assume that $K = \Cp$. By Lemma \ref{twist-Xn} we now may replace $\mathfrak{X}$ and $\mathfrak{X}_n$ by $\mathfrak{B}$ and $\mathfrak{B}_n$, respectively. This time the Mittag-Leffler decomposition is of the form
\begin{equation*}
  \mathcal{O}_{\Cp}(\mathfrak{B}^-(r_n)) \oplus \mathcal{O}_{\Cp}(\mathfrak{B}) \xrightarrow{\cong} \mathcal{O}_{\Cp}(\mathfrak{B} \setminus \mathfrak{B}_n)
\end{equation*}
with $\mathfrak{B}^-(r_n)$ denoting the open disk of radius $r_n := p^{1/e(q-1)q^{en-1}}$ around $0$. Hence
\begin{equation*}
  \mathscr{R}_{\Cp}(\mathfrak{X}) \cong \mathscr{R}_{\Cp}(\mathfrak{B}) = \big( \varinjlim_n \mathcal{O}_{\Cp}(\mathfrak{B}^-(r_n)) \big) \oplus \mathcal{O}_{\Cp}(\mathfrak{B}) \ .
\end{equation*}
This reduces us to showing that $\varinjlim_n \mathcal{O}_{\Cp}(\mathfrak{B}^-(r_n))$ is a regular inductive limit which, moreover, is complete. But the restriction maps $\mathcal{O}_{\Cp}(\mathfrak{B}^-(r_n)) \longrightarrow \mathcal{O}_{\Cp}(\mathfrak{B}^-(r_{n+1}))$ go through two affinoid disks for which the corresponding restriction maps are compactoid as recalled in the proof of Prop.\ \ref{compactoid}. Hence $\varinjlim_n \mathcal{O}_{\Cp}(\mathfrak{B}^-(r_n))$ is a compactoid inductive limit and, in particular, is regular and complete by \cite{PGS} Thm.\ 11.3.5.

iii. This follows from i. and Prop.\ \ref{compactoid}.i  and \cite{PGS} Thm.\ 11.2.4(ii), Thm.\ 8.5.7(vi), and Cor.\ 11.2.15.
\end{proof}

The reflexivity of $\mathscr{R}_K(\mathfrak{X})$ already implies that $\mathscr{R}_K(\mathfrak{X})$ is quasicomplete. To see that it, in fact, is always  complete, we need an additional argument. Since $\mathscr{R}_K(\mathfrak{X})$ is an inductive limit it is not surprising that the inductive tensor product topology will play a role. For two locally convex $K$-vector spaces $V$ and $W$ let $V \otimes_{K,\iota} W$ denote their tensor product equipped with the inductive tensor product topology (cf.\ \cite{NFA} \S17) and let $V \widehat{\otimes}_{K,\iota} W$ denote its completion. Note that for Fr\'echet spaces $V$ and $W$ we have $V \otimes_K W = V \otimes_{K,\iota} W$ (cf. \cite{NFA} Prop.\ 17.6).

\begin{proposition}\phantomsection\label{complete}
\begin{itemize}
  \item[i.]  $\mathscr{R}_{\Cp}(\mathfrak{X}) =  \Cp\, \widehat{\otimes}_{K,\iota}\, \mathscr{R}_K(\mathfrak{X})$.
  \item[ii.] $\mathscr{R}_K(\mathfrak{X}) \subseteq \mathscr{R}_{\Cp}(\mathfrak{X})$ is a topological inclusion.
  \item[iii.] $\mathscr{R}_K(\mathfrak{X}) = \mathscr{R}_{\Cp}(\mathfrak{X})^{G_K}$ is closed in $\mathscr{R}_{\Cp}(\mathfrak{X})$.
  \item[iv.]  $\mathscr{R}_K(\mathfrak{X})$ is complete.
\end{itemize}
\end{proposition}
\begin{proof}
i. So far we know from Prop.\ \ref{compactoid}.ii and Prop.\ \ref{regular} that
\begin{equation*}
  \mathscr{R}_{\Cp}(\mathfrak{X}) = \varinjlim_n\, \mathcal{O}_{\Cp}(\mathfrak{X} \setminus \mathfrak{X}_n)
   = \varinjlim_n\,  \Cp\, \widehat{\otimes}_K\, \mathcal{O}_K(\mathfrak{X} \setminus \mathfrak{X}_n)
   = \varinjlim_n\,  \Cp\, \widehat{\otimes}_{K,\iota}\, \mathcal{O}_K(\mathfrak{X} \setminus \mathfrak{X}_n)
\end{equation*}
is Hausdorff and complete. The inductive tensor product commutes with locally convex inductive limits (\cite{Eme} Lemma 1.1.30) so that we have
\begin{equation*}
  \varinjlim_n\,  \Cp \otimes_{K,\iota} \mathcal{O}_K(\mathfrak{X} \setminus \mathfrak{X}_n) = \Cp \otimes_{K,\iota} \mathscr{R}_K(\mathfrak{X})
\end{equation*}
and hence
\begin{equation*}
  \big( \varinjlim_n\,  \Cp \otimes_{K,\iota} \mathcal{O}_K(\mathfrak{X} \setminus \mathfrak{X}_n) \big)^{\widehat{}} = \Cp \widehat{\otimes}_{K,\iota} \mathscr{R}_K(\mathfrak{X}) \ .
\end{equation*}
It therefore remains to check that, for any inductive system $(E_n)_n$ of locally convex vector spaces such that the locally convex inductive limit $\varinjlim_n \widehat{E_n}$ of the completions $\widehat{E_n}$ is Hausdorff and complete, we have
\begin{equation*}
  \varinjlim_n \widehat{E_n} = \widehat{\varinjlim_n E_n} \ .
\end{equation*}
By the universal properties we have a commutative diagram of natural continuous maps
\begin{equation*}
  \xymatrix{
  \varinjlim_n \widehat{E_n} \ar[rr]^{\gamma}
                &  &   \widehat{\varinjlim_n E_n}    \\
                & \varinjlim_n E_n ,  \ar[ul]^{\alpha}  \ar[ur]            }
\end{equation*}
which all three have dense image. By our assumption on the upper left term the map $\alpha$ extends uniquely to a continuous map $\beta : \widehat{\varinjlim_n E_n} \longrightarrow \varinjlim_n \widehat{E_n}$. Of course, $\beta$ has dense image as well. On the other hand it necessarily satisfies $\gamma \circ \beta = \id$. Since $\varinjlim_n \widehat{E_n}$ is assumed to be Hausdorff this implies that $\beta$ has closed image. It follows that $\beta$ is bijective and therefore that $\beta$ and $\gamma$ are topological isomorphisms.

ii. This is a consequence of i. and \cite{NFA} Cor.\ 17.5.iii.

iii. The identity follows from Remark \ref{Galois-fixed}. The second part of the assertion then is immediate from the fact that $G_K$ acts on $\mathscr{R}_{\Cp}(\mathfrak{X})$ by continuous (semilinear) automorphisms.

iv. This follows from ii./iii. and Prop.\ \ref{regular}.iii.
\end{proof}

\begin{corollary}\label{scalar-ext-R}
$\mathscr{R}_E(\mathfrak{X}) =  E\, \widehat{\otimes}_{K,\iota}\, \mathscr{R}_K(\mathfrak{X})$ for any complete intermediate field $K \subseteq E \subseteq \Cp$.
\end{corollary}
\begin{proof}
Now that we know from Prop.\ \ref{complete}.iv that $\mathscr{R}_E(\mathfrak{X})$ is complete we may repeat the argument in the proof of Prop.\ \ref{complete}.i with $E$ instead of $\Cp$.
\end{proof}

We summarize the additional features of the LT-isomorphism, which we have established by now:

--- As explained in the proof of Lemma \ref{twist-Xn} the isomorphism $\kappa$ induces compatible topological identifications
\begin{equation*}
  \mathcal{O}_{\Cp}(\mathfrak{X} \setminus \mathfrak{X}_n) = \mathcal{O}_{\Cp}(\mathfrak{B} \setminus \mathfrak{B}_n) \ ,
\end{equation*}
which restrict to isometric identifications between the subrings of bounded functions. The twisted Galois action is well defined on the right hand side and corresponds to the standard Galois action on the left hand side.

--- By passing to the inductive limit we obtain the twisted Galois action by topological automorphisms on $\mathscr{R}_{\Cp}(\mathfrak{B})$ as well as the topological identification
\begin{equation*}
  \mathscr{R}_{\Cp}(\mathfrak{X}) = \mathscr{R}_{\Cp}(\mathfrak{B}) \ .
\end{equation*}

--- By restriction from $\mathscr{R}_{\Cp}(.)$ to $\mathscr{E}_{\Cp}^\dagger(.)$ we obtain the twisted Galois action by $\|\ \|_1$-isometries on $\mathscr{E}_{\Cp}^\dagger(\mathfrak{B})$ as well as the $\|\ \|_1$-isometric identification
\begin{equation*}
  \mathscr{E}_{\Cp}^\dagger(\mathfrak{X}) = \mathscr{E}_{\Cp}^\dagger(\mathfrak{B}) \ .
\end{equation*}
Both, by completion, extend to $\mathscr{E}_{\Cp}(.)$. Since we know that $\|\ \|_1$ is a multiplicative norm on the right hand side it must be a multiplicative norm on $\mathscr{E}_K(\mathfrak{X})$ as well.

\begin{proposition}\phantomsection\label{twisted-R-E}
$\mathscr{R}_{\Cp}(\mathfrak{X}) =  \Cp\, \widehat{\otimes}_{K,\iota}\, \mathscr{R}_K(\mathfrak{B})$ and $\mathscr{R}_K(\mathfrak{X}) = \mathscr{R}_{\Cp}(\mathfrak{B})^{G_K,*}$.
\end{proposition}
\begin{proof}
This is a restatement of Prop.\ \ref{complete}.i/iii.
\end{proof}

\subsection{The monoid action}\label{sec:monoid}

For any $a \in o_L$ the map $g \longmapsto ag$ on $G$ is locally $L$-analytic. This induces an action of the multiplicative monoid $o_L \setminus \{0\}$ on the vector spaces of locally analytic functions $C^{an}(G,K) \subseteq C^{an}(G_0,K)$ given by $f \longmapsto a^*(f)(g) := f(ag)$. Obviously, with $\chi \in \widehat{G}(K)$, resp.\ $\in \widehat{G}_0(K)$, also $a^*(\chi)$ is a character in $\widehat{G}(K)$, resp.\ in $\widehat{G}_0(K)$. In this way we obtain actions of the ring $o_L$ on these groups. It is clear that under the bijection \eqref{f:variety} the action on the target, which we just have defined, corresponds to the obvious $o_L$-action on the second factor of the tensor product in the source. This shows that the action on character groups in fact comes from an $o_L$-action on the rigid character varieties $\mathfrak{X}_0$ and $\mathfrak{X}$ which, moreover, respects each of the affinoid subgroups $\mathfrak{X}(r)$. Moreover, from these actions on character varieties we obtain translation actions by the multiplicative monoid $o_L \setminus \{0\}$ on the corresponding rings of global holomorphic functions, which will be denoted by $(a,f) \longmapsto a_*(f)$. Note also that these actions respect the respective subrings of bounded holomorphic functions.

Each unit in $o_L^\times$ with $\mathfrak{X}(r)$ also preserves the admissible open subset $\mathfrak{X} \setminus \mathfrak{X}(r)$. The $o_L^\times$-action on $\mathcal{O}_K^b(\mathfrak{X}) \subseteq \mathcal{O}_K(\mathfrak{X})$ therefore extends to the rings $\mathscr{R}_K(\mathfrak{X})$, $\mathscr{E}_K^\dagger(\mathfrak{X})$, $\mathscr{E}_K(\mathfrak{X})$, and $\mathscr{E}_K^{\leq 1}(\mathfrak{X})$ (being isometric in the norm $\|\ \|_1$).

\begin{lemma}\label{Gamma-cont-O}
The $o_L^\times$-action on $\mathcal{O}_K(\mathfrak{X} \setminus \mathfrak{X}(r_0))$, for any $r_0 \in [p^{- 1/e -1/(p-1)},1) \cap p^\QQ$, is continuous.
\end{lemma}
\begin{proof}
Coming from an action on the variety $\mathfrak{X} \setminus \mathfrak{X}(r_0)$ each unit in $o_L^\times$ acts by a continuous ring automorphism on $\mathcal{O}_K(\mathfrak{X} \setminus \mathfrak{X}(r_0))$. Since $o_L^\times$ is compact and $\mathcal{O}_K(\mathfrak{X} \setminus \mathfrak{X}(r_0))$ as a Fr\'echet space is barrelled it remains, by the nonarchimedean Banach-Steinhaus theorem (cf.\ \cite{NFA} Prop.\ 6.15), to show that the orbit maps
\begin{align*}
  \rho_f : o_L^\times & \longrightarrow \mathcal{O}_K(\mathfrak{X} \setminus \mathfrak{X}(r_0)) \\
  a & \longmapsto a_*(f) \ ,
\end{align*}
for any $f \in \mathcal{O}_K(\mathfrak{X} \setminus \mathfrak{X}(r_0))$, are continuous. But in the subsequent section \ref{sec:locan} we will establish, under the above assumption on $r_0$, the stronger fact that these orbit maps even are differentiable.
\end{proof}

\begin{lemma}\label{pi}
For any $r \in [p^{-p/(p-1)},1) \cap p^\QQ$ we have $(\pi_L^*)^{-1}(\mathfrak{X} \setminus \mathfrak{X}(r)) \supseteq \mathfrak{X} \setminus \mathfrak{X}(r^{1/p})$.
\end{lemma}
\begin{proof}
Suppose there is an $x \in \mathfrak{X} \setminus \mathfrak{X}(r^{1/p})$ such that $\pi_L^*x \in \mathfrak{X}(r)$. Since $\pi_L$ as well as any unit in $o_L^\times$ preserve $\mathfrak{X}(r)$ it follows that $p^*x \in \mathfrak{X}(r)$. This contradicts Lemma 3.3 in \cite{ST}.
\end{proof}

The lemma implies that the action of $\pi_L$ and hence of the full multiplicative monoid $o_L \setminus \{0\}$ extends to the rings $\mathscr{R}_K(\mathfrak{X})$, $\mathscr{E}_K^\dagger(\mathfrak{X})$, $\mathscr{E}_K(\mathfrak{X})$, and $\mathscr{E}_K^{\leq 1}(\mathfrak{X})$. By construction the action of $\pi_L$ on $\mathscr{E}_K(\mathfrak{X})$ is decreasing in the norm $\|\ \|_1$. But the action of $p$, as a consequence of \cite{ST} Lemma 3.3, is isometric. It follows that the full monoid $o_L \setminus \{0\}$ acts isometrically on $\mathscr{E}_K(\mathfrak{X})$.

\begin{lemma}\label{Gamma-cont-R}
The $o_L \setminus \{0\}$-action on $\mathscr{R}_K(\mathfrak{X})$ is continuous.
\end{lemma}
\begin{proof}
Each $a \in o_L^\times$ acts continuously on $\mathcal{O}_K(\mathfrak{X} \setminus \mathfrak{X}(r_0))$ and therefore gives rise, in the limit, to a continuous automorphism of $\mathscr{R}_K(\mathfrak{X})$. The prime element $\pi_L$, by Lemma \ref{pi}, maps $\mathfrak{X} \setminus \mathfrak{X}(r_0)$ to $\mathfrak{X} \setminus \mathfrak{X}(r_0^p)$, for any $p^{-\frac{1}{p-1}} \leq r_0 < 1$, and therefore again gives rise to a continuous endomorphism of $\mathscr{R}_K(\mathfrak{X})$. The continuity of the orbit maps
\begin{align*}
  \rho_f : o_L^\times & \longrightarrow \mathscr{R}_K(\mathfrak{X}) \\
  a & \longmapsto a_*(f) \ ,
\end{align*}
for any $f \in \mathscr{R}_K(\mathfrak{X})$, follows immediately from Lemma \ref{Gamma-cont-O}. As a locally convex inductive limit of Fr\'echet spaces $\mathscr{R}_K(\mathfrak{X})$ is barrelled so that the assertion now follows from the nonarchimedean Banach-Steinhaus theorem.
\end{proof}

Observing that $\kappa_{[a](z)} = a^*(\kappa_z)$ one checks that the isomorphism $\kappa : \mathfrak{B}_{/\Cp} \xrightarrow{\;\cong\;} \mathfrak{X}_{/\Cp}$ from section \ref{sec:LT} is equivariant for the $o_L$-action on both sides. It follows that the ring isomorphisms as summarized after Cor.\ \ref{scalar-ext-R} all are equivariant for the actions of the monoid $o_L \setminus \{0\}$.

Traditionally one thinks of an $o_L \setminus \{0\}$-action as an action of the multiplicative group $\Gamma_L := o_L^\times$ together with a commuting endomorphism $\varphi_L$ which represents the action of $\pi_L$. We note that $\varphi_L$ is injective on $\mathscr{E}_K(\mathfrak{X})$ and a fortiori on $\mathcal{O}_K(\mathfrak{X})$.

For later purposes we need to discuss, for any $n \geq 1$, the structure of the $\pi_L^n$-torsion subgroup
\begin{equation*}
    \mathfrak{X}[\pi_L^n] := \ker (\mathfrak{X} \xrightarrow{\; (\pi_L^n)^* \;} \mathfrak{X}) \ .
\end{equation*}

\begin{lemma}\label{torsion}
$\mathfrak{X}[\pi_L^n](\Cp) \cong o_L/\pi_L^n o_L$ as $o_L$-modules.
\end{lemma}
\begin{proof}
As a consequence of \cite{ST} Lemma 2.1 the bijection \eqref{f:variety} induces an $o_L$-isomorphism
\begin{equation*}
    \mathfrak{X}[\pi_L^n](\Cp) \cong \ker\big[ \mathfrak{B}_1(\Cp) \otimes_{\Zp} \Hom_{\Zp}(o_L, {\Zp}) \xrightarrow{\; \id \otimes \pi_L^n \;}
    \mathfrak{B}_1(\Cp) \otimes_{\Zp} \Hom_{\Zp}(o_L, {\Zp}) \big] \ .
\end{equation*}
The short exact sequence $0 \longrightarrow o_L \xrightarrow{\; \pi_L^n \;} o_L \longrightarrow o_L/\pi_L^n o_L \longrightarrow 0$ gives rise to the short exact sequence
\begin{equation*}
    0 \longrightarrow \Hom_{\Zp}(o_L, {\Zp}) \xrightarrow{\; \pi_L^n \;} \Hom_{\Zp}(o_L, {\Zp}) \longrightarrow \Ext^1_{\Zp}(o_L/\pi_L^n o_L, {\Zp}) \longrightarrow 0
\end{equation*}
which in turn leads to the exact sequence
\begin{multline*}
    0 \longrightarrow \Tor^{\Zp}_1(\mathfrak{B}_1(\Cp), \Ext^1_{\Zp}(o_L/\pi_L^n o_L, {\Zp})) \longrightarrow \\
    \mathfrak{B}_1(\Cp) \otimes_{\Zp} \Hom_{\Zp}(o_L, {\Zp}) \xrightarrow{\; \id \otimes \pi_L^n \;}
    \mathfrak{B}_1(\Cp) \otimes_{\Zp} \Hom_{\Zp}(o_L, {\Zp}) \ .
\end{multline*}
We deduce that $\mathfrak{X}[\pi_L^n](\Cp) \cong \Tor^{\Zp}_1(\mathfrak{B}_1(\Cp), \Ext^1_{\Zp}(o_L/\pi_L^n o_L, {\Zp}))$. Since $\Hom_{\Zp}(o_L, {\Zp})$ is a free $o_L$-module of rank one we have $\Ext^1_{\Zp}(o_L/\pi_L^n o_L, {\Zp})) \cong o_L/\pi_L^n o_L$ as $o_L$- and, a fortiori, as $\Zp$-module. On the other hand the torsion subgroup of $\mathfrak{B}_1(\Cp)$ is the group of $p$-power roots of unity which is isomorphic to $\Qp / \Zp$. It follows that
\begin{align*}
    \mathfrak{X}[\pi_L^n](\Cp) & \cong \Tor^{\Zp}_1(\mathfrak{B}_1(\Cp), o_L/\pi_L^n o_L) \\
    & \cong \Tor^{\Zp}_1(\Qp / \Zp, o_L/\pi_L^n o_L) \\
    & = \text{torsion subgroup of $o_L / \pi_L^n o_L$} \\
    & = o_L / \pi_L^n o_L \ .
\end{align*}
\end{proof}

\subsection{The action of $\Lie(\Gamma_L)$}\label{sec:locan}

We begin by setting up some axiomatic formalism. In the following we view $\Gamma_L = o_L^\times$ as a locally $L$-analytic group. Its Lie algebra is $L$. Let $(B,\|\ \|_B)$ be a $K$-Banach space which carries a $\Gamma_L$-action by continuous $K$-linear automorphisms. We consider the following condition:
\begin{align}\label{f:condition-locan}
  \text{There is an $m \geq 2$ such that, in the operator norm on $B$, we have} \\
   \text{$\|\gamma - 1\| < p^{-\frac{1}{p-1}}$ for any $\gamma \in 1 + p^m o_L$.} \qquad\qquad\qquad \nonumber
\end{align}
If $B$ has the orthogonal basis $(v_i)_{i \in I}$ then \eqref{f:condition-locan} follows from the existence of a constant $0 < C < p^{-\frac{1}{p-1}}$ such that $\|(\gamma - 1)(v_i)\|_B \leq C  \|v_i\|_B$ for any $\gamma \in 1 + p^m o_L$ and any $i \in I$.

\begin{lemma}\label{locan}
Suppose that \eqref{f:condition-locan} holds; then the $\Gamma_L$-action on $B$ is locally $\Qp$-analytic.
\end{lemma}
\begin{proof}
The main point is to show that the orbit maps
\begin{align*}
  \rho_b : \Gamma_L & \longrightarrow B \\
  \gamma & \longmapsto \gamma b,
\end{align*}
for any $b \in B$, are locally $\Qp$-analytic. Since the multiplication in $\Gamma_L$ is locally $L$-analytic it suffices to show that $\rho_b | 1 + p^m o_L$ with $m$ as in the assumption is locally $\Qp$-analytic. On the one hand we fix a basis $\gamma_1, \ldots, \gamma_r$ of the free $\Zp$-module $1 + p^m o_L$ of rank $r := [L:\Qp]$ and write $1 + p^m o_L \ni \gamma = \gamma_1^{x_1(\gamma)} \cdot \ldots \cdot \gamma_r^{x_r(\gamma)}$. Then $\gamma \longmapsto (x_1(\gamma), \ldots , x_r(\gamma))$ is a global $\Qp$-chart of $1 + p^m o_L$. On the other hand we consider the $K$-Banach algebra $A$ of continuous linear endomorphisms of $B$ with the operator norm. We have the bijections
\begin{align*}
  \{a' \in A : \|a'\| < p^{-\frac{1}{p-1}} \} & \xrightarrow{\;\exp\;} \{a \in A : \|a - 1\| < p^{-\frac{1}{p-1}} \} \\
   & \xleftarrow{\;\log\;} \\
   a' & \longmapsto \exp(a') = \sum_{n \geq 0} \frac{1}{n!} (a')^n \\
   - \sum_{n \geq 1} \frac{1}{n} (1-a)^n = \log (1+(a-1)) = \log(a) & \longleftarrow a
\end{align*}
which are inverse to each other and which ``preserve'' norms (cf.\ \cite{B-GAL} II\S8.4 Prop.\ 4). By our assumption we have in $A$ the expansion
\begin{align}\label{f:expansion}
  \gamma & = \exp(\log(\gamma)) = \sum_{n \geq 0} \frac{1}{n!} \log(\gamma)^n = \sum_{n \geq 0} \frac{1}{n!} \big( \sum_{i=1}^r x_i(\gamma) \log(\gamma_i)\big)^n \\
  & = \sum_{(n_1, \ldots, n_r) \in \ZZ_{\geq 0}^r} \frac{c_{(n_1, \ldots,n_r)}}{(n_1 + \ldots + n_r)!} \prod_{i=1} (\log(\gamma_i))^{n_i} \prod_{i=1}^r x_i(\gamma)^{n_i} \nonumber
\end{align}
on $1 + p^m o_L$, where the integers $c_{(n_1,\ldots,c_r)}$ denote the usual polynomial coefficients. Since $d := \max_i \|\log(\gamma_i)\| < p^{-\frac{1}{p-1}}$ by assumption the coefficient operators
\begin{equation*}
  \nabla_{B,(n_1,\ldots,n_r)} := \frac{c_{(n_1, \ldots,n_r)}}{(n_1 + \ldots + n_r)!} \prod_{i=1} (\log(\gamma_i))^{n_i}
\end{equation*}
satisfy
\begin{align*}
  \|\nabla_{B,(n_1,\ldots,n_r)}\| & \leq |(n_1 + \ldots n_r)!|^{-1} d^{n_1 + \ldots n_r} \\
   & = |(n_1 + \ldots + n_r)!|^{-1} |p|^{\frac{n_1 + \ldots + n_r}{p-1}} (d/|p|^{\frac{1}{p-1}})^{n_1 + \ldots + n_r} \\
   & \leq (d/|p|^{\frac{1}{p-1}})^{n_1 + \ldots n_r} \xrightarrow{\; n_1 + \ldots + n_r \rightarrow \infty\;} 0
\end{align*}
where the last inequality comes from $|p|^{\frac{n}{p-1}} \leq |n!|$ (cf.\ \cite{B-GAL} II\S8.1 Lemma 1). Evaluating \eqref{f:expansion} in $b \in B$ therefore produces an expansion of $\rho_b$ into a power series convergent on $1 + p^m o_L$.

It follows in particular that the orbit maps $\rho_b$ are continuous, i.e., that the $\Gamma_L$-action on $B$ is separately continuous. But, since $\Gamma_L$ is compact, the nonarchimedean Banach-Steinhaus theorem (cf.\ \cite{NFA} Prop.\ 6.15) then implies that the action $\Gamma_L \times B \rightarrow B$ is jointly continuous.
\end{proof}

The operators $\nabla_{B,(n_1, \ldots, n_r)}$ on $B$ which we have constructed in the proof of Lemma \ref{locan} have the following conceptual interpretation. First we observe that
\begin{equation*}
  \nabla_{B,(n_1, \ldots, n_r)} = \prod_{i=1}^r \nabla_{B,(\ldots,0,n_i,0, \ldots)} = \prod_{i=1}^r \frac{1}{n_i !} \nabla_{B,\underline{i}}^{n_i}
\end{equation*}
where $\underline{i} := (\ldots,0,1,0,\ldots)$ is the multi-index with entry one in the $i$th place. On the other hand the derived action of the Lie algebra of $\Gamma_L$ on $B$ is given by
\begin{equation*}
  \mathfrak{x} b = \frac{d}{dt} \exp_{\Gamma_L} (t \mathfrak{x}) b_{\big| t=0} \qquad\text{for $\mathfrak{x} \in \Lie(\Gamma_L) = L$ and $b \in B$}
\end{equation*}
where $\exp_{\Gamma_L}$ is an exponential map for $\Gamma_L$ (cf.\ \cite{Fea} \S3.1) and where $t$ varies in a small neighbourhood of zero in $\Zp$. Note that the usual exponential function $\exp : \Lie(\Gamma_L) = L - - -> o_L^\times = \Gamma_L$ is an exponential map for the locally $L$-analytic group $\Gamma_L$ (cf.\ \cite{B-GAL} III\S4.3 Ex.\ 2). We put $\mathfrak{x}_i := \log \gamma_i$. Using the expansion \eqref{f:expansion} we compute
\begin{align*}
  \exp_{\Gamma_L} (t \mathfrak{x}_j) b & = \sum_{(n_1,\ldots,n_r) \in \ZZ_{\geq 0}^r}  \nabla_{B,(n_1, \ldots, n_r)} (b) \prod_{i=1}^r x_i(\exp (t \mathfrak{x}_j))^{n_i} \\
   & = \sum_{(n_1,\ldots,n_r) \in \ZZ_{\geq 0}^r}  \nabla_{B,(n_1, \ldots, n_r)} (b) \prod_{i=1}^r x_i(\exp (\mathfrak{x}_j)^t)^{n_i} \\
   & = \sum_{(n_1,\ldots,n_r) \in \ZZ_{\geq 0}^r}  \nabla_{B,(n_1, \ldots, n_r)} (b) \prod_{i=1}^r x_i(\gamma_j^t)^{n_i} \\
   & = \sum_{n=0}^\infty \frac{1}{n!} \nabla_{B,\underline{j}}^n (b) t^n
\end{align*}
and hence
\begin{equation}\label{f:nabla}
  \frac{d}{dt} \exp_{\Gamma_L} (t \mathfrak{x}_j) b_{\big| t=0} = \nabla_{B,\underline{j}}(b) \ .
\end{equation}
This proves the following.

\begin{corollary}\label{locan2}
If \eqref{f:condition-locan} holds, then the operator $\nabla_{B,\underline{j}}$ coincides with the derived action of $\log \gamma_j \in L = \Lie(\Gamma_L)$ on $B$.
\end{corollary}

The commutativity of $\Gamma_L$ implies that the operators $\nabla_{B,(n_1,\ldots,n_r)}$ commute with each other. It follows that the derived $\Lie(\Gamma_L)$-action on $B$ is through commuting operators. In the following we denote by $\nabla_B$, or simply by $\nabla$, the operator corresponding to the element $1 \in L = \Lie(\Gamma_L)$.

We remark that, although $\Gamma_L$ is a locally $L$-analytic group, the action on $B$ only is locally $\Qp$-analytic so that the derived action $\Lie(\Gamma_L) \longrightarrow \End_K(B)$ only is $\Qp$-linear in general.

Next we consider the situation where $B$ in addition is a Banach algebra such that $\|\ \|_B$ is submultiplicative and $M$ is a finitely generated projective $B$-module. We choose generators $e_1, \ldots, e_d$ of $M$ and consider the $B$-module homomorphism
\begin{align*}
  B^d & \longrightarrow M \\
  (b_1, \ldots, b_d) & \longmapsto \sum_{i=1}^d b_i e_i \ .
\end{align*}
On $B^d$ we have the maximum norm $\|(b_1, \ldots, b_d)\|_{B^d} := \max_i \|b_i\|_B$ and on $M$ the corresponding quotient norm
\begin{equation*}
  \|x\|_M := \inf \{\max_i \|b_i\|_B : x = \sum_{i=1}^d b_i e_i\} \ .
\end{equation*}
To see that $\|\ \|_M$ indeed is a norm we observe that the above map, by the projectivity of $M$, is the projection map onto the first summand of a suitable $B$-module isomorphism $B^d \xrightarrow{\cong} M \oplus M'$. This isomorphism is continuous if we equip $M$ and $M'$ with the corresponding quotient topologies. By the open mapping theorem it has to be a topological isomorphism. Hence the quotient topology on $M$ is Hausdorff and complete (cf.\ \cite{NFA} Prop.\ 8.3). In particular, $\|\ \|_M$ is a norm. This construction makes $M$ into a $K$-Banach space whose topology is independent of the choice of basis. We assume that $\Gamma_L$ acts continuously by $K$-linear automorphisms on $M$ which are semilinear with respect to the $\Gamma_L$-action on $B$.

\begin{proposition}\label{locan3}
If \eqref{f:condition-locan} holds for $B$, then also the $\Gamma_L$-action on $M$ satisfies \eqref{f:condition-locan} and, in particular, is locally $\Qp$-analytic.
\end{proposition}
\begin{proof}
Let the natural number $m$ be as in the condition \eqref{f:condition-locan} for $B$ in Lemma \ref{locan}. By the continuity of the $\Gamma$-action on $M$ we find an $m' \geq m$ such that $\|(\gamma - 1)(e_i)\|_M < p^{-2}$ (and hence $\|\gamma (e_i)\|_M = \|e_i\|_M = 1$) for any $\gamma \in 1 + p^{m'} o_L$ and any $1 \leq i \leq d$. For any such $\gamma$ and any $x = \sum_{i=1}^d b_i e_i \in M$ we then compute
\begin{align*}
    \|(\gamma - 1)(x)\|_M & = \| \sum_i (\gamma - 1)(b_i e_i)\|_M \\
    & = \| \sum_i \big( (\gamma - 1)(b_i) \gamma(e_i) + b_i (\gamma - 1)(e_i) \big) \|_M \\
    & \leq \max ( \max_i \|(\gamma - 1)\| \|b_i\|_B \|\gamma (e_i)\|_M , \max_i \|b_i\|_B \|(\gamma - 1)(e_i)\|_M) \\
    & \leq \max (\|\gamma - 1\|, p^{-2}) \cdot \max_i \|b_i\|_B \ .
\end{align*}
It follows that
\begin{equation*}
    \|(\gamma - 1)(x)\|_M \leq \max (\|\gamma - 1\|, p^{-2}) \cdot \|x\|_M < p^{-\frac{1}{p-1}} \|x\|_M \ .
\end{equation*}
\end{proof}

In a first step we apply this formalism to the $o_L^\times$-action on the open disk $\mathfrak{B}$. For any $r_0 \leq r$ in $(0,1) \cap p^{\QQ}$ let $\mathfrak{B}(r_0,r)_{/L}$ denote the $L$-affinoid annulus of inner, resp.\ outer, radius $r_0$, resp.\ $r$, around zero. It is preserved by the $o_L^\times$-action.

\begin{proposition}\label{annuli-locan}
The $\Gamma_L$-actions on $\mathcal{O}_K (\mathfrak{B}(r))$ and on $\mathcal{O}_K (\mathfrak{B}(r_0,r))$, induced by the Lubin-Tate formal $o_L$-module, verify the condition \eqref{f:condition-locan} and are locally $L$-analytic.
\end{proposition}
\begin{proof}
The other case being simpler we only treat the $\Gamma_L$-action on $\mathcal{O}_K (\mathfrak{B}(r_0,r))$. First of all we verify the condition \eqref{f:condition-locan}. The elements of $\mathcal{O}_K (\mathfrak{B}(r_0,r))$ are the Laurent series in the coordinate $Z$ which converge on $\mathfrak{B}(r_0,r))(\Cp)$. The maximum modulus principle tells us that $\{Z^n\}_{n \in \ZZ}$ is an orthogonal basis of the $K$-Banach space $\mathcal{O}_K (\mathfrak{B}(r_0,r))$ with respect to the supremum norm $\|\ \|$. It therefore suffices to find an $m \geq 2$ such that
\begin{equation*}
  \|\gamma (Z^n) - Z^n \| \leq p^{-2} \|Z^n\| \qquad\text{for any $n \in \ZZ$ and any $\gamma \in 1 + p^m o_L$}.
\end{equation*}
We have $\gamma(Z) = [\gamma](Z) = \gamma Z + \ldots = Zu_\gamma$ with $u_\gamma = \gamma + \ldots \in o_L[[Z]]^\times$, and we compute
\begin{equation*}
  \gamma(Z^n) - Z^n = Z^n(u_\gamma^n - 1)
   =
   \begin{cases}
   Z^n(u_\gamma - 1)(u_\gamma^{n-1} + \ldots 1) & \text{if $n > 0$} \\
   Z^n(u_\gamma - 1)(-u_\gamma^n - \ldots - u_\gamma^{-1}) & \text{if $n < 0$}.
   \end{cases}
\end{equation*}
It follows that $\|\gamma (Z^n) - Z^n \| \leq \|u_\gamma - 1\| \|Z^n\|$. By the proof of Lemma 2.1.1 in \cite{KR} there exists an $m \geq 2$ such that $\|u_\gamma - 1\| = \| \frac{\gamma(Z)}{Z} - 1\| \leq p^{-2}$. This shows that \eqref{f:condition-locan} holds true. Lemma \ref{locan} then tells us that the $\Gamma_L$-action is locally $\Qp$-analytic.

Next we establish that the derived $\Lie(\Gamma_L)$-action is $L$-bilinear. Since $\Lie(\Gamma_L)$ acts by continuous derivations it suffices to check that $\mathfrak{x}Z = \mathfrak{x} \cdot 1Z$ holds true for any sufficiently small $\mathfrak{x} \in L = \Lie(\Gamma_L)$ (where $\cdot$ on the right hand side denotes the scalar multiplication). We compute
\begin{align*}
  \mathfrak{x} Z & = \frac{d}{dt} [\exp_{\Gamma_L} (t \mathfrak{x})] (Z)_{\big| t=0} = \frac{d}{dt} [\exp (t \mathfrak{x})](Z)_{\big| t=0} \\
  & = \frac{d}{dt} \exp_{LT}(\log_{LT}([\exp (t \mathfrak{x})](Z)))_{\big| t=0}
   = \frac{d}{dt} \exp_{LT}(\exp (t \mathfrak{x}) \cdot \log_{LT}(Z))_{\big| t=0} \\
   & = \mathfrak{x} \cdot \frac{d}{dt} \exp_{LT}(\exp (t) \cdot \log_{LT}(Z))_{\big| t=0} = \mathfrak{x} \cdot 1Z \ .
\end{align*}
The fourth identity uses the fact that the logarithm $\log_{LT}$ of $LT$ is ``$o_L$-linear'' (\cite{Lan} 8.6 Lemma 2).

Finally, by looking at the Taylor expansion
\begin{equation*}
  \exp(\mathfrak{x})b = \sum_{n=0}^\infty \frac{1}{n!} \mathfrak{x}^n b = \sum_{n=0}^\infty \frac{1}{n!} \mathfrak{x}^n \cdot 1^n b
\end{equation*}
we see that the orbit maps $\rho_b$ are locally $L$-analytic.
\end{proof}

\begin{lemma}\label{smalldisk-locan}
For any $r \in (0,p^{-\frac{1}{p-1}}) \cap p^\QQ$ the $\Gamma_L$-action on $\mathcal{O}_K (\mathfrak{X}(r))$ verifies the condition \eqref{f:condition-locan} and is locally $L$-analytic.
\end{lemma}
\begin{proof}
Lemma \ref{small-disk} reduces us to proving the assertion for the $\Gamma_L$-action on $\mathcal{O}_K (\mathfrak{B}(r))$ which is induced by the multiplication action of $o_L^\times$ on the disk $\mathfrak{B}(r)$. But this is seen by an almost trivial version of the reasoning in the proof of Prop.\ \ref{annuli-locan}.
\end{proof}

According to Prop.\ \ref{annuli-locan} we have derived $\Lie(\Gamma_L)$-actions on $\mathcal{O}_K (\mathfrak{B}(r))$ and $\mathcal{O}_K (\mathfrak{B}(r_0,r))$ which are $L$-bilinear.  For $r_0' \leq r_0 \leq r \leq r'$ the inclusions $\mathcal{O}_K (\mathfrak{B}(r')) \subseteq \mathcal{O}_K (\mathfrak{B}(r))$ and $\mathcal{O}_K (\mathfrak{B}(r_0',r')) \subseteq \mathcal{O}_K (\mathfrak{B}(r_0,r))$ respects these actions. By first a projective limit and then a direct limit argument we therefore obtain compatible $L$-bilinear $\Lie(\Gamma_L)$-actions on
\begin{equation*}
  \mathcal{O}_K(\mathfrak{B}) \subseteq \mathcal{O}_K(\mathfrak{B} \setminus \mathfrak{B}(r)) \subseteq \mathscr{R}_K(\mathfrak{B}) \ .
\end{equation*}
Next we use the LT-isomorphism in section \ref{sec:LT} to obtain $L$-bilinear $\Lie(\Gamma_L)$-actions on
\begin{equation*}
  \mathcal{O}_K(\mathfrak{X}_n) ,\; \mathcal{O}_K(\mathfrak{X}(s,s'))\ \text{(with $s,s'$ as in Prop.\ \ref{quasi-Stein}), and} \quad   \mathcal{O}_K(\mathfrak{X}) \subseteq \mathcal{O}_K(\mathfrak{X} \setminus \mathfrak{X}_n) \subseteq \mathscr{R}_K(\mathfrak{X}) \ .
\end{equation*}
Recall that $\mathscr{R}_K(\mathfrak{B})$ and $\mathscr{R}_K(\mathfrak{X})$ are the locally convex inductive limits of the Fr\'echet spaces $\mathcal{O}_K(\mathfrak{B} \setminus \mathfrak{B}(r))$ and $\mathcal{O}_K(\mathfrak{X} \setminus \mathfrak{X}_n)$, respectively. Hence all the above locally convex $K$-vector spaces are barrelled (\cite{NFA} Examples at the end of \S6). The orbit maps $\rho_b$ for the $\Gamma_L$-action on these spaces (with the exception of $\mathcal{O}_K(\mathfrak{X}_n)$ and $\mathcal{O}_K(\mathfrak{X}(s,s'))$) are no longer locally $L$-analytic but they still are differentiable (\cite{Fea} 3.1.2). Hence these actions still are derived in the sense that they are given by the usual formula $(\mathfrak{x},b) \longmapsto \mathfrak{x} b = \frac{d}{dt} \exp_{\Gamma_L} (t \mathfrak{x}) b_{\big| t=0}$. The $\Gamma_L$-action on each $\mathcal{O}_K(\mathfrak{X}_n)$ and each $\mathcal{O}_K(\mathfrak{X}(s,s'))$ satisfies the condition \eqref{f:condition-locan} and is locally $L$-analytic.

For convenience we introduce the following notion.

\begin{definition}\label{def:L-analytic}
A differentiable (continuous) $\Gamma_L$-action on a barrelled locally convex $K$-vector space $V$ is called $L$-analytic if the derived action $\Lie(\Gamma_L) \times V \longrightarrow V$ is $L$-bilinear.
\end{definition}

We observe that with $V$ the induced $\Gamma_L$-action on any $\Gamma_L$-invariant closed barrelled subspace of $V$ is $L$-analytic as well.

If $\mathscr{F}(X,Y)$ denotes the formal group law of $LT$ then $\frac{\partial \mathscr{F}}{\partial Y}(0,Z)$ is a unit in $o_L[[Z]]$ and we put $g_{LT}(Z) := \big( \frac{\partial \mathscr{F}}{\partial Y}(0,Z) \big)^{-1}$. Then $g_{LT}(Z) dZ$ is, up to scalars, the unique invariant differential form on $LT$ (\cite{Haz} \S5.8). As before we let $\log_{LT}(Z) = Z + \ldots$ denote the unique formal power series in $L[[Z]]$ whose formal derivative is $g_{LT}$. This $\log_{LT}$ is the logarithm of $LT$ and lies in $\mathcal{O}_L(\mathfrak{B})$ (\cite{Lan} 8.6). In particular, $g_{LT}dZ = d\log_{LT}$ and $\mathcal{O}_L(\mathfrak{B}) dZ = \mathcal{O}_L(\mathfrak{B}) d\log_{LT}$. The invariant derivation $\partial_\mathrm{inv}$ on $\mathcal{O}_K(\mathfrak{B})$ corresponding to the form $d\log_{LT}$ is determined by
\begin{equation*}
  dF = \partial_\mathrm{inv} (F) d\log_{LT} = \partial_\mathrm{inv} (F) g_{LT} dZ = \frac{\partial F}{\partial Z} dZ
\end{equation*}
and hence is given by
\begin{equation*}
  \partial_\mathrm{inv}(F) = g_{LT}^{-1} \frac{\partial F}{\partial Z} \ .
\end{equation*}
The identity
\begin{equation*}
  \nabla_{\mathcal{O}_K(\mathfrak{B})} = \log_{LT} \cdot \partial_\mathrm{inv}
\end{equation*}
is shown in \cite{KR} Lemma 2.1.4.

Since the rigid variety $\mathfrak{X}$ is smooth of dimension one its sheaf of holomorphic differential forms is locally free of rank one. The group structure of $\mathfrak{X}$ forces this sheaf to even be free (cf.\ \cite{DG} II \S4.3.4). Hence the $\mathcal{O}_L(\mathfrak{X})$-module $\Omega_L(\mathfrak{X})$ of global holomorphic differential forms on $\mathfrak{X}$ is free of rank one. We claim that the differential form $d\log_{\mathfrak{X}}$ is invariant (for the group structure on $\mathfrak{X}$) and is a basis of $\Omega_L(\mathfrak{X})$. By the commutative diagram after Lemma 3.4 in \cite{ST} the function $\log_{\mathfrak{X}}$ corresponds, under the LT-isomorphism $\kappa$, to a nonzero scalar multiple of $\log_{LT}$. This implies our claim over $\Cp$ and then, by a simple descent argument with respect to the twisted Galois action, also over $L$. The invariant derivation $\partial_\mathrm{inv}$ on $\mathcal{O}_K(\mathfrak{X})$ corresponding to the form $d\log_{\mathfrak{X}}$ is defined by
\begin{equation*}
  df = \partial_\mathrm{inv} (f) d\log_{\mathfrak{X}} \ .
\end{equation*}
Using the LT-isomorphism it follows that
\begin{equation*}
  \nabla_{\mathcal{O}_K(\mathfrak{X})} = \log_{\mathfrak{X}} \cdot \partial_\mathrm{inv} \ .
\end{equation*}

\subsection{$(\varphi_L,\Gamma_L)$-modules}

Let $M$ be any finitely generated module over some topological ring $R$. The canonical topology of $M$ is defined to be the quotient topology with respect to a surjective $R$-module homomorphism $\alpha : R^m \longrightarrow M$. It makes $M$ into a topological $R$-module. If the multiplication in $R$ is only separately continuous then the module multiplication $R \times M \longrightarrow M$ is only separately continuous as well. Any $R$-module homomorphism between two finitely generated $R$-modules is continuous for the canonical topologies. We also need a semilinear version of this latter fact.

\begin{remark}\label{semilinear-cont}
Let $\psi : R \longrightarrow S$ be a continuous homomorphism of topological rings, let $M$ and $N$ be finitely generated $R$- and $S$-modules, respectively, and let $\alpha : M \longrightarrow N$ be any $\psi$-linear map (i.e., $\alpha(rm) = \psi(r)\alpha(m)$ for any $r \in R$ and $m \in M$); then $\alpha$ is continuous for the canonical topologies on $M$ and $N$.
\end{remark}
\begin{proof}
The map
\begin{align*}
  \alpha^{lin} : S \otimes_{R,\psi} M & \longrightarrow N \\
  s \otimes m & \longmapsto s \alpha(m)
\end{align*}
is $S$-linear. We pick free presentations $\lambda : R^\ell \twoheadrightarrow M$ and $\mu : S^m \twoheadrightarrow N$. Then we find an $S$-linear map $\beta$ such that the diagram
\begin{equation*}
  \xymatrix{
    R^\ell \ar@{>>}[d]_{\lambda} \ar[r]^-{\psi^\ell} & S^\ell = S \otimes_{R,\psi} R^\ell \ar@{>>}[d]_{\id \otimes \lambda} \ar[r]^-{\beta} & S^m \ar@{>>}[d]^{\mu} \\
    M \ar@/_2pc/[rr]^-{\alpha} \ar[r]^-{m \mapsto 1 \otimes m} & S \otimes_{R,\psi} M \ar[r]^-{\alpha^{lin}} & N   }
\end{equation*}
is commutative. All maps except possibly the lower left horizontal arrow are continuous. The universal property of the quotient topology then implies that $\alpha$ must be continuous as well.
\end{proof}

Suppose now that $M$ is finitely generated projective over $R$. Then the homomorphism $\alpha$ has a continuous section $\sigma : M \longrightarrow R^m$. Hence $M$ is topologically isomorphic to the submodule $\sigma(M)$ of $R^m$ (equipped with the subspace topology). Suppose further that $R$ is Hausdorff, resp.\ complete. Then $R^m$ is Hausdorff, resp.\ complete. We see that $\sigma(M)$ and $M$ are Hausdorff. Furthermore it follows that $\sigma(M) = \ker(\id_M - \sigma\circ\alpha)$ is closed in $R^m$. Hence, if $R$ is complete, then also $M$ is complete.

In our applications $R$ usually is a locally convex $K$-algebra. If such an $R$ is barrelled then the canonical topology on any $M$ is barrelled as well (cf.\ \cite{NFA} Ex.\ 4 after Cor.\ 6.16).

\begin{definition}\label{def:modR}
A $(\varphi_L,\Gamma_L)$-module $M$ over $\mathscr{R}_K(\mathfrak{X})$ is a finitely generated projective $\mathscr{R}_K(\mathfrak{X})$-module $M$ which carries a semilinear continuous (for the canonical topology) $o_L \setminus \{0\}$-action such that the $\mathscr{R}_K(\mathfrak{X})$-linear map
\begin{align*}
    \varphi_M^{lin} :  \mathscr{R}_K(\mathfrak{X}) \otimes_{\mathscr{R}_K(\mathfrak{X}),\varphi_L} M & \xrightarrow{\; \cong \;} M \\
    f \otimes m & \longmapsto f \varphi_M (m)
\end{align*}
is bijective (writing the action of $\pi_L$ on $M$ as $\varphi_M$). Let $\Mod_L(\mathscr{R}_K(\mathfrak{X}))$ denote the category of all  $(\varphi_L,\Gamma_L)$-modules over $\mathscr{R}_K(\mathfrak{X})$.
\end{definition}

The $(\varphi_L,\Gamma_L)$-modules over $\mathscr{R}_K(\mathfrak{X})$ are Hausdorff and complete. They also are barrelled since $\mathscr{R}_K(\mathfrak{X})$ as a locally convex inductive limit (Prop.\ \ref{regular}.i) of Fr\'echet spaces is barrelled (cf.\ \cite{NFA} Ex.\ 2 and 3 after Cor.\ 6.16).

We briefly discuss scalar extension for $(\varphi_L,\Gamma_L)$-modules. Let $K \subseteq E \subseteq \Cp$ be another complete intermediate field. First we make the following simple observation.

\begin{remark}\label{tensor-barrelled}
Let $V_1$ and $V_2$ be two barrelled locally convex $K$-vector spaces; then $V_1 \otimes_{K,\iota} V_2$ is barrelled as well.
\end{remark}
\begin{proof}
The inductive tensor product topology on $V_1 \otimes_K V_2$ is the finest locally convex topology such that all linear maps
\begin{align*}
  V_1 & \longrightarrow V_1 \otimes_K V_2 & \text{and}\qquad\qquad   V_2 & \longrightarrow V_1 \otimes_K V_2 & \text{for any $v_i \in V_i$} \\
  v & \longmapsto v \otimes v_2 &    v & \longmapsto v_1 \otimes v &
\end{align*}
are continuous. It is basically by definition that any locally convex final topology with respect to maps which originate from barrelled spaces, like the above one, is barrelled.
\end{proof}

\begin{lemma}\label{scalar-ext}
For any $(\varphi_L,\Gamma_L)$-module $M$ over $\mathscr{R}_K(\mathfrak{X})$ we have:
\begin{itemize}
  \item[i.]  $\mathscr{R}_E(\mathfrak{X}) \otimes_{\mathscr{R}_K(\mathfrak{X})} M = E\, \widehat{\otimes}_{K,\iota}\, M$;
  \item[ii.] $\mathscr{R}_E(\mathfrak{X}) \otimes_{\mathscr{R}_K(\mathfrak{X})} M$ is a $(\varphi_L,\Gamma_L)$-module over $\mathscr{R}_E(\mathfrak{X})$.
\end{itemize}
\end{lemma}
\begin{proof}
i. Obviously we have the algebraic identity $(E \otimes_K \mathscr{R}_K(\mathfrak{X})) \otimes_{\mathscr{R}_K(\mathfrak{X})} M = E \otimes_K M$. We claim that this is a topological identity $(E \otimes_{K,\iota} \mathscr{R}_K(\mathfrak{X})) \otimes_{\mathscr{R}_K(\mathfrak{X})} M = E \otimes_{K,\iota} M$, which extends to a topological isomorphism $(E\, \widehat{\otimes}_{K,\iota}\, \mathscr{R}_K(\mathfrak{X})) \otimes_{\mathscr{R}_K(\mathfrak{X})} M = E\, \widehat{\otimes}_{K,\iota}\, M$ . Note that on the left hand sides the topology is given as follows: By viewing $M$ as a topological direct summand of some $\mathscr{R}_K(\mathfrak{X})^m$ we realize the left hand sides as topological direct summands of $(E \otimes_{K,\iota} \mathscr{R}_K(\mathfrak{X}))^m$ and $(E\, \widehat{\otimes}_{K,\iota}\, \mathscr{R}_K(\mathfrak{X}))^m$, respectively. Hence our claim reduces to the case $M = \mathscr{R}_K(\mathfrak{X})$ where it is obvious. It remains to recall from Cor.\ \ref{scalar-ext-R}  that $\mathscr{R}_E(\mathfrak{X}) = E\, \widehat{\otimes}_{K,\iota}\, \mathscr{R}_K(\mathfrak{X})$.

ii. Clearly $\mathscr{R}_E(\mathfrak{X}) \otimes_{\mathscr{R}_K(\mathfrak{X})} M$ is finitely generated projective over $\mathscr{R}_E(\mathfrak{X})$. The $o_L \setminus \{0\}$-action extends by semilinearity to $\mathscr{R}_E(\mathfrak{X}) \otimes_{\mathscr{R}_K(\mathfrak{X})} M$ and satisfies
\begin{align*}
  \mathscr{R}_E(\mathfrak{X}) \otimes_{\mathscr{R}_E(\mathfrak{X}),\varphi_L} (\mathscr{R}_E(\mathfrak{X}) \otimes_{\mathscr{R}_K(\mathfrak{X})} M)  & = \mathscr{R}_E(\mathfrak{X}) \otimes_{\mathscr{R}_K(\mathfrak{X})} (\mathscr{R}_K(\mathfrak{X}) \otimes_{\mathscr{R}_K(\mathfrak{X}),\varphi_L} M) \\
  & \cong \mathscr{R}_E(\mathfrak{X}) \otimes_{\mathscr{R}_K(\mathfrak{X})} M \ .
\end{align*}
It remains to establish the continuity of the $o_L \setminus \{0\}$-action. Because of i. we may view it as the completed $E$-linear extension of the $o_L \setminus \{0\}$-action on $M$. Hence each individual element in $o_L \setminus \{0\}$-action certainly acts by a continuous linear endomorphism on $V := E \otimes_{K,\iota} M$ and then also on $\widehat{V} = E\, \widehat{\otimes}_{K,\iota}\, M$. We still have to check that the resulting group actions $o_L^\times \times V \longrightarrow V$ and $o_L^\times \times \widehat{V} \longrightarrow \widehat{V}$ are continuous. Since $M$ is barrelled it follows from Remark \ref{tensor-barrelled} that $V$ is barrelled. For the continuity of the action on $V$ it therefore suffices, by the usual Banach-Steinhaus argument, to check the continuity of the orbit maps
\begin{align*}
  \rho_v : o_L^\times & \longrightarrow E \otimes_{K,\iota} M \qquad\text{for any $v \in E \otimes_{K,\iota} M$} \\
  a & \longmapsto a(v) \ .
\end{align*}
But this follows easily from the continuity of the $o_L^\times$-action on $M$. Let $\mathcal{C}_c(o_L^\times,V)$ denote the locally convex $K$-vector space of all continuous $V$-valued maps on $o_L^\times$ equipped with the compact-open topology. By \cite{B-TG} X.3.4 Thm.\ 3 the continuity of the action $o_L^\times \times V \longrightarrow V$ is equivalent to the continuity of the linear map
\begin{align*}
  V & \longrightarrow \mathcal{C}_c(o_L^\times,V) \\
  v & \longmapsto \rho_v \ .
\end{align*}
Hence all solid arrows in the diagram
\begin{equation*}
  \xymatrix{
    V \ar[d]_{\subseteq} \ar[r] & \mathcal{C}_c(o_L^\times,V) \ar[d]^{\subseteq} \\
    \widehat{V} \ar@{-->}[r] & \mathcal{C}_c(o_L^\times,\widehat{V})   }
\end{equation*}
are continuous linear maps. Since $\mathcal{C}_c(o_L^\times,\widehat{V})$ is Hausdorff and complete (cf.\ \cite{NFA} Example in \S17) there is a unique continuous linear map $\widehat{V} \longrightarrow \mathcal{C}_c(o_L^\times,\widehat{V})$ which makes the diagram commutative. It corresponds to a continuous map $o_L^\times \times \widehat{V} \longrightarrow \widehat{V}$ which is easily seen to be the original $o_L^\times$-action on $\widehat{V}$. The latter therefore is continuous.
\end{proof}

At various points we will need the following technical descent result.

\begin{proposition}\label{descent}
Let $M$ be a $(\varphi_L,\Gamma_L)$-modules over $\mathscr{R}_K(\mathfrak{X})$. Then there exist a $p^{-p/(p-1)} \leq r_0 < 1$, a finitely generated projective $\mathcal{O}_K(\mathfrak{X} \setminus \mathfrak{X}(r_0))$-module $M_0$ with a semilinear continuous $o_L^\times$-action, and a semilinear continuous homomorphism
\begin{equation*}
  \varphi_{M_0} : M_0 \longrightarrow \mathcal{O}_K(\mathfrak{X} \setminus \mathfrak{X}(r_0^{1/p})) \otimes_{\mathcal{O}_K(\mathfrak{X} \setminus \mathfrak{X}(r_0))} M_0
\end{equation*}
such that the induced $\mathcal{O}_K(\mathfrak{X} \setminus \mathfrak{X}(r_0^{1/p}))$-linear map
\begin{equation*}
  \mathcal{O}_K(\mathfrak{X} \setminus \mathfrak{X}(r_0^{1/p})) \otimes_{\mathcal{O}_K(\mathfrak{X} \setminus \mathfrak{X}(r_0)), \varphi_L} M_0 \xrightarrow{\; \cong \;} \mathcal{O}_K(\mathfrak{X} \setminus \mathfrak{X}(r_0^{1/p})) \otimes_{\mathcal{O}_K(\mathfrak{X} \setminus \mathfrak{X}(r_0))} M_0
\end{equation*}
is an isomorphism and such that
\begin{equation*}
  \mathscr{R}_K(\mathfrak{X}) \otimes_{\mathcal{O}_K(\mathfrak{X} \setminus \mathfrak{X}(r_0))} M_0 = M
\end{equation*}
with the $o_L^\times$ actions on both sides as well as $\varphi_L \otimes \varphi_{M_0}$ and $\varphi_M$ corresponding to each other.
\end{proposition}
\begin{proof}
For some appropriate integer $m \geq 1$ we can view $M \subseteq \mathscr{R}_K(\mathfrak{X})^m$ as the image of some projector $\Pi \in M_{m\times m}(\mathscr{R}_K(\mathfrak{X}))$. The matrix $\Pi$ lies in $M_{m\times m}(\mathcal{O}_K(\mathfrak{X} \setminus \mathfrak{X}(r_0)))$ for some $r_0 \geq p^{-p/(p-1)}$, and we may define finitely generated projective $\mathcal{O}_K(\mathfrak{X} \setminus \mathfrak{X}(r))$-modules $M(r) := \Pi(\mathcal{O}_K(\mathfrak{X} \setminus \mathfrak{X}(r))^m)$ for any $r_0 \leq r < 1$. We have $M = \mathscr{R}_K(\mathfrak{X}) \otimes_{\mathcal{O}_K(\mathfrak{X} \setminus \mathfrak{X}(r))} M(r)$ and $M = \bigcup_{r_0 \leq r < 1} M(r)$.

Since $M(r)$ is finitely generated we further have $\varphi_M (M(r)) \subseteq M(r')$ for some $r' \geq r$. But any set of generators for $M(r)$ also is a set of generators for $M(r')$. It follows (cf.\ Lemma \ref{pi}) that $\varphi_M (M(r')) \subseteq M(r'^{1/p})$. The associated linear map \begin{align*}
  \mathcal{O}_K(\mathfrak{X} \setminus \mathfrak{X}(r'^{1/p})) \otimes_{\mathcal{O}_K(\mathfrak{X} \setminus \mathfrak{X}(r')), \varphi_L} M(r') & \longrightarrow M(r'^{1/p}) \\
  f \otimes m & \longmapsto f \varphi_M(m)
\end{align*}
has the property that its base change to $\mathscr{R}_K(\mathfrak{X})$ is an isomorphism. The cokernel being finitely generated must already vanish after enlarging $r'$ sufficiently. Then the map is surjective and, by the projectivity of the modules, splits. Hence the kernel is finitely generated as well and vanishes after further enlarging $r'$. By enlarging the initial $r_0$ we therefore may assume that
\begin{equation*}
  \mathcal{O}_K(\mathfrak{X} \setminus \mathfrak{X}(r^{1/p})) \otimes_{\mathcal{O}_K(\mathfrak{X} \setminus \mathfrak{X}(r)), \varphi_L} M(r) \cong M(r^{1/p}) \qquad\text{for any $r_0 \leq r < 1$}.
\end{equation*}
By Remark \ref{semilinear-cont} the map $\varphi_M : M(r) \longrightarrow M(r^{1/p})$ is continuous.

By assumption the orbit map $\rho_m : o_L^\times \longrightarrow M$, for any $m \in M$, which sends $a$ to $a(m)$, is continuous. Hence its image is compact and, in particular, bounded. But $M$ is, as a consequence of Prop.\ \ref{regular}.i, the regular inductive limit of the $M(r)$. It follows  that the image of $\rho_m$ already is contained in some $M(r)$ and is bounded in the canonical topology of $M(r)$ as an $\mathcal{O}_K(\mathfrak{X} \setminus \mathfrak{X}(r))$-module. Using Prop.\ \ref{compactoid}.i we obtain that the image of $\rho_m$ is compactoid in $M(r)$. If we apply this to finitely many generators of $M(r_0)$ then we see that, by further enlarging $r_0$, we also may assume that the $o_L^\times$-action on $M$ preserves $M(r)$ for any $r \geq r_0$. By \cite{PGS} Cor.\ 3.8.39 the continuous inclusion $M(r) \subseteq M$ restricts to a homeomorphism between the image of $\rho_m$ in $M(r)$ and its image in $M$. It follows that $\rho_m : o_L^\times \longrightarrow M(r)$ is continuous. On the other hand, for each individual $a \in o_L^\times$, the map $a : M(r) \longrightarrow M(r)$ is $a_*$-linear and hence continuous by Remark \ref{semilinear-cont}. Together we have shown that the $o_L^\times$-action on $M(r)$ is separately continuous. Since $o_L^\times$ is compact and the Fr\'echet space $M(r)$ is barrelled it is, in fact, jointly continuous by the nonarchimedean Banach-Steinhaus theorem.
\end{proof}

\begin{proposition}\label{differentiable}
Any continuous (for the canonical topology) semilinear $\Gamma_L$-action on a finitely generated projective module $M$ over any of the rings $\mathcal{O}_K(\mathfrak{X} \setminus \mathfrak{X}_n)$, for $n \geq 1$, or $\mathscr{R}_K(\mathfrak{X})$ is differentiable.
\end{proposition}
\begin{proof}
First we consider $M$ over $\mathscr{R}_K(\mathfrak{X})$. As seen from its proof the descent result Prop.\ \ref{descent} holds equally true without a $\varphi_M$. Hence, for some sufficiently big $n_0$, we find a finitely generated projective $\mathcal{O}_K(\mathfrak{X} \setminus \mathfrak{X}_{n_0})$-module $M_{n_0}$ with a continuous semilinear $\Gamma_L$-action such that $M = \mathscr{R}_K(\mathfrak{X}) \otimes_{\mathcal{O}_K(\mathfrak{X} \setminus \mathfrak{X}_{n_0})} M_{n_0}$ as $\Gamma_L$-modules. For any $n \geq n_0$, the finitely generated projective module $M_n := \mathcal{O}_K(\mathfrak{X} \setminus \mathfrak{X}_n) \otimes_{\mathcal{O}_K(\mathfrak{X} \setminus \mathfrak{X}_{n_0})} M_{n_0}$ over $\mathcal{O}_K(\mathfrak{X} \setminus \mathfrak{X}_n)$ carries the continuous semilinear (diagonal) $\Gamma_L$-action. The $\Gamma_L$-equivariant maps $M_n \longrightarrow M_{n*1} \longrightarrow M$ are continuous (by Remark \ref{semilinear-cont}), and $M = \bigcup_{n \geq n_0} M_n$. This reduces the differentiability of $M$ to the differentiability of each $M_n$.

So, in the following we fix an $n \geq 1$ and consider $M$ over $\mathcal{O}_K(\mathfrak{X} \setminus \mathfrak{X}_n)$. We abbreviate $\mathfrak{Y} := \mathfrak{X} \setminus \mathfrak{X}_n$. According to Prop.\ \ref{quasi-Stein} the quasi-Stein space $\mathfrak{Y}$ has an admissible covering by an increasing sequence of $\Gamma_L$-invariant affinoid subdomains $\mathfrak{X}(s_1,s_1') \subseteq \ldots \subseteq \mathfrak{X}(s_i,s_i') \subseteq \ldots$ such that the restriction maps $\mathcal{O}_K(\mathfrak{X}(s_{i+1},s_{i+1}')) \longrightarrow \mathcal{O}_K(\mathfrak{X}(s_i,s_i'))$ have dense image. We then have the finitely generated projective modules $M_i := \mathcal{O}_K(\mathfrak{X}(s_i,s_i')) \otimes_{\mathcal{O}_K(\mathfrak{Y})} M$ over $\mathcal{O}_K(\mathfrak{X}(s_i,s_i'))$ with a continuous semilinear (diagonal) $\Gamma_L$-action. The covering property implies that $M = \varprojlim_i M_i$ holds true topologically for the canonical topologies (as well as $\Gamma_L$-equivariantly). This reduces the differentiability of $M$ to the differentiability of the $\Gamma_L$-action on each $M_i$. But, as we have discussed before Def.\ \ref{def:L-analytic}, the $\Gamma_L$-action on $\mathcal{O}_K(\mathfrak{X}(s_i,s_i'))$ satisfies the condition \eqref{f:condition-locan}. Hence, by Prop.\ \ref{locan3}, the $\Gamma_L$-action on $M_i$ even is locally $\Qp$-analytic.

\textit{Addendum:} First suppose that $M$ over $\mathscr{R}_K(\mathfrak{X})$ is $K$-analytic. Then each $M_n$ is $L$-analytic since $M_n$ is a $\Lie(\Gamma_L$-invariant $K$-vector subspace of $M$. Next suppose that $M$ over $\mathcal{O}_K(\mathfrak{X} \setminus \mathfrak{X}_n)$ is $L$-analytic. Then each $M_i$ is $L$-analytic since the natural map $M \longrightarrow M_i$ has dense image (cf.\ \cite{ST0} Thm.\ in \S3).
\end{proof}

This result allows us to introduce the full subcategory $\Mod_{L,an}(\mathscr{R}_K(\mathfrak{X}))$ of all $L$-analytic (Def.\ \ref{def:L-analytic}) $(\varphi_L,\Gamma_L)$-modules in $\Mod_L(\mathscr{R}_K(\mathfrak{X}))$.

There is a useful duality functor on the category $\Mod_L(\mathscr{R}_K(\mathfrak{X}))$. Let $M$ be $(\varphi_L,\Gamma_L)$-module over $\mathscr{R}_K(\mathfrak{X})$. The dual module $M^* := \Hom_{\mathscr{R}_K(\mathfrak{X})}(M,\mathscr{R}_K(\mathfrak{X}))$ again is finitely generated projective over $\mathscr{R}_K(\mathfrak{X})$.

\begin{remark}\label{simple-conv}
For any finitely generated projective $\Mod_L(\mathscr{R}_K(\mathfrak{X}))$-module $N$ the canonical topology on $\Hom_{\mathscr{R}_K(\mathfrak{X})}(N,\mathscr{R}_K(\mathfrak{X}))$ coincides with the topology of pointwise convergence.
\end{remark}
\begin{proof}
Since the formation of both topologies commutes with direct sums it suffices to consider $N = \mathscr{R}_K(\mathfrak{X})$, in which case the assertion is straightforward..
\end{proof}

It is a $(\varphi_L,\Gamma_L)$-module with respect to
\begin{equation*}
  \gamma(\alpha) := \gamma \circ \alpha \circ \gamma^{-1} \qquad\text{and}\qquad \varphi_{M^*}(\alpha) := \varphi_L^{lin} \circ (\id_{\mathscr{R}_K(\mathfrak{X})} \otimes\, \alpha) \circ (\varphi_M^{lin})^{-1}
\end{equation*}
for any $\gamma \in \Gamma_L$ and any $\alpha \in \Hom_{\mathscr{R}_K(\mathfrak{X})}(M,\mathscr{R}_K(\mathfrak{X}))$. Each individual such operator on $M^*$ is continuous by Remark \ref{semilinear-cont}. If we use $\varphi_M^{lin}$ and $\varphi_L^{lin}$ to identify $\Hom_{\mathscr{R}_K(\mathfrak{X})}(M,\mathscr{R}_K(\mathfrak{X}))$ and
\begin{multline*}
  \Hom_{\mathscr{R}_K(\mathfrak{X})}(\mathscr{R}_K(\mathfrak{X}) \otimes_{\mathscr{R}_K(\mathfrak{X}),\varphi_L} M, \mathscr{R}_K(\mathfrak{X}) \otimes_{\mathscr{R}_K(\mathfrak{X}),\varphi_L} \mathscr{R}_K(\mathfrak{X}))  \\ = \Hom_{\mathscr{R}_K(\mathfrak{X})}(M, \mathscr{R}_K(\mathfrak{X}) \otimes_{\mathscr{R}_K(\mathfrak{X}),\varphi_L} \mathscr{R}_K(\mathfrak{X}))
\end{multline*}
then $\varphi_{M^*}^{lin}$ becomes the map
\begin{align*}
   \mathscr{R}_K(\mathfrak{X}) \otimes_{\mathscr{R}_K(\mathfrak{X}),\varphi_L} \Hom_{\mathscr{R}_K(\mathfrak{X})}(M,\mathscr{R}_K(\mathfrak{X})) & \longrightarrow \Hom_{\mathscr{R}_K(\mathfrak{X})}(M, \mathscr{R}_K(\mathfrak{X}) \otimes_{\mathscr{R}_K(\mathfrak{X}),\varphi_L} \mathscr{R}_K(\mathfrak{X})) \\
  f \otimes \alpha & \longmapsto [m \mapsto f \otimes \alpha(m)] \
\end{align*}
which does not involve the $(\varphi_L,\Gamma_L)$-structure any longer. To see that the latter map is bijective we may first reduce, since $M$ is projective, to the case of a finitely generated free module $M$ and then to the case $M = \mathscr{R}_K(\mathfrak{X})$, in which the bijectivity is obvious. Since $\Gamma_L$ is compact and $M^*$ is barrelled it remains, by the nonarchimedean Banach-Steinhaus theorem, to show, for the joint continuity of the $\Gamma_L$-action on $M^*$, that for any $\alpha \in M^*$ the map $\Gamma_L \longrightarrow M^*$ sending $\gamma$ to $\gamma(\alpha)$ is continuous. Because of Remark \ref{simple-conv} it suffices to check that, for any $m \in M$, the map $\Gamma_L \longrightarrow M$ sending $\gamma$ to $\gamma(\alpha(m)) = \gamma(\alpha(\gamma^{-1}(m))$ is continuous. This is a straightforward consequence of the continuity of the $\Gamma_L$-action on $M$ and on $\mathscr{R}_K(\mathfrak{X})$.

As an application we make the following technically helpful observation.

\begin{remark}\label{summand-of-free}
For any $M$ in $\Mod_L(\mathscr{R}_K(\mathfrak{X}))$ the $(\varphi_L,\Gamma_L)$-module $M \oplus M^*$ is free over $\mathscr{R}_K(\mathfrak{X})$.
\end{remark}
\begin{proof}
As a module $M$ is isomorphic to a direct sum of invertible ideals. Hence it suffices to consider an invertible ideal $I$ in $\mathscr{R}_K(\mathfrak{X})$. Then $I \oplus I^* \cong I \oplus I^{-1} \cong \mathscr{R}_K(\mathfrak{X}) \oplus I I^{-1} = \mathscr{R}_K(\mathfrak{X}) \oplus \mathscr{R}_K(\mathfrak{X})$ by Cor.\ \ref{pruefer2}.ii.
\end{proof}

Of course, everything above makes sense and is valid for $\mathfrak{B}$ replacing $\mathfrak{X}$. In particular, we have the categories $\Mod_{L,an}(\mathscr{R}_K(\mathfrak{B})) \subseteq \Mod_L(\mathscr{R}_K(\mathfrak{B}))$ of all $L$-analytic, resp.\ of all, $(\varphi_L,\Gamma_L)$-modules over $\mathscr{R}_K(\mathfrak{B})$.

We also add the following fact which will be crucial for the definition of etale $L$-analytic  $(\varphi_L,\Gamma_L)$-modules in section \ref{sec:etalepgm}.

\begin{proposition}\label{unitsR}
$\mathscr{R}_K(\mathfrak{X})^\times = \mathscr{E}_K^\dagger(\mathfrak{X})^\times$.
\end{proposition}
\begin{proof}
Let $f \in \mathscr{R}_K(\mathfrak{X})^\times$. We have $f \in \mathcal{O}_K(\mathfrak{X} \setminus \mathfrak{X}_n)^\times$ for some sufficiently large $n$. It suffices to show that $f$ is bounded. By symmetry then $f^{-1}$ is bounded as well so that $f \in \mathcal{O}_K^b(\mathfrak{X} \setminus \mathfrak{X}_n)^\times \subseteq \mathscr{E}_K^\dagger(\mathfrak{X})^\times$. Since boundedness can be checked after scalar extension to $\Cp$ we may assume that $K = \Cp$. Using the LT-isomorphism this reduces us to showing that any unit in $\mathcal{O}_K(\mathfrak{B} \setminus \mathfrak{B}_n)$ is bounded. But this is well known to follow from \cite{Laz} Prop.\ 4.1).
\end{proof}

\section{Construction of $(\varphi_L,\Gamma_L)$-modules}

\subsection{An application of the Colmez-Sen-Tate method}
\label{sec:key}

In the following the Colmez-Sen-Tate formalism as formulated in \cite{BC} will be of crucial technical importance. In this section we therefore recall it, complement it somewhat, and prove an essential additional result.

Let $A$ be an $L$-Banach algebra with a submultiplicative norm $\|\ \|_A$. For any complete intermediate field $\Qp \subseteq K \subseteq \Cp$ we denote by $A_K := K\, \widehat{\otimes}_L\, A$ the tensor product completed with respect to the tensor product norm $\|\ \|_{A_K} := |\ | \otimes \|\ \|_A$.
\footnote{There is a subtle point here which one has to keep in mind. Let $A$ be a reduced affinoid $L$-algebra with supremum norm $\|\ \|_{\mathrm{sup},A}$. The affinoid $K$-algebra $A_K$ again is reduced (\cite{Co1} Lemma 3.3.1(1)). But, in general, the supremum norm $\|\ \|_{\mathrm{sup},A_K}$ is NOT equal to the tensor product norm $\|\ \|_{A_K} := |\ | \otimes \|\ \|_{\mathrm{sup},A}$; the two are only equivalent (\cite{BGR} Thm.\ 6.2.4/1).

Suppose that $K/L$ is unramified. Then $gr_{|\ |}(K) = k_K \otimes_{k_L} gr_{|\ |}(L)$, where $k_L$ denotes the residue class field of $L$. We therefore have
\begin{equation*}
  gr_{\|\ \|_{A_K}}(A_K) = gr_{|\ |}(K) \otimes_{gr_{|\ |}(L)} gr_{\|\ \|_{\mathrm{sup},A}}(A) = k_K \otimes_{k_L} gr_{\|\ \|_{\mathrm{sup},A}}(A) \ .
\end{equation*}
The supremum norm being power-multiplicative the algebra $gr_{\|\ \|_{\mathrm{sup},A}}(A)$ is reduced. Since $k_K /k_L$ is unramified the above right hand algebra is reduced as well. Hence $gr_{\|\ \|_{A_K}}(A_K)$ is reduced, which in turn implies that $\|\ \|_{A_K}$ is power-multiplicative and therefore must be equal to $\|\ \|_{\mathrm{sup},A_K}$.}
If $\Qp \subseteq F \subseteq \Cp$ is an arbitrary intermediate field then $\widehat{F}$  denotes its completion.

The Galois group $G_L$ acts continuously and isometrically on $\Cp$ and hence continuously, semilinearly, and isometrically on $A_{\Cp}$ (through the first factor). Of course, it then also acts continuously on $\GL_m(A_{\Cp})$ for any $m \geq 1$.

\begin{remark}\label{AST}
$(A_{\Cp})^{G_K} = A_K$.
\end{remark}
\begin{proof}
It suffices to show that $(\Cp\, \widehat{\otimes}_K\, A)^{G_K} \subseteq A$. It is easy to see that any element of $\Cp\, \widehat{\otimes}_K\, A$ is contained in $\Cp\, \widehat{\otimes}_K\, A_0$ for some Banach subspace of countable type $A_0 \subseteq A$. Hence we may assume that $A$ is of countable type. The case of a finite dimensional $A$ being trivial we then further may assume, by \cite{PGS} Cor.\ 2.3.9, that $A = c_0(K)$ (notation as in the proof of Lemma \ref{descent-compactoid}). In this case we have $(\Cp\, \widehat{\otimes}_K\, c_0(K))^{G_K} = c_0(\Cp)^{G_K} = c_0(K)$.
\end{proof}

We put $L_n := L(\mu_{p^n})$ and $L^{cyc} := \bigcup_n L_n$ and we let $H_L^{cyc} := \Gal(\overline{L}/L^{cyc})$ and $\Gamma_L^{cyc} := \Gal(L^{cyc}/L)$. According to \cite{BC} Prop.s 3.1.4 and 4.1.1 the Banach algebra $A_{\Cp}$ verifies the Colmez-Sen-Tate conditions for any constants $c_1, c_2 > 0$ and $c_3 > \frac{1}{p-1}$. It is not necessary to here recall the content of the Colmez-Sen-Tate conditions; for the sake of clarity we only point out that in the notations of loc.\ cit.\ we have $\widetilde{\Lambda} = A_{\Cp}$ and $\Lambda_{H_L^{cyc},n} = L_n \otimes_L A = A_{L_n}$. What is important is that this has the following consequences.

\begin{proposition}\phantomsection\label{TS}
\begin{itemize}
  \item[i.] For any sufficiently large $n$ there is a $G_L$-invariant decomposition into $A_{L_n}$-submodules $A_{\widehat{L^{cyc}}} = A_{L_n} \oplus X_{L,n}$; in particular, $X_{L,n}^{\Gamma_{L_n}^{cyc}} = 0$.
  \item[ii.] Given any continuous $1$-cocycle $c : G_L \longrightarrow \GL_m(A_{\Cp})$ there exists a finite Galois extension $L'/L$ and an integer $n \geq 0$ such that there is a continuous $1$-cocycle on $G_L$ which is cohomologous to $c$, has values in $\GL_m(A_{L_n'})$, and is trivial on $H_{L'}^{cyc}$.
\end{itemize}
\end{proposition}
\begin{proof}
i. This is immediate from the condition (TS2) in \cite{BC} \S3.1. ii. This is a somewhat less precise form of \cite{BC} Prop.\ 3.2.6 (the $k$ there is irrelevant since $p$ is invertible in $\widetilde{\Lambda}$).
\end{proof}

Let $P$ be a finitely generated free $A_{\Cp}$-module of rank $r$ and consider any continuous semilinear action of $G_L$ on $P$. Note that $P$ as well as $\End_{A_{\Cp}}(P)$ are naturally topological $A_{\Cp}$-modules.

\begin{corollary}\label{TS2}
We have:
\begin{itemize}
  \item[i.] $P^{H_L^{cyc}}$ is a projective $A_{\widehat{L^{cyc}}}$-module of rank $r$, and $A_{\Cp} \otimes_{A_{\widehat{L^{cyc}}}} P^{H_L^{cyc}} = P$.
  \item[ii.] $P^{H_L^{cyc}}$ contains, for any sufficiently large $n$, a $\Gamma_L^{cyc}$-invariant $A_{L_n}$-submodule $Q_n$ which satisfies:
      \begin{itemize}
        \item[a)] $Q_n$ is projective of rank $r$,
        \item[b)] $A_{\widehat{L^{cyc}}} \otimes_{A_{L_n}} Q_n = P^{H_L^{cyc}}$, and
        \item[c)] there is a finite extension $L' / L_n$ such that $A_{L'} \otimes_{L_n} Q_n$ is a free $A_{L'}$-module.
      \end{itemize}
  \item[iii.] for any $Q_n$ as in ii. there is an $m \geq n$ such that $P^{G_L} \subseteq A_{L_m} \otimes_{A_{L_n}} Q_n$; in particular, $P^{G_L}$ is a submodule of a finitely generated free $A$-module.
\end{itemize}
\end{corollary}
\begin{proof}
i. and ii. We fix a free $A$-module $P_0$ of rank $r$ such that $A_{\Cp} \otimes_A P_0 = P$, and we let act $G_L$ continuously and semilinearly on $\End_{A_{\Cp}}(P)$ by
\begin{align*}
    G_L \times \End_{A_{\Cp}}(P) & \longrightarrow \End_{A_{\Cp}}(P) \\
    (\sigma, \alpha) & \longmapsto (\sigma \otimes \id_{P_0}) \circ \alpha \circ (\sigma^{-1} \otimes \id_{P_0}) \ .
\end{align*}
Then
\begin{align*}
    c : G_L & \longrightarrow \Aut_{A_{\Cp}}(P) \\
    \sigma & \longmapsto \sigma \circ (\sigma^{-1} \otimes \id_{P_0})
\end{align*}
is a continuous 1-cocycle. By Prop.\ \ref{TS}.ii there are a finite Galois extension $L'/L$, a natural number $n \geq 0$, and an element $\beta \in \Aut_{A_{\Cp}}(P)$ such that the cohomologous cocycle
\begin{equation*}
    c'(\sigma) = \beta \circ c(\sigma) \circ (\sigma \otimes \id_{P_0}) \circ \beta^{-1} \circ (\sigma^{-1} \otimes \id_{P_0}) = \beta \circ \sigma \circ \beta^{-1} \circ (\sigma^{-1} \otimes \id_{P_0})
\end{equation*}
on $G_L$ has values in $\Aut_{A_{L_n'}}(A_{L_n'} \otimes_A P_0)$ and is trivial on $H_{L'}^{cyc}$. It follows that $\sigma = \beta^{-1} \circ (\sigma \otimes \id_{P_0}) \circ \beta$ for any $\sigma \in H_{L'}^{cyc}$ and hence that
\begin{equation*}
    P^{H_{L'}^{cyc}} = \beta^{-1}(P^{{H_{L'}^{cyc}} \otimes \id_{P_0} = 1}) = \beta^{-1}(A_{\Cp}^{H_{L'}^{cyc}} \otimes_A P_0) = \beta^{-1}(A_{\widehat{L'^{cyc}}} \otimes_A P_0) \ .
\end{equation*}
In particular, $P^{H_{L'}^{cyc}}$ is a free $A_{\widehat{L'^{cyc}}}$-module of rank $r$ and $A_{\Cp} \otimes_{A_{\widehat{L'^{cyc}}}} P^{H_{L'}^{cyc}} = P$. Moreover, the residual action of $G_L/H_{L'}^{cyc}$ on $P^{H_{L'}^{cyc}}$ leaves the free $A_{L_n'}$-submodule $Q' := \beta^{-1}(A_{L_n'} \otimes_A P_0)$ of rank $r$  invariant.

By the usual finite Galois descent formalism (cf.\ \cite{BC} Prop.\ 2.2.1 and \cite{B-AC} II.5.3 Prop.\ 4) we conclude that $P^{H_L^{cyc}}$ is a projective $A_{\widehat{L^{cyc}}}$-module of rank $r$ with $A_{\Cp} \otimes_{A_{\widehat{L^{cyc}}}} P^{H_L^{cyc}} = P$ and that the $\Gamma_L^{cyc}$-action on $P^{H_L^{cyc}}$ leaves invariant the projective $A_{L_n'}^{H_L^{cyc}}$-submodule $Q := (Q')^{H_L^{cyc}}$ of rank $r$. By construction we have $A_{\widehat{L^{cyc}}} \otimes_{A_{L_n'}^{H_L^{cyc}}} Q = P^{H_L^{cyc}}$. By enlarging $n$ we achieve that $A_{L_n'}^{H_L^{cyc}} = A_{L_n}$.

iii. We fix a $Q_n$ as in ii. We may assume that $L'/L_n$ is Galois and that $L^{cyc} \cap L' = L_n$. By assumption the $A_{L'}$-module $Q' := A_{L'} \otimes_{A_{L_n}} Q_n$ is free. For any sufficiently big $m \geq n$ we have
\begin{align*}
  P^{H_{L'}^{cyc}} & = A_{\widehat{L'^{cyc}}} \otimes_{A_{\widehat{L^{cyc}}}} P^{H_L^{cyc}} = A_{\widehat{L'^{cyc}}} \otimes_{A_{L_n}} Q_n = A_{\widehat{L'^{cyc}}} \otimes_{A_{L'}} Q' \\
  & = (A_{L'_m} \otimes_{A_{L'}} Q') \oplus (X_{L',m} \otimes_{A_{L'}} Q') \ ,
\end{align*}
where the first and the last identity uses finite Galois descent and Prop.\ \ref{TS}.i, respectively. Suppose that we find sufficiently big integers $m \geq \ell \geq n$ such that
\begin{equation}\label{f:X-inv}
  (X_{L',m} \otimes_{A_{L'}} Q')^{\Gamma_{L'_\ell}^{cyc}} = 0 \ .
\end{equation}
We then obtain
\begin{equation*}
  P^{G_L} \subseteq (A_{L'_m} \otimes_{A_{L'}} Q')^{H_L^{cyc}} = (A_{L'_m} \otimes_{A_{L_n}} Q_n)^{H_L^{cyc}} = A_{L_m} \otimes_{A_{L_n}} Q_n
\end{equation*}
with the last identity again using finite Galois descent.

In order to establish \eqref{f:X-inv} we note that the group $\Gamma_{L'}^{cyc}$ acts $A_{L'}$-linearly on $Q'$. The same computation as in the proof of Prop.\ \ref{locan3} shows that for a sufficiently big $\ell \geq n$ we have
\begin{equation*}
  \|(\gamma - 1)(x)\|_{Q'} < p^{-c_3} \|x\|_{Q'}  \qquad\text{for any $\gamma \in \Gamma_{L'_\ell}^{cyc}$ and $x \in Q'$}.
\end{equation*}
We fix an element $\gamma_0 \in \Gamma_{L'_\ell}^{cyc}$ of infinite order. The Colmez-Sen-Tate condition (TS3) in \cite{BC} \S3.1 says that, for $m \geq \ell$ sufficiently big, we have
\begin{equation*}
   \|(\gamma_0 - 1)(a)\|_{A_{\Cp}} \geq p^{-c_3} \|a\|_{A_{\Cp}}    \qquad\text{for any $a \in X_{L',m}$}.
\end{equation*}
We claim that $\gamma_0$ has no nonzero fixed vector in $X_{L',m} \otimes_{A_{L'}} Q'$.

Let $e_1, \ldots, e_r$ be a basis of $Q'$ over $A_{L'}$. We equip $P = A_{\Cp} e_1 \oplus \ldots \oplus A_{\Cp} e_r$ with the norm $\|\sum_{i=1}^r a_i e_i\| := \max_i \|a_i\|_{A_{\Cp}}$. We may assume that $\|\ \|_{Q'} = \|\ \|_{|Q'}$. For any $0 \neq y = \sum_i a_i e_i \in X_{L',m} \otimes_{A_{L'}} Q'$ (with $a_i \in X_{L',m}$) we have
\begin{equation*}
   \|\sum_i (\gamma_0 - 1)(a_i) \cdot e_i \| = \max_i \|(\gamma_0 - 1)(a_i)\|_{A_{\Cp}} \geq p^{-c_3} \max_i \|a_i\|_{A_{\Cp}}
\end{equation*}
and
\begin{equation*}
  \| \sum_i \gamma_0(a_i)(\gamma_0 - 1)(e_i) \| \leq  \max_i \|a_i\|_{A_{\Cp}} \cdot \|(\gamma_0 - 1)(e_i)\|_{Q'} < p^{-c_3} \max_i \|a_i\|_{A_{\Cp}} \ .
\end{equation*}
It follows that
\begin{equation*}
  (\gamma_0 - 1)(y) = (\gamma_0 - 1)(\sum_i a_i e_i ) = \sum_i (\gamma_0 - 1)(a_i) \cdot e_i + \sum_i \gamma_0(a_i)(\gamma_0 - 1)(e_i) \neq 0 \ .
\end{equation*}
\end{proof}

We fix an $n$ together with $Q_n$ as in Cor.\ \ref{TS2}.ii and such that, for simplicity, $\Gamma_{L_n}^{cyc}$ is topologically cyclic. The group $\Gamma_{L_n}^{cyc}$ is locally $\Qp$-analytic since it is, via the cyclotomic character $\chi_{cyc}$, an open subgroup of $\Zp^\times$. It acts $A_{L_n}$-linearly on $Q_n$. The condition \eqref{f:condition-locan} trivially holds for the trivial $\Gamma_{L_n}^{cyc}$-action on $A_{L_n}$. Therefore Prop.\ \ref{locan3} applies so that:

\begin{corollary}\label{qnlocan}
The $\Gamma_{L_n}^{cyc}$-action on $Q_n$ is locally $\Qp$-analytic.
\end{corollary}

In particular, we have the linear operator $\nabla_{Sen}$ on $Q_n$ which is given by the derived action of $\nabla_{cyc} := \Lie(\chi_{cyc})^{-1}(1) \in \Lie(\Gamma_{L_n}^{cyc})$. If $\gamma$ is a topological generator of $\Gamma_{L_n}^{cyc}$ then $\nabla_{Sen}$ is explicitly given by
\begin{equation*}
  \nabla_{Sen} = \frac{\log (\gamma^{p^j})}{\log \chi_{cyc}(\gamma^{p^j})} \quad\text{for any sufficiently big $j \in \ZZ_{\geq 0}$}.
\end{equation*}
The notation $\nabla_{Sen}$ is justified by the fact that, being defined by deriving the action of $\Gamma_L^{cyc}$, these operators for various $n$ and $Q_n$ all are compatible in an obvious sense.

\begin{lemma}\label{Sen-zero}
Suppose that $\nabla_{Sen} = 0$. Then $P^{G_L}$ is a projective $A$-module of rank $r$ and $A_{\Cp} \otimes_A P^{G_L} = P$.
\end{lemma}
\begin{proof}
The assumption implies that $\gamma^{p^j}$ for any sufficiently big $j$ fixes $Q_n$. Replacing $n$ by $n+j$ and $Q_n$ by $A_{L_{n+j}} \otimes_{A_{L_n}} Q_n$ we therefore may assume that $Q_n \subseteq P^{G_{L_n}}$ and that $n$ satisfies Prop.\ \ref{TS}.i. We then compute
\begin{align*}
  P^{G_{L_n}} & = (P^{H_L^{cyc}})^{\Gamma_{L_n}^{cyc}} = (A_{\widehat{L^{cyc}}} \otimes_{A_{L_n}} Q_n)^{\Gamma_{L_n}^{cyc}} \\
   & = ((A_{L_n} \oplus X_{L,n}) \otimes_{A_{L_n}} Q_n)^{\Gamma_{L_n}^{cyc}} \\
   & = Q_n \oplus  (X_{L,n} \otimes_{A_{L_n}} Q_n)^{\Gamma_{L_n}^{cyc}} \\
   & = Q_n \oplus  (X_{L,n}^{\Gamma_{L_n}^{cyc}} \otimes_{A_{L_n}} Q_n) \\
   & = Q_n
\end{align*}
where the fourth identity uses the projectivity of $Q_n$. The assertion follows from this by Galois descent.
\end{proof}

Now we suppose given an $L$-Banach algebra $B$ (with submultiplicative norm) on which the group $o_L^\times$ acts continuously by ring automorphisms. By linearity and continuity this action extends to a continuous $o_L^\times$-action on $B_{\Cp}$ (compare the proof of Lemma \ref{scalar-ext}.ii). It allows us to introduce the twisted $G_L$-action on $B_{\Cp} = \Cp \widehat{\otimes}_L B$ by
\begin{equation*}
  ^{\sigma*}(c \otimes b) := \sigma(c) \otimes \tau(\sigma^{-1})(b) \qquad\text{for $c \in \Cp$, $b \in B$, and $\sigma \in G_L$}.
\end{equation*}
The fixed elements $A := B_{\Cp}^{G_L,*}$ in $B_{\Cp}$ with respect to this twisted action form an $L$-Banach algebra. We assume that
\begin{equation}\label{f:form}
  A_{\Cp} = B_{\Cp} \ .
\end{equation}

We also suppose that our free $A_{\Cp}$-module $P$ is of the form
\begin{equation*}
  P = B_{\Cp} \otimes_B P_0
\end{equation*}
for some finitely generated projective $B$-module $P_0$ which carries a continuous semilinear action of $o_L^\times$. By the arguments in the proof of Lemma \ref{scalar-ext} this action extends to a continuous semilinear action on $P$ (and we have $P = \Cp\, \widehat{\otimes}_L\, P_0$). We then have a twisted $G_L$-action on $P$ as well by
\begin{equation*}
  ^{\sigma*}(b \otimes x) := {^{\sigma*}b} \otimes \tau(\sigma^{-1})(x) \qquad\text{for $b \in B_{\Cp}$, $x \in P_0$, and $\sigma \in G_L$}.
\end{equation*}
The corresponding fixed elements $P^{G_L,*}$ form an $A$-module. For an element in $P$ of the form $c \otimes x$ with $c \in \Cp$ and $x \in P_0$ we, of course, simply have $^{\sigma*}(c \otimes x) = \sigma(c) \otimes \tau(\sigma^{-1})(x)$. Using this observation (and again the reasoning in the proof of Lemma \ref{scalar-ext}.ii) one checks that this twisted $G_L$-action on $P$ is continuous. It follows that $P^{G_L,*}$ is closed in $P$. Obviously the $\Gamma_L$-action on $P^{G_L,*}$ then is continuous for the topology induced by the (canonical) topology of $P$. In our later applications we will show that the $A$-module $P^{G_L,*}$ is finitely generated projective. As a consequence of the open mapping theorem there is only one complete and Hausdorff module topology on a finitely generated module over a Banach algebra (cf.\ \cite{BGR} 3.7.3 Prop.\ 3). Hence in those applications we will have that the $\Gamma_L$-action on $P^{G_L,*}$ is continuous for the canonical topology as an $A$-module.

Our goal is the following sufficient condition for the vanishing of $\nabla_{Sen}=0$.

\begin{proposition}\label{nabsenul}
If the actions of $o_L^\times$ on $B$ and $P_0$ are locally $L$-analytic then $\nabla_{Sen}=0$ on $Q_n$.
\end{proposition}

We start with some preliminary results. First we consider any $p$-adic Lie group $H$, any $L$-Banach space representation $X$ of $H$, and any $L$-Banach space $Y$. Then $X \widehat{\otimes}_L Y$ becomes an $L$-Banach space representation of $H$ through the $H$-action on the first factor. We denote by $X^H$ the Banach subspace of $H$-fixed vectors in $X$ and by $X_{la}$ the $H$-invariant subspace of locally analytic vectors with respect to $H$. As in \cite{Eme} Def.\ 3.5.3 and Thm.\ 3.5.7 we view $X_{la}$ as the locally convex inductive limit
\begin{equation*}
  X_{la} = \varinjlim_U X_{\mathbb{U}-an} \ ,
\end{equation*}
where $U$ runs over all analytic open subgroups of $H$, $\mathbb{U}$ denotes the rigid analytic group corresponding to $U$, and $X_{\mathbb{U}-an}$ is the $L$-Banach space of $\mathbb{U}$-analytic vectors in $X$ (loc.\ cit.\ Def.\ 3.3.1 and Prop.\ 3.3.3). If $C^{an}(\mathbb{U},X) = \mathcal{O}(\mathbb{U}) \widehat{\otimes}_{\Qp} X$ denotes the $L$-Banach space of $X$-valued rigid analytic function on $\mathbb{U}$ then the orbit maps $x \longrightarrow \rho_x(h) := hx$ induce an isomorphism of Banach spaces
\begin{equation*}
  X_{\mathbb{U}-an} \xrightarrow{\;\cong\;} C^{an}(\mathbb{U},X)^U \ ,
\end{equation*}
where the right hand side is the subspace of $U$-fixed vectors for the continuous action $(h,f) \longmapsto (hf)(h') := h(f(h^{-1}h'))$. The subspace $X_{\mathbb{U}-an}$ of $X$ is $U$-invariant. But the inclusion map $X_{\mathbb{U}-an} \xrightarrow{\subseteq} X$ only is continuous in general (and not a topological inclusion). Furthermore, $X_{\mathbb{U}-an}$ is a locally $\Qp$-analytic representation of $U$ (loc.\ cit.\ the paragraph after Def.\ 3.3.1, Cor.\ 3.3.6, and Cor.\ 3.6.13). In particular, we have the derived action of the Lie algebra $\Lie(U) = \Lie(H)$ on $X_{\mathbb{U}-an}$. We also remark that, if $X_{\mathbb{U}-an} = X$ as vector spaces, then this is, in fact, an identity of Banach spaces (loc.\ cit.\ Thm.\ 3.6.3).

\begin{lemma}\label{ant}
For any analytic open subgroup $U$ of $H$ we have:
\begin{itemize}
  \item[i.]  $(X \widehat{\otimes}_L Y)^U = X^U \widehat{\otimes}_L Y$;
  \item[ii.] $(X \widehat{\otimes}_L Y)_{\mathbb{U}-an} = X_{\mathbb{U}-an} \, \widehat{\otimes}_{L}\, Y$;
  \item[iii.] if $X_{\mathbb{U}-an} = X$ then, with respect to the $U$-action
\begin{align*}
  U \times C^{an}(\mathbb{U},X) & \longrightarrow C^{an}(\mathbb{U},X) \\
  (h,f) & \longmapsto h(f(h^{-1}.)) \ ,
\end{align*}
we have $C^{an}(\mathbb{U},X)_{\mathbb{U}-an} = C^{an}(\mathbb{U},X)$.
\end{itemize}
\end{lemma}
\begin{proof}
i. This follows by considering a Banach base of $Y$ (cf.\ \cite{NFA} Prop.\ 10.1).

ii. Using i. we compute
\begin{align*}
  (X \widehat{\otimes}_L Y)_{\mathbb{U}-an} & = C^{an}(\mathbb{U},X \widehat{\otimes}_L Y)^U = (\mathcal{O}(\mathbb{U}) \widehat{\otimes}_{\Qp} X \widehat{\otimes}_L Y)^U    \\
   & = (\mathcal{O}(\mathbb{U}) \widehat{\otimes}_{\Qp} X)^U \widehat{\otimes}_L Y = C^{an}(\mathbb{U},X)^U \widehat{\otimes}_L Y  \\
   & = X_{\mathbb{U}-an} \widehat{\otimes}_L Y \ .
\end{align*}

iii. This is a consequence of \cite{Eme} Prop.\ 3.3.4 and Lemma 3.6.4.
\end{proof}

Next we suppose given an $L$-Banach algebra $D$ with a continuous action of $H$ together with an $H$-invariant Banach subalgebra $C$. Further, we let $Z$ be a finitely generated projective $C$-module (with its canonical Banach module structure) which carries a continuous semilinear $H$-action. Then $D \otimes_C Z$ is a Banach module over $D$ with a semilinear continuous diagonal $H$-action. The dual module $Z^* := \Hom_C(Z,C)$ is finitely generated projective as well. The natural semilinear $H$-action on $Z^*$ is continuous; this follows from Remark \ref{semilinear-cont} and an analog of Remark \ref{simple-conv} by the same arguments as those before Remark \ref{summand-of-free}. Corresponding properties hold for $Z_D^* := \Hom_D(D \otimes_C Z,D)$.

\begin{lemma}\label{Uan-dual}
If $Z_{\mathbb{U}-an} = Z$ then also $(Z^*)_{\mathbb{U}-an} = Z^*$.
\end{lemma}
\begin{proof}
The assumption $C_{\mathbb{U}-an} = C$ implies, by Lemma \ref{ant}.iii, that, with respect to the $U$-action
\begin{align*}
  U \times C^{an}(\mathbb{U},C) & \longrightarrow C^{an}(\mathbb{U},C) \\
  (h,f) & \longmapsto h(f(h^{-1}.)) \ ,
\end{align*}
we have $C^{an}(\mathbb{U},C)_{\mathbb{U}-an} = C^{an}(\mathbb{U},C)$. Let $z \in Z$ and $\alpha \in Z^*$. Applying this to the map $f(h) := \alpha(hz)$, which by assumption lies in $C^{an}(\mathbb{U},C)$, we obtain that $[h \longmapsto (h\alpha)(z)] \in C^{an}(\mathbb{U},C)$ for any $z \in Z$. Let now $z_1, \ldots, z_r$ be $C$-module generators of $Z$. Evaluation at the $z_i$ gives a topological inclusion of Banach modules $Z^* \hookrightarrow C^r$. What we have shown is that the continuous orbit map $\rho_\alpha$ composed with this inclusion lies in $C^{an}(\mathbb{U},C^r)$. It then follows from \cite{Eme} Prop.\ 2.1.23 that $\rho_\alpha$ must lie in $C^{an}(\mathbb{U},Z^*)$. Hence $\alpha \in (Z^*)_{\mathbb{U}-an}$.
\end{proof}

\begin{lemma}\label{Uan-projective}
Suppose that there is an analytic open subgroup $U \subseteq H$ such that $C_{\mathbb{U}-an} = C$ and $Z_{\mathbb{U}-an} = Z$; then the natural map $D_{\mathbb{U}-an} \otimes_C Z \xrightarrow{\cong} (D \otimes_C Z)_{\mathbb{U}-an}$ is an isomorphism of Banach modules over $C$.
\end{lemma}
\begin{proof}
It follows from \cite{Fea} 3.3.14 or \cite{Eme} Prop.\ 3.3.12 that $D_{\mathbb{U}-an}$ is a Banach module over $C_{\mathbb{U}-an} = C$ and that the obvious map $D_{\mathbb{U}-an} \otimes_C Z \longrightarrow D \otimes_C Z$ factorizes continuously through $(D \otimes_C Z)_{\mathbb{U}-an}$. Since $Z$ is projective the latter map is injective.

\textit{Step 1:} We show that $(\id \otimes \alpha)((D \otimes_C Z)_{\mathbb{U}-an}) \subseteq D_{\mathbb{U}-an}$ for any $\alpha \in Z^*$. The map
\begin{align*}
  Z^* & \longrightarrow \Hom_D(D \otimes_C Z,D) = Z_D^*  \\
  \alpha & \longmapsto \id \otimes \alpha
\end{align*}
is a continuous map of $H$-representations on Banach spaces. Since $(Z^*)_{\mathbb{U}-an} = Z^*$ we obtain that $\id \otimes \alpha \in \Hom_D(D \otimes_C Z,D)_{\mathbb{U}-an}$. By the same references as above the obvious $H$-equivariant pairing
\begin{equation*}
  \Hom_D(D \otimes_C Z,D) \otimes_D (D \otimes_C Z) \longrightarrow D
\end{equation*}
induces a corresponding pairing between the spaces $(.)_{\mathbb{U}-an}$.

\textit{Step 2:} We have that, for any (abstract) $C$-module $S$ and any nonzero element $u \in S \otimes_C Z$, there exists an $\alpha \in Z^*$ such that $(\id \otimes \alpha)(u) \neq 0$ in $S$. This immediately reduces to the case of a finitely generated free module $Z$, in which it is obvious.

We now apply Step 2 with $S := D/D_{\mathbb{U}-an}$. Because of Step 1 we must have $(D \otimes_C Z)_{\mathbb{U}-an} \subseteq D_{\mathbb{U}-an} \otimes_C Z$. Hence the asserted map is a continuous bijection. The open mapping theorem then implies that it even is a topological isomorphism.
\end{proof}

Consider any infinitely ramified $p$-adic Lie extension $L_\infty$ of $L$, and let $\Gamma := \Gal(L_\infty/L)$. The completion $\hat{L}_\infty$ of $L_\infty$ is an $L$-Banach representation of $\Gamma$. Following \cite{STLAV}, we denote by $\hat{L}_\infty^{\la}$ the set of elements of $\hat{L}_\infty$ that are locally analytic for the action of $\Gamma$. The structure of $\hat{L}_\infty^{\la}$ is studied in \cite{STLAV}, where it is proved that, informally: $\hat{L}_\infty^{\la}$ looks like a space of power series in $\dim(\Gamma) - 1$ variables. In particular, there is (\cite{STLAV} Prop.\ 6.3) a nonzero element $D \in \Cp \otimes_{\Qp} \Lie(\Gamma)$ such that $D=0$ on $\hat{L}_\infty^{\la}$. This element is the pullback of the Sen operator attached to a certain representation of $\Gamma$.

Now let $L_\infty$ be the Lubin-Tate extension of $L$ such that $\Gal(\overline{L}/L_\infty) = \ker(\chi_{LT})$. Then $\chi_{LT} : \Gamma_L^{LT} := \Gal(L_\infty/L) \xrightarrow{\cong} \Gamma_L = o_L^\times$, and $d\chi_{LT} : \Lie(\Gamma_L^{LT}) \xrightarrow{\cong} \Lie(\Gamma_L) = L$. With $\Sigma := \Gal(L/\Qp)$ we have the $L$-linear isomorphism
\begin{align}\label{diag:splitting}
  L \otimes_{\Qp} \Lie(\Gamma_L^{LT}) & \xrightarrow{\;\cong\;} \bigoplus_{\sigma \in \Sigma} L \\
  a \otimes \mathfrak{x} & \longmapsto (a\sigma(d\chi_{LT}(\mathfrak{x})))_\sigma \ .  \nonumber
\end{align}
For any $\sigma \in \Sigma$ we let $\nabla_\sigma$ the element in the left hand side which maps to the tuple with entry $1$, resp.\ $0$, at $\sigma$, resp.\ at all $\sigma' \neq \sigma$ (it should be considered as the ``derivative in the direction of $\sigma$''). Then
\begin{equation*}
  1 \otimes \mathfrak{x} = \sum_{\sigma \in \Sigma} \sigma(d\chi_{LT}(\mathfrak{x})) \cdot \nabla_\sigma  \qquad\text{for any $\mathfrak{x} \in \Lie(\Gamma_L^{LT})$} \ .
\end{equation*}
Therefore, for any locally $\Qp$-representation $V$ of $\Gamma_L^{LT}$, the Taylor expansion of its orbit maps (\cite{STla} p.\ 452) has the form
\begin{align*}
  \gamma v & = \sum_{n=0}^\infty \frac{1}{n!} \log_{\Gamma_L^{LT}}(\gamma)^n v = \sum_{n=0}^\infty \frac{1}{n!} (\sum_{\sigma \in \Sigma} \sigma(d\chi_{LT}(\log_{\Gamma_L^{LT}}(\gamma))) \cdot \nabla_\sigma)^n v  \\
     & = \sum_{n=0}^\infty \frac{1}{n!} (\sum_{\sigma \in \Sigma} \sigma(\log(\chi_{LT}(\gamma))) \cdot \nabla_\sigma)^n v \\
   & = v + \sum_{\sigma \in \Sigma} \log (\sigma(\chi_{LT}(\gamma))) \cdot \nabla_\sigma(x) + \ \text{higher terms}
\end{align*}
for any $v \in V$, and any small enough $\gamma \in \Gamma_L^{LT}$.

Let $L^{max}$ be the compositum of $L_\infty$ and $L^{cyc}$.

\begin{lemma}\label{essdispro}
If $L_\infty \cap L^{cyc}$ is a finite extension of $L$, then $\nabla_{\id} + \nabla_{cyc} = 0$ on $(\widehat{L^{max}})^{\la}$.
\end{lemma}

\begin{proof}
As a consequence of our assumption we have
\begin{equation*}
  \Lie(\Gal(L^{max}/L)) = \Lie(\Gamma_L^{LT}) \oplus \Lie(\Gamma^{cyc}_L) \ .
\end{equation*}
\cite{STLAV} Prop.\ 6.3, there exists a nonzero element $D \in \Cp \otimes_{\Qp} \Lie(\Gal(L^{max}/L))$ such that $D=0$ on $(\widehat{L^{max}})^{\la}$. We claim that $D$ is a scalar multiple of $\nabla_{\id} + \nabla_{cyc}$. For each element $\sigma \in \Sigma \setminus \{\id\}$, the field $(\hat{L}_\infty)^{\la}$ contains the variable $x_\sigma$ such that $\gamma x_\sigma = x_\sigma + \log(\sigma(\chi_{LT}(\gamma)))$ for any $\gamma \in \Gamma_L^{LT}$ (\cite{STLAV} \S4.2). It follows that $\nabla_\sigma (x_\sigma) = 1$ and $\nabla_{\sigma'} (x_\sigma) = 0$ if $\sigma' \neq \sigma$. By our assumption we also have $\nabla_{cyc}(x_\sigma) = 0$.

The character $\tau = \chi_{cyc} \chi_{LT}^{-1}$ has a Hodge-Tate weight equal to zero, so that there exists $z \in \Cp^\times$ such that $g(z) = z \cdot \tau(g)$ for any $g \in G_L$. It is then obvious that $z \in (\widehat{L^{max}})^{\la}$. Secondly this implies that $\mathfrak{x}z = d\tau(\mathfrak{x}) \cdot z = d\chi_{cyc}(\mathfrak{x}) \cdot z - d\chi_{LT}(\mathfrak{x}) \cdot z$, for any $\mathfrak{x} \in \Lie(\Gal(L^{max}/L))$, and hence  $\nabla_{\id} (z) = -z = - \nabla_{cyc} (z)$.

Write $D = \sum_{\sigma \in \Sigma} \delta_\sigma \cdot \nabla_\sigma + \delta_{cyc} \cdot\nabla_{cyc}$. Applying $D$ to $x_\sigma$ with $\sigma \in \Sigma \setminus \{\id\}$, we find that $\delta_\sigma = 0$. Applying $D$ to $z$, we find that $\delta_{\id} = \delta_{cyc}$. This implies the claim and hence our assertion.
\end{proof}

\begin{proof}[Proof of Proposition \ref{nabsenul}]
Let $H^{max} := \Gal(\overline{\QQ}_p/L^{max})$. Using Remark \ref{AST} we have
\begin{equation*}
  P^{H^{max}} = (A_{\Cp} \otimes_B P_0)^{H^{max}} = (A_{\Cp})^{H^{max}} \otimes_B P_0 = (A \co_L \widehat{L^{max}}) \otimes_B P_0 \ .
\end{equation*}
We have to consider the twisted action of the $p$-adic Lie group $\Gal(L^{max}/L)$ on the above terms. On $Q_n \subseteq P^{H_L^{cyc}} \subseteq P^{H^{max}}$ the group $\Gal(L^{max}/L)$ acts through its quotient $\Gamma_L^{cyc}$. By Cor.\ \ref{qnlocan} the $\Gamma_L^{cyc}$-action on $Q_n$ is locally $\Qp$-analytic. Moreover, by assumption the $\Gal(L^{max}/L)$-actions (through the map $\tau^{-1}$) on $B$ and $P_0$ are locally $\Qp$-analytic. We therefore find, by \cite{Eme} Cor.\ 3.6.13, an analytic open subgroup $U \subseteq \Gal(L^{max}/L)$ such that $Q_n = (Q_n)_{\mathbb{U}-an}$, $B = B_{\mathbb{U}-an}$, and  $P_0 = (P_0)_{\mathbb{U}-an}$. Using Lemma \ref{ant}.ii and Lemma \ref{Uan-projective} we then obtain that
\begin{align*}
  Q_n \subseteq (P^{H^{max}})_{\mathbb{U}-an} & = ((A \co_L \widehat{L^{max}}) \otimes_B P_0)_{\mathbb{U}-an} = (A \co_L \widehat{L^{max}})_{\mathbb{U}-an} \otimes_B P_0   \\
  & = (A \co_L (\widehat{L^{max}})_{\mathbb{U}-an}) \otimes_B P_0 \ .
\end{align*}
Recall that the twisted $G_L$-action on $P_0$ is given by ${^{g*}x} = \tau(g^{-1})(x)$. Since, by our assumption, the  $o_L^\times$-action on $P_0$ is locally $L$-analytic this twisted $\Gal(L^{max}/L)$-action on $P_0$ is locally $\Qp$-analytic with the corresponding derived action of $L \otimes_{\Qp} \Lie(\Gal(L^{max}/L))$ factorizing through the map
\begin{equation}\label{diag:L-ana}
  \xymatrix{
    L \otimes_{\Qp} \Lie(\Gal(L^{max}/L)) \ar[d]_{\id \otimes \pr \oplus \id \otimes \pr} \ar[r]^-{\id \otimes -d\tau} & L \otimes_{\Qp} L \ar[rr]^-{a \otimes b \mapsto ab} & & L = \Lie(o_L^\times)  \\
    L \otimes_{\Qp} \Lie(\Gamma_L^{LT}) \oplus L \otimes_{\Qp} \Lie(\Gamma_L^{cyc}) \ar[d]_{\eqref{diag:splitting}\, \oplus \, \id \otimes d\chi_{cyc}}  &  &  \\
    (\sum_{\sigma \in \Sigma} L) \oplus L \ar[uurrr]_-{\qquad\ (\sum_\sigma a_\sigma) + c \;\mapsto\; a_{\id} - c} &  &    }
\end{equation}

We now distinguish two cases, according to whether there is a finite extension $M$ of $L$ such that $L^{cyc} \subseteq M \cdot L_\infty$ or whether there is a finite extension $M$ of $L$ such that $L^{cyc} \cap L_\infty = M$. Since $\Gamma^{cyc}_{L_n}$ is an open subgroup of $\Zp^\times$, we always are in one of these two cases.

\textit{First case:} By \cite{Ser} III.A4 Prop.\ 4 the character $\chi_{LT}$ corresponds, via the reciprocity isomorphism of local class field theory, to the projection map $o_L^\times \times \pi_L^{\widehat{\ZZ}} \xrightarrow{\pr} o_L^\times$. One deduces from this that, in the present situation, $\chi_{cyc}$ differs from $N_{L/\Qp} \circ \chi_{LT}$ by a character of finite order. It follows that $\tau = \chi_{cyc} \chi_{LT}^{-1}$ is equal to a finite order character times $\prod_{\sigma \in \Sigma, \sigma \neq \id} \sigma \circ \chi_{LT}$. We also have $\Lie(\Gal(L^{max}/L)) = \Lie(\Gamma_L^{LT})$ in this case. By feeding this information into the above commutative diagram \eqref{diag:L-ana} one easily deduces that $\nabla_{\id}$ acts as zero on $P_0$. On the other hand, by \cite{STLAV} Cor.\ 4.3, we have $\nabla_{\id} = 0$ on $(\widehat{L^{max}})^{\la}$. This implies that $\nabla_{\id} = 0$ on $(A \co_L (\widehat{L^{max}})_{\mathbb{U}-an}) \otimes_B P_0$. The natural map $\Lie(\Gal(L^{max}/L)) = \Lie(\Gamma_L^{LT}) \to \Lie(\Gamma_L^{cyc})$ sends $\nabla_{\id}$ to $\nabla_{Sen}$. This implies the result.

\textit{Second case:} This time the perpendicular maps in the diagram \eqref{diag:L-ana} are bijective. One immediately deduces that $\nabla_{\id} + \nabla_{cyc} = 0$ acts as zero on $P_0$. On the other hand we know from Lemma \ref{essdispro} that $\nabla_{\id} + \nabla_{cyc} = 0$ on $(\widehat{L^{max}})^{\la}$. Hence $\nabla_{\id} + \nabla_{cyc} = 0$ on $(A \co_L (\widehat{L^{max}})_{\mathbb{U}-an}) \otimes_B P_0$. The natural map $\Lie(\Gal(L^{max} / L)) \to \Lie(\Gamma_L^{cyc})$ sends $\nabla_{\id}$ to $0$ and $\nabla_{cyc}$ to $\nabla_{Sen}$. This again implies the result.
\end{proof}

We now can state our main technical result.

\begin{theorem}\label{Sen-vanishes}
If the $o_L^\times$-actions on $B$ and $P_0$ are locally $L$-analytic then $P^{G_L,*}$ is a finitely generated projective $A$-module of the same rank as $P$ and $A_{\Cp} \otimes_A P^{G_L,*} = P$.
\end{theorem}

\begin{proof}
This follows from Lemma \ref{Sen-zero} and Prop.\ \ref{nabsenul}.
\end{proof}

\begin{remark}\label{untwisting}
All of the above remains valid if we replace $\tau$ by $\tau^{-1}$.
\end{remark}

\subsection{The equivalence between $L$-analytic $(\varphi_L,\Gamma_L)$-modules over $\mathfrak{B}$ and $\mathfrak{X}$}
\label{sec:TS}

We now explore the diagram
\begin{equation*}
     \xymatrix{
     & \Cp \times_{\Qp} \mathfrak{B} \ar[ld] \ar[r]^{\kappa}_{\cong} & \Cp \times_L \mathfrak{X} \ar[rd] & \\
    \mathfrak{B}_{/L} & & & \mathfrak{X}    }
\end{equation*}
in order to transfer modules between the two sides. As in section \ref{sec:LT} we treat $\kappa$ as an identification. Recall also that $\kappa$ is equivariant for the $o_L \setminus \{0\}$-actions on both sides.

As a consequence of Remark \ref{Galois-fixed} we have (cf.\ Lemma \ref{twist-Xn})
\begin{equation*}
  \mathscr{R}_L(\mathfrak{B}) = \mathscr{R}_{\Cp}(\mathfrak{B})^{G_L} \qquad\text{and}\qquad \mathscr{R}_L(\mathfrak{X}) = \mathscr{R}_{\Cp}(\mathfrak{B})^{G_L,*} \ ,
\end{equation*}
where the twisted Galois action is given by $(\sigma,f) \longmapsto {^{\sigma *}f} := ({^\sigma f})([\tau(\sigma^{-1}](.))$; it commutes with the $o_L \setminus \{0\}$-action. By extending the concept of the twisted Galois action we will construct a functor from $\Mod_L(\mathscr{R}_K(\mathfrak{B}))$ to $\Mod_L(\mathscr{R}_K(\mathfrak{X}))$.

Let $M$ be a $(\varphi_L,\Gamma_L)$-module over
$\mathscr{R}_L(\mathfrak{B})$. Then $M_{\Cp} :=
\mathscr{R}_{\Cp}(\mathfrak{B}) \otimes_{\mathscr{R}_L(\mathfrak{B})} M$
is a $(\varphi_L,\Gamma_L)$-module over $\mathscr{R}_{\Cp}(\mathfrak{B})$ by (the analog for $\mathfrak{B}$ of) Lemma \ref{scalar-ext}.ii. One easily checks that the twisted $G_L$-action on $M_{\Cp}$ given by
\begin{equation*}
    (\sigma, f \otimes m) \longmapsto \sigma * (f \otimes m) := {^{\sigma *}f} \otimes \tau(\sigma^{-1})(m) \ ,
\end{equation*}
for $\sigma \in G_L$, $f \in \mathscr{R}_{\Cp}(\mathfrak{B})$, and $m \in M$, is well defined. Observe that this twisted Galois action and the action of $o_L\backslash \{0\}$ on
$M_{\Cp}$ commute with each other. Hence the fixed elements
\begin{equation*}
      M_\mathfrak{X} := (\mathscr{R}_{\Cp}(\mathfrak{B}) \otimes_{\mathscr{R}_L(\mathfrak{B})} M)^{G_L,*}
\end{equation*}
form an $\mathscr{R}_L(\mathfrak{X})$-module which is invariant under the $o_L \setminus \{0\}$-action.

In order to analyze $M_\mathfrak{X}$ as an $\mathscr{R}_L(\mathfrak{X})$-module we use the result from the previous section \ref{sec:key}.

\begin{theorem}\phantomsection\label{descent-to-X}
\begin{itemize}
  \item[i.] $M_\mathfrak{X}$ is a finitely generated projective $\mathscr{R}_L(\mathfrak{X})$-module.
  \item[ii.] If $M$ is $L$-analytic then the rank of $M_\mathfrak{X}$ is equal to the rank of $M$ and
\begin{equation*}
    \mathscr{R}_{\Cp}(\mathfrak{X}) \otimes_{\mathscr{R}_L(\mathfrak{X})} M_\mathfrak{X} = \mathscr{R}_{\Cp}(\mathfrak{B}) \otimes_{\mathscr{R}_L(\mathfrak{B})} M \ .
\end{equation*}
\end{itemize}

\end{theorem}
\begin{proof}
By the analog for $\mathfrak{B}$ of Prop.\ \ref{descent} we find an $n \geq 1$ and an $\mathcal{O}_L(\mathfrak{X} \setminus \mathfrak{X}_n)$-submodule $M_n \subseteq M$ such that
\begin{itemize}
  \item[--] $M_n$ is finitely generated projective over $\mathcal{O}_L(\mathfrak{B} \setminus \mathfrak{B}_n)$, and $\mathscr{R}_L(\mathfrak{B}) \otimes_{\mathcal{O}_L(\mathfrak{B} \setminus \mathfrak{B}_n)} M_n = M$,
  \item[--] $M_n$ is $o_L^\times$-invariant, and the induced $o_L^\times$-action on $M_n$ is continuous,
  \item[--] $\varphi_M$ restricts to a continuous homomorphism
\begin{equation*}
  \varphi_{M_n} : M_n \longrightarrow \mathcal{O}_L(\mathfrak{B} \setminus \mathfrak{B}_{n+1}) \otimes_{\mathcal{O}_L(\mathfrak{B} \setminus \mathfrak{B}_n)} M_n
\end{equation*}
   such that the induced linear map
\begin{equation*}
  \mathcal{O}_L(\mathfrak{B} \setminus \mathfrak{B}_{n+1}) \otimes_{\mathcal{O}_L(\mathfrak{B} \setminus \mathfrak{B}_n),\varphi_L} M_n \xrightarrow{\;\cong\;}
  \mathcal{O}_L(\mathfrak{B} \setminus \mathfrak{B}_{n+1}) \otimes_{\mathcal{O}_L(\mathfrak{B} \setminus \mathfrak{B}_n)} M_n
\end{equation*}
      is an isomorphism.
\end{itemize}
For any $m \geq n$ we define $M_m := \mathcal{O}_L(\mathfrak{B} \setminus \mathfrak{B}_m) \otimes_{\mathcal{O}_L(\mathfrak{B} \setminus \mathfrak{B}_n)} M_n$. The above three properties hold correspondingly for any $m \geq n$. Each $M_{m,\Cp} :=  \mathcal{O}_{\Cp}(\mathfrak{B} \setminus \mathfrak{B}_m) \otimes_{\mathcal{O}_L(\mathfrak{B} \setminus \mathfrak{B}_m)} M_m$ carries an obvious twisted $G_L$-action such that the identity $M_{\Cp} = \mathscr{R}_{\Cp}(\mathfrak{B}) \otimes_{\mathcal{O}_{\Cp}(\mathfrak{B} \setminus \mathfrak{B}_m)} M_{m,\Cp}$ is compatible with the twisted $G_L$-actions on both sides. The fixed elements $M_{m,\Cp}^{G_L,*}$ form an $\mathcal{O}_L(\mathfrak{X} \setminus \mathfrak{X}_m)$-module (cf.\ Lemma \ref{twist-Xn}), and $M_{\Cp}^{G_L,*} = (\bigcup_{m \geq n}  M_{m,\Cp})^{G_L,*} = \bigcup_{m \geq n} M_{m,\Cp}^{G_L,*}$. We claim that it suffices to show:
\begin{itemize}
  \item[iii.] $M_{m,\Cp}^{G_L,*}$, for any $m \geq n$, is a finitely generated projective $\mathcal{O}_L(\mathfrak{X} \setminus \mathfrak{X}_m)$-module.
  \item[iv.] For any $m \geq n$, the natural map
\begin{equation*}
  \mathcal{O}_L(\mathfrak{X} \setminus \mathfrak{X}_{m+1}) \otimes_{\mathcal{O}_L(\mathfrak{X} \setminus \mathfrak{X}_m)} M_{m,\Cp}^{G_L,*} \xrightarrow{\;\cong\;} M_{m+1,\Cp}^{G_L,*}
\end{equation*}
  is an isomorphism.
  \item[v.] If $M$ is $L$-analytic then $\mathcal{O}_{\Cp}(\mathfrak{X} \setminus \mathfrak{X}_n) \otimes_{\mathcal{O}_L(\mathfrak{X} \setminus \mathfrak{X}_n)} M_{n,\Cp}^{G_L,*} = M_{n,\Cp}$.
\end{itemize}
Suppose that iii. - v. hold true. By iv. we have $M_{m,\Cp}^{G_L,*} = \mathcal{O}_L(\mathfrak{X} \setminus \mathfrak{X}_m) \otimes_{\mathcal{O}_L(\mathfrak{X} \setminus \mathfrak{X}_n)} M_{n,\Cp}^{G_L,*}$ for any $m \geq n$. Hence
\begin{align*}
  M_{\Cp}^{G_L,*} & = \bigcup_{m \geq n} \big( \mathcal{O}_L(\mathfrak{X} \setminus \mathfrak{X}_m) \otimes_{\mathcal{O}_L(\mathfrak{X} \setminus \mathfrak{X}_n)} M_{n,\Cp}^{G_L,*} \big) \\
   & = \big( \bigcup_{m \geq n} \mathcal{O}_L(\mathfrak{X} \setminus \mathfrak{X}_m) \big) \otimes_{\mathcal{O}_L(\mathfrak{X} \setminus \mathfrak{X}_n)} M_{n,\Cp}^{G_L,*} \\
   & = \mathscr{R}_L(\mathfrak{X}) \otimes_{\mathcal{O}_L(\mathfrak{X} \setminus \mathfrak{X}_n)} M_{n,\Cp}^{G_L,*} \ ,
\end{align*}
and iii. therefore implies i. Assuming that $M$ is $L$-analytic we deduce from v. that
\begin{align*}
   \mathscr{R}_{\Cp}(\mathfrak{X}) \otimes_{ \mathscr{R}_L(\mathfrak{X})} M_{\Cp}^{G_L,*} & = \mathscr{R}_{\Cp}(\mathfrak{X}) \otimes_{\mathcal{O}_L(\mathfrak{X} \setminus \mathfrak{X}_n)} M_{n,\Cp}^{G_L,*} \\
    & = \mathscr{R}_{\Cp}(\mathfrak{X}) \otimes_{\mathcal{O}_{\Cp}(\mathfrak{X} \setminus \mathfrak{X}_n)} M_{n,\Cp} \\
    & = \mathscr{R}_{\Cp}(\mathfrak{B}) \otimes_{\mathcal{O}_{\Cp}(\mathfrak{B} \setminus \mathfrak{B}_n)} M_{n,\Cp} \\
    & = M_{\Cp} \ ,
\end{align*}
which is the second part of ii. Since all rings involved are integral domains the first part of ii. is a consequence of the second part by \cite{B-AC} II.5.3 Prop.\ 4.

In order to establish iii. - v. we fix an $m \geq n$ and abbreviate $\mathfrak{Y} := \mathfrak{X} \setminus \mathfrak{X}_m$ and $\mathfrak{A} := \mathfrak{B} \setminus \mathfrak{B}_m$. We recall from Prop.\ \ref{quasi-Stein} and its proof that we have increasing sequences $\mathfrak{Y}_1 \subset \ldots \subset \mathfrak{Y}_j \subset \ldots \subset \mathfrak{Y}$ and $\mathfrak{A}_1 \subset \ldots \subset \mathfrak{A}_j \subset \ldots \subset \mathfrak{A}$ of open affinoid subdomains (over $L$) which form admissible open coverings of $\mathfrak{Y}$ and $\mathfrak{A}$, respectively, and such that:
\begin{itemize}
  \item[a.] All restriction maps $\mathcal{O}_K(\mathfrak{Y}_{j+1}) \longrightarrow \mathcal{O}_K(\mathfrak{Y}_j)$ and $\mathcal{O}_K(\mathfrak{A}_{j+1}) \longrightarrow \mathcal{O}_K(\mathfrak{A}_j)$ are compactoid (proof of Prop.\ \ref{compactoid}.i) and have dense image. In particular, these coverings exhibit $\mathfrak{Y}$ and $\mathfrak{A}$ as Stein spaces.
  \item[b.] Each $\mathfrak{A}_j$ is a subannulus of $\mathfrak{A}$.
  \item[c.] Every $\mathfrak{Y}_j$ and every $\mathfrak{A}_j$ is $o_L^\times$-invariant, and the induced $o_L^\times$-actions on $\mathcal{O}_L(\mathfrak{Y}_j)$ and $\mathcal{O}_L(\mathfrak{A}_j)$ satisfy the condition \eqref{f:condition-locan} and are locally $L$-analytic (Prop.\ \ref{annuli-locan} and discussion before Def.\ \ref{def:L-analytic}).
  \item[d.] The LT-isomorphism $\kappa : \mathfrak{X}_{/\Cp} \xrightarrow{\cong} \mathfrak{B}_{/\Cp}$ restricts, for any $j \geq 1$, to an isomorphism $\mathfrak{Y}_{j/\Cp} \xrightarrow{\cong} \mathfrak{A}_{j/\Cp}$. With respect to the resulting identifications
\begin{equation*}
     \Cp\, \widehat{\otimes}_L\, \mathcal{O}_L(\mathfrak{Y}_j) = \mathcal{O}_{\Cp}(\mathfrak{Y}_j) = \mathcal{O}_{\Cp}(\mathfrak{A}_j) = \Cp\, \widehat{\otimes}_L\, \mathcal{O}_{\Cp}(\mathfrak{A}_j)
\end{equation*}
the $G_L$-action through the coefficients $\Cp$ on the left hand side corresponds to the twisted $G_L$-action on the right hand side; in particular, $\mathcal{O}_L(\mathfrak{Y}_j) = \mathcal{O}_{\Cp}(\mathfrak{A}_j)^{G_L,*}$.
\end{itemize}
As explained, for example, in \cite{ST} \S3 to give a coherent sheaf $\mathcal{S}$ on a quasi-Stein space $\mathfrak{Z}$ with defining admissible affinoid covering $\mathfrak{Z}_1 \subset \ldots \subset \mathfrak{Z}_j \subset \ldots \subset \mathfrak{Z}$ is the same as giving the finitely generated $\mathcal{O}(\mathfrak{Z}_j)$-modules $\mathcal{S}(\mathfrak{Z}_j)$ together with isomorphisms
\begin{equation*}
    \mathcal{O}(\mathfrak{Z}_j) \otimes_{\mathcal{O}(\mathfrak{Z}_{j+1})} \mathcal{S}(\mathfrak{Z}_{j+1}) \xrightarrow{\cong} \mathcal{S}(\mathfrak{Z}_j) \ .
\end{equation*}
Then
\begin{equation*}
    \mathcal{S}(\mathfrak{Z}) = \varprojlim \mathcal{S}(\mathfrak{Z}_j) \qquad\text{and}\qquad \mathcal{S}(\mathfrak{Z}_j) = \mathcal{O}(\mathfrak{Z}_j) \otimes_{\mathcal{O}(\mathfrak{Z})} \mathcal{S}(\mathfrak{Z}) \ .
\end{equation*}
Moreover, by Prop.\ \ref{gruson}, the coherent sheaf $\mathcal{S}$ is a vector bundle if and only if $\mathcal{S}(\mathfrak{Z})$ is a finitely generated projective $\mathcal{O}(\mathfrak{Z})$-module.

In our situation therefore the modules $M_m$ and $M_{m,\Cp}$ are the global sections of vector bundles $\mathcal{M}_m$ and $\mathcal{M}_{m,\Cp}$ on $\mathfrak{A}$ and $\mathfrak{A}_{/\Cp}$, respectively. As a consequence of property c. and Prop.\ \ref{locan3} the continuous $o_L^\times$-action on $M_m$ extends semilinearly to compatible locally $\Qp$-analytic $o_L^\times$-actions on the system $(\mathcal{M}_m(\mathfrak{A}_j) = \mathcal{O}_L(\mathfrak{A}_j) \otimes_{\mathcal{O}_L(\mathfrak{A})} M_m)_j$. For later use we note right away that, as explained in the Addendum to the proof of Prop.\ \ref{differentiable}, each $\mathcal{M}_m(\mathfrak{A}_j)$ is $L$-analytic if $M$ is $L$-analytic. As explained before Prop.\ \ref{nabsenul} these give rise to compatible (continuous) twisted $G_L$-actions on the system $(\mathcal{M}_{m/\Cp}(\mathfrak{A}_{j/\Cp}) = \mathcal{O}_{\Cp}(\mathfrak{A}_j) \otimes_{\mathcal{O}_L(\mathfrak{A})} M_m = \mathcal{O}_{\Cp}(\mathfrak{A}_j) \otimes_{\mathcal{O}_{\Cp}(\mathfrak{A})} M_{m,\Cp})_j$. By construction the latter is a system of finitely generated projective modules (all of the same rank, of course). But because of the property b. the rings $\mathcal{O}_{\Cp}(\mathfrak{Y}_j) = \mathcal{O}_{\Cp}(\mathfrak{A}_j)$ are principal ideal domains. Hence we actually have a system of free modules. Each member of this system therefore satisfies the assumptions of the previous section \ref{sec:key}. The second half of Cor.\ \ref{TS2}.iii then tells us that by passing to the fixed elements for the twisted $G_L$-action we obtain a system $(\mathcal{M}_{m/\Cp}(\mathfrak{A}_{j/\Cp})^{G_L,*})_j$ of $\mathcal{O}_L(\mathfrak{Y}_j)$-modules $\mathcal{M}_{m/\Cp}(\mathfrak{A}_{j/\Cp})^{G_L,*}$ which are submodules of finitely generated free modules. The affinoid $\mathfrak{Y_j}$ being one dimensional and smooth the ring $\mathcal{O}_L(\mathfrak{Y}_j)$ is a Dedekind domain, over which submodules of finitely generated free modules are finitely generated projective. Hence $(\mathcal{M}_{m/\Cp}(\mathfrak{A}_{j/\Cp})^{G_L,*})_j$ is a system of finitely generated projective modules. We obviously have
\begin{equation*}
  \varprojlim_j \mathcal{M}_{m/\Cp}(\mathfrak{A}_{j/\Cp})^{G_L,*} = (\varprojlim_j \mathcal{M}_{m/\Cp}(\mathfrak{A}_{j/\Cp}))^{G_L,*} =   \mathcal{M}_{m/\Cp}(\mathfrak{A}_{/\Cp})^{G_L,*} = M_{m/\Cp}^{G_L,*} \ .
\end{equation*}
If we show that the transition maps
\begin{equation}\label{f:transition}
    \mathcal{O}_L(\mathfrak{Y}_j) \otimes_{\mathcal{O}_L(\mathfrak{Y}_{j+1})}                \mathcal{M}_{m/\Cp}(\mathfrak{A}_{j+1/\Cp})^{G_L,*} \longrightarrow \mathcal{M}_{m/\Cp}(\mathfrak{A}_{j/\Cp})^{G_L,*}
\end{equation}
are isomorphisms then the system $(\mathcal{M}_{m/\Cp}(\mathfrak{A}_{j/\Cp})^{G_L,*})_j$ corresponds to a vector bundle $\mathcal{M}_{\mathfrak{Y}}$ on $\mathfrak{Y}$ with global sections $M_{m/\Cp}^{G_L,*}$. This proves iii.

To see that \eqref{f:transition}, for a fixed $j$, is an isomorphism we use Cor.\ \ref{TS2} again which implies the existence of finite extensions $L \subseteq L_1 \subseteq L_2 \subseteq \Cp$ and of a finitely generated projective $(G_L,\ast)$-invariant $\mathcal{O}_{L_1}(\mathfrak{Y}_{j+1})$-submodule $Q$ of $\mathcal{M}_{m/\Cp}(\mathfrak{A}_{j+1/\Cp})$ with the properties a) - c) in Cor.\ \ref{TS2}.ii and such that
\begin{align*}
  \mathcal{M}_{m/\Cp}(\mathfrak{A}_{j+1/\Cp})^{G_L,*} & \subseteq \mathcal{O}_{L_2}(\mathfrak{Y}_{j+1}) \otimes_{\mathcal{O}_{L_1}(\mathfrak{Y}_{j+1})} Q \\ & \subseteq \mathcal{M}_{m/\Cp}(\mathfrak{A}_{j+1/\Cp}) = \mathcal{O}_{\Cp}(\mathfrak{Y}_{j+1}) \otimes_{\mathcal{O}_{L_1}(\mathfrak{Y}_{j+1})} Q
\end{align*}
and hence, in particular, that
\begin{equation*}
  \mathcal{M}_{m/\Cp}(\mathfrak{A}_{j+1/\Cp})^{G_L,*} = (\mathcal{O}_{L_2}(\mathfrak{Y}_{j+1}) \otimes_{\mathcal{O}_{L_1}(\mathfrak{Y}_{j+1})} Q)^{G_L,*} \ .
\end{equation*}
We deduce that
\begin{align}\label{f:transition2}
  \mathcal{M}_{m/\Cp}(\mathfrak{A}_{j/\Cp}) & =  \mathcal{O}_{\Cp}(\mathfrak{Y}_j) \otimes_{\mathcal{O}_{\Cp}(\mathfrak{Y}_{j+1})}      \mathcal{M}_{m/\Cp}(\mathfrak{A}_{j+1/\Cp}) =  \mathcal{O}_{\Cp}(\mathfrak{Y}_j) \otimes_{\mathcal{O}_{L_1}(\mathfrak{Y}_{j+1})} Q \\
   & = \mathcal{O}_{\Cp}(\mathfrak{Y}_j) \otimes_{\mathcal{O}_{L_1}(\mathfrak{Y}_j)} (\mathcal{O}_{L_1}(\mathfrak{Y}_j) \otimes_{\mathcal{O}_{L_1}(\mathfrak{Y}_{j+1})} Q) \ .    \nonumber
\end{align}

We claim that for $P := \mathcal{M}_{m/\Cp}(\mathfrak{A}_{j/\Cp})$ we may take as input for Cor.\ \ref{TS2}.iii the submodule $\mathcal{O}_{L_1}(\mathfrak{Y}_j) \otimes_{\mathcal{O}_{L_1}(\mathfrak{Y}_{j+1})} Q$. The properties a) and c) are obvious. For b) we have to check that the inclusion
\begin{equation*}
  \mathcal{O}_{\widehat{L^{cyc}}}(\mathfrak{Y}_j) \otimes_{\mathcal{O}_{L_1}(\mathfrak{Y}_j)} (\mathcal{O}_{L_1}(\mathfrak{Y}_j) \otimes_{\mathcal{O}_{L_1}(\mathfrak{Y}_{j+1})} Q) \subseteq \mathcal{M}_{m/\Cp}(\mathfrak{A}_{j/\Cp})^{H_L^{cyc}}
\end{equation*}
is an equality. Since the ring homomorphism $\mathcal{O}_{\widehat{L^{cyc}}}(\mathfrak{Y}_j) \longrightarrow \mathcal{O}_{\Cp}(\mathfrak{Y}_j) = \Cp \, \widehat{\otimes}_{\widehat{L^{cyc}}} \mathcal{O}_{\widehat{L^{cyc}}}(\mathfrak{Y}_j)$ is faithfully flat (\cite{Co1} Lemma 1.1.5(1)) it suffices to see that the base change to $\mathcal{O}_{\Cp}(\mathfrak{Y}_j)$ of the above inclusion is an equality. But both sides become equal to $\mathcal{M}_{m/\Cp}(\mathfrak{A}_{j/\Cp})$, the left hand side by \eqref{f:transition2} and the right hand side by Cor.\ \ref{TS2}.i.

As output from Cor.\ \ref{TS2}.iii we then obtain the first identity in
\begin{align*}
  \mathcal{M}_{m/\Cp}(\mathfrak{A}_{j/\Cp})^{G_L,*} & =  (\mathcal{O}_{L_2}(\mathfrak{Y}_j) \otimes_{\mathcal{O}_{L_1}(\mathfrak{Y}_{j+1})} Q)^{G_L,*} \\
   & = (\mathcal{O}_L(\mathfrak{Y}_j) \otimes_{\mathcal{O}_L(\mathfrak{Y}_{j+1})}   (\mathcal{O}_{L_2}(\mathfrak{Y}_{j+1}) \otimes_{\mathcal{O}_{L_1}(\mathfrak{Y}_{j+1})} Q))^{G_L,*} \ .
\end{align*}
Since the twisted $G_L$-actions are continuous and since the profinite group $G_L$ is topologically finitely generated (cf.\ \cite{NSW} Thm.\ 7.5.14) we, after picking finitely many topological generators $\sigma_1, \ldots, \sigma_r$ of $G_L$, may write
\begin{align*}
  \mathcal{M}_{m/\Cp}(\mathfrak{A}_{j+1/\Cp})^{G_L,*} & = (\mathcal{O}_{L_2}(\mathfrak{Y}_{j+1}) \otimes_{\mathcal{O}_{L_1}(\mathfrak{Y}_{j+1})} Q)^{G_L,*} \\
  & = \ker \Big(\mathcal{O}_{L_2}(\mathfrak{Y}_{j+1}) \otimes_{\mathcal{O}_{L_1}(\mathfrak{Y}_{j+1})} Q \xrightarrow{\prod 1 - \sigma_i} \prod_{i=1}^r \mathcal{O}_{L_2}(\mathfrak{Y}_{j+1}) \otimes_{\mathcal{O}_{L_1}(\mathfrak{Y}_{j+1})} Q \Big).
\end{align*}
The ring homomorphism $\mathcal{O}_L(\mathfrak{Y}_{j+1}) \longrightarrow \mathcal{O}_L(\mathfrak{Y}_j)$ is flat (\cite{BGR} Cor.\ 7.3.2/6). Hence the formation of the kernel above commutes with base extension along this homomorphism. It follows that
\begin{align*}
  \mathcal{O}_L(\mathfrak{Y}_j) \otimes_{\mathcal{O}_L(\mathfrak{Y}_{j+1})}                \mathcal{M}_{m/\Cp}(\mathfrak{A}_{j+1/\Cp})^{G_L,*} & = \mathcal{O}_L(\mathfrak{Y}_j) \otimes_{\mathcal{O}_L(\mathfrak{Y}_{j+1})}   (\mathcal{O}_{L_2}(\mathfrak{Y}_{j+1}) \otimes_{\mathcal{O}_{L_1}(\mathfrak{Y}_{j+1})} Q)^{G_L,*} \\
   & = (\mathcal{O}_L(\mathfrak{Y}_j) \otimes_{\mathcal{O}_L(\mathfrak{Y}_{j+1})}   (\mathcal{O}_{L_2}(\mathfrak{Y}_{j+1}) \otimes_{\mathcal{O}_{L_1}(\mathfrak{Y}_{j+1})} Q))^{G_L,*} \\
   & = \mathcal{M}_{m/\Cp}(\mathfrak{A}_{j/\Cp})^{G_L,*}  \ .
\end{align*}
This settles the isomorphism \eqref{f:transition}.

To check the isomorphism in iv. we temporarily indicate the dependance of our coverings on $m$ by writing $\mathfrak{Y}_j^{(m)}$ and $\mathfrak{A}_j^{(m)}$. Examining again the proof of Prop.\ \ref{quasi-Stein} one easily sees that, for a fixed $m$, these coverings can be chosen so that $\mathfrak{Y}_j^{(m+1)} \subseteq \mathfrak{Y}_j^{(m)}$ and $\mathfrak{A}_j^{(m+1)} \subseteq \mathfrak{A}_j^{(m)}$ for any $j$. But then the exact same argument as above for \eqref{f:transition} shows that the natural maps
\begin{equation*}
    \mathcal{O}_L(\mathfrak{Y}_j^{(m+1)}) \otimes_{\mathcal{O}_L(\mathfrak{Y}_j^{(m)})}                \mathcal{M}_{m/\Cp}(\mathfrak{A}_{j/\Cp}^{(m)})^{G_L,*} \xrightarrow{\;\cong\;} \mathcal{M}_{m+1/\Cp}(\mathfrak{A}_{j/\Cp}^{(m+1)})^{G_L,*} = \mathcal{M}_{m/\Cp}(\mathfrak{A}_{j/\Cp}^{(m+1)})^{G_L,*}
\end{equation*}
are isomorphisms. This means that the restriction to $\mathfrak{X} \setminus \mathfrak{X}_{m+1}$ of the vector bundle $\mathcal{M}_{\mathfrak{X} \setminus \mathfrak{X}_m}$ on $\mathfrak{X} \setminus \mathfrak{X}_m$ coincides with the vector bundle $\mathcal{M}_{\mathfrak{X} \setminus \mathfrak{X}_{m+1}}$. On global sections this is equivalent to the isomorphism in iv.

Now we assume that $M$ is $L$-analytic. Earlier in the proof we had observed already that then each $\mathcal{M}_m(\mathfrak{A}_j)$ is $L$-analytic. In this situation Thm.\ \ref{Sen-vanishes} tells us that
\begin{equation}\label{f:form2}
    \mathcal{O}_{\Cp}(\mathfrak{Y}_j) \otimes_{\mathcal{O}_L(\mathfrak{Y}_j)}                \mathcal{M}_{m/\Cp}(\mathfrak{A}_{j/\Cp})^{G_L,*} = \mathcal{M}_{m/\Cp}(\mathfrak{A}_{j/\Cp}) \ .
\end{equation}
Being finitely generated projective, $\mathcal{O}_{\Cp}(\mathfrak{Y}) \otimes_{\mathcal{O}_L(\mathfrak{Y})} M_\mathfrak{Y}(\mathfrak{Y})$ must be the global sections of a coherent sheaf on $\mathfrak{Y}_{/\Cp} \cong \mathfrak{A}_{/\Cp}$. But one easily deduces from \eqref{f:form2} that this sheaf is $\mathcal{M}_{m/\Cp}$. It follows that
\begin{equation*}
     \mathcal{O}_{\Cp}(\mathfrak{Y}) \otimes_{\mathcal{O}_L(\mathfrak{Y})} M_{m/\Cp}^{G_L,*} = \mathcal{O}_{\Cp}(\mathfrak{Y}) \otimes_{\mathcal{O}_L(\mathfrak{Y})} M_\mathfrak{Y}(\mathfrak{Y}) = \mathcal{M}_{m/\Cp}(\mathfrak{A}_{/\Cp}) = M_{m,\Cp} \ ,
\end{equation*}
which shows v.
\end{proof}

\begin{lemma}\label{prop:N}
$M_\mathfrak{X}$, for any $M$ in $\Mod_L(\mathscr{R}_L(\mathfrak{B}))$, is an object in
$\Mod_L(\mathscr{R}_L(\mathfrak{X}))$.
\end{lemma}
\begin{proof}
It remains to show that:
\begin{itemize}
  \item[a)] The $\Gamma_L$-action on $M_\mathfrak{X}$ is continuous for the canonical topology of $M_\mathfrak{X}$ as an $\mathscr{R}_L(\mathfrak{X})$-module.
  \item[b)] The map
\begin{equation*}
  \varphi_{M_\mathfrak{X}}^{lin} : \mathscr{R}_L(\mathfrak{X})\otimes_{\mathscr{R}_L(\mathfrak{X}), \varphi_L} M_\mathfrak{X} \longrightarrow  M_\mathfrak{X}
\end{equation*}
is an isomorphism.
\end{itemize}
We will use the notations in the proof of the previous Thm.\ \ref{descent-to-X}.

Ad a): By Remark \ref{semilinear-cont} each individual $\gamma \in \Gamma_L$ acts by a continuous automorphism. Since $M_\mathfrak{X}$ is barrelled it therefore suffices, by the Banach-Steinhaus theorem, to check that the orbit maps $\gamma \longrightarrow \gamma m$, for $m \in M_\mathfrak{X}$, are continuous. But $M_\mathfrak{X} = \bigcup_{m \geq n} M_{m,\Cp}^{G_L,*}$, and the inclusion maps $M_{m,\Cp}^{G_L,*} \longrightarrow M_\mathfrak{X}$ are continuous for the canonical topologies (again by Remark \ref{semilinear-cont}). This reduces us to showing the continuity of the orbit maps for each $M_{m,\Cp}^{G_L,*}$. Now we use that, by \eqref{f:transition}, we have
\begin{equation*}
  M_{m,\Cp}^{G_L,*} = \varprojlim_j \mathcal{M}_{m/\Cp}(\mathfrak{A}_{j/\Cp})^{G_L,*}
\end{equation*}
at least algebraically and (cf.\ Fact \ref{coherent}.i)
\begin{equation*}
  \mathcal{M}_{m/\Cp}(\mathfrak{A}_{j/\Cp})^{G_L,*} =  \mathcal{O}_L(\mathfrak{Y}_j) \otimes_{\mathcal{O}_L(\mathfrak{X} \setminus \mathfrak{X}_m)}  M_{m,\Cp}^{G_L,*} \ .
\end{equation*}
In fact, the first identity holds topologically if we equip the left hand side with the canonical topology and the right hand side with the projective limit of the canonical topologies. This is most easily seen by writing $\mathcal{M}_{m/\Cp}(\mathfrak{A}_{j/\Cp})^{G_L,*}$ as a direct summand of a finitely generated free module in which case the corresponding identity is evidently topological. This finally reduces us further to the continuity of the orbit maps of $\Gamma_L$ on the $\mathcal{O}_L(\mathfrak{Y}_j)$-module $ \mathcal{M}_{m/\Cp}(\mathfrak{A}_{j/\Cp})^{G_L,*}$. Since this module is finitely generated projective over the Banach algebra $\mathcal{O}_L(\mathfrak{Y}_j) = \mathcal{O}_{\Cp}(\mathfrak{A}_j)^{G_L,*}$ we have pointed out the required continuity already in the discussion before Prop.\ \ref{nabsenul}.

Ad b): This time we choose the coverings $\{\mathfrak{Y}_j^{(m)}\}_j$ and $\{\mathfrak{A}_j^{(m)}\}_j$ of $\mathfrak{X} \setminus \mathfrak{X}_m$ and $\mathfrak{B} \setminus \mathfrak{B}_m$, respectively, in such a way that $\pi_L^* \mathfrak{Y}_j^{(m+1)} \subseteq \mathfrak{Y}_j^{(m)}$ and $\pi_L^* \mathfrak{A}_j^{(m+1)} \subseteq \mathfrak{A}_j^{(m)}$ for any $j$.

In a first step we show that the map
\begin{align*}
    \mathcal{O}_L(\mathfrak{Y}_j^{(m+1)}) \otimes_{\mathcal{O}_L(\mathfrak{Y}_j^{(m)}),\varphi_L}                \mathcal{M}_{m/\Cp}(\mathfrak{A}_{j/\Cp}^{(m)})^{G_L,*} & \xrightarrow{\;\cong\;} \mathcal{M}_{m+1/\Cp}(\mathfrak{A}_{j/\Cp}^{(m+1)})^{G_L,*}  \\
    f \otimes m & \longmapsto f \varphi_M(m)
\end{align*}
is an isomorphism. As listed at the beginning of the proof of Thm.\ \ref{descent-to-X} we have the isomorphisms
\begin{equation*}
  \mathcal{O}_L(\mathfrak{B} \setminus \mathfrak{B}_{m+1}) \otimes_{\mathcal{O}_L(\mathfrak{B} \setminus \mathfrak{B}_m),\varphi_L} M_m \xrightarrow{\;\cong\;} M_{m+1} \ ,
\end{equation*}
hence
\begin{equation*}
  \mathcal{O}_{\Cp}(\mathfrak{B} \setminus \mathfrak{B}_{m+1}) \otimes_{\mathcal{O}_{\Cp} (\mathfrak{B} \setminus \mathfrak{B}_m),\varphi_L} M_{m/\Cp} \xrightarrow{\;\cong\;} M_{m+1/\Cp} \ ,
\end{equation*}
and therefore
\begin{equation*}
  \mathcal{O}_{\Cp}(\mathfrak{A}_j^{(m+1)}) \otimes_{\mathcal{O}_{\Cp} (\mathfrak{A}_j^{(m)}),\varphi_L} \mathcal{M}_{m/\Cp}(\mathfrak{A}_{j/\Cp}^{(m)}) \xrightarrow{\;\cong\;} \mathcal{M}_{m+1/\Cp}(\mathfrak{A}_{j/\Cp}^{(m+1)})
\end{equation*}
induced by $\varphi_M$. Since the $o_L \setminus \{0\}$-action commutes with the twisted $G_L$-action the latter map restricts to the isomorphism
\begin{equation*}
  \big[\mathcal{O}_{\Cp}(\mathfrak{A}_j^{(m+1)}) \otimes_{\mathcal{O}_{\Cp} (\mathfrak{A}_j^{(m)}),\varphi_L} \mathcal{M}_{m/\Cp}(\mathfrak{A}_{j/\Cp}^{(m)}) \big]^{G_L,*} \xrightarrow{\;\cong\;} \mathcal{M}_{m+1/\Cp}(\mathfrak{A}_{j/\Cp}^{(m+1)})^{G_L,*} \ .
\end{equation*}
This reduces us to showing that the obvious inclusions induce the isomorphism
\begin{multline*}
  \mathcal{O}_L(\mathfrak{Y}_j^{(m+1)}) \otimes_{\mathcal{O}_L(\mathfrak{Y}_j^{(m)}),\varphi_L}                \mathcal{M}_{m/\Cp}(\mathfrak{A}_{j/\Cp}^{(m)})^{G_L,*} \\
  \xrightarrow{\ \cong\ }
  \big[ \mathcal{O}_{\Cp}(\mathfrak{Y}_j^{(m+1)}) \otimes_{\mathcal{O}_{\Cp}(\mathfrak{Y}_j^{(m)}),\varphi_L}                \mathcal{M}_{m/\Cp}(\mathfrak{A}_{j/\Cp}^{(m)}) \big]^{G_L,*} \ .
\end{multline*}
For this we proceed entirely similarly as in the argument for \eqref{f:transition} in the proof of Thm.\ \ref{descent-to-X}. By  Cor.\ \ref{TS2} we find finite extensions $L \subseteq L_1 \subseteq L_2 \subseteq \Cp$ and a finitely generated projective $(G_L,\ast)$-invariant $\mathcal{O}_{L_1}(\mathfrak{Y}_j^{(m)})$-submodule $Q$ of $X := \mathcal{M}_{m/\Cp}(\mathfrak{A}_{j/\Cp}^{(m)})$ with the properties a) - c) and such that
\begin{equation*}
   X^{G_L,*} \subseteq \mathcal{O}_{L_2}(\mathfrak{Y}_j^{(m)}) \otimes_{\mathcal{O}_{L_1}(\mathfrak{Y}_j^{(m)})} Q \subseteq X = \mathcal{O}_{\Cp}(\mathfrak{Y}_j^{(m)}) \otimes_{\mathcal{O}_{L_1}(\mathfrak{Y}_j^{(m)})} Q
\end{equation*}
and hence, in particular, that
\begin{equation*}
  X^{G_L,*} = (\mathcal{O}_{L_2}(\mathfrak{Y}_j^{(m)}) \otimes_{\mathcal{O}_{L_1}(\mathfrak{Y}_j^{(m)})} Q)^{G_L,*} \ .
\end{equation*}
We deduce that
\begin{align*}
   \mathcal{O}_{\Cp}(\mathfrak{Y}_j^{(m+1)}) \otimes_{\mathcal{O}_{\Cp}(\mathfrak{Y}_j^{(m)}),\varphi_L} X & =  \mathcal{O}_{\Cp}(\mathfrak{Y}_j^{(m+1)}) \otimes_{\mathcal{O}_{L_1}(\mathfrak{Y}_j^{(m)}),\varphi_L} Q \\
   & = \mathcal{O}_{\Cp}(\mathfrak{Y}_j^{(m+1)}) \otimes_{\mathcal{O}_{L_1}(\mathfrak{Y}_j^{(m+1)}} (\mathcal{O}_{L_1}(\mathfrak{Y}_j^{(m+1)}) \otimes_{\mathcal{O}_{L_1}(\mathfrak{Y}_j^{(m)}),\varphi_L} Q) \ .
\end{align*}
This shows that for $P := \mathcal{O}_{\Cp}(\mathfrak{Y}_j^{(m+1)}) \otimes_{\mathcal{O}_{\Cp}(\mathfrak{Y}_j^{(m)}),\varphi_L} X$ we may take as input for Cor.\ \ref{TS2}.iii the submodule $\mathcal{O}_{L_1}(\mathfrak{Y}_j^{(m+1)}) \otimes_{\mathcal{O}_{L_1}(\mathfrak{Y}_j^{(m)}),\varphi_L} Q$, and as output we obtain the first identity in
\begin{multline*}
  \big[ \mathcal{O}_{\Cp}(\mathfrak{Y}_j^{(m+1)}) \otimes_{\mathcal{O}_{\Cp}(\mathfrak{Y}_j^{(m)}),\varphi_L} X \big]^{G_L,*}  =  (\mathcal{O}_{L_2}(\mathfrak{Y}_j^{(m+1)}) \otimes_{\mathcal{O}_{L_1}(\mathfrak{Y}_j^{(m)}),\varphi_L} Q)^{G_L,*} \\
    = (\mathcal{O}_L(\mathfrak{Y}_j^{(m+1)}) \otimes_{\mathcal{O}_L(\mathfrak{Y}_j^{(m)}),\varphi_L}   (\mathcal{O}_{L_2}(\mathfrak{Y}_j^{(m)}) \otimes_{\mathcal{O}_{L_1}(\mathfrak{Y}_j^{(m)})} Q))^{G_L,*} \ .
\end{multline*}
Choosing finitely many topological generators $\sigma_1, \ldots, \sigma_r$ of $G_L$ we may write
\begin{align*}
  X^{G_L,*} & = (\mathcal{O}_{L_2}(\mathfrak{Y}_j^{(m)}) \otimes_{\mathcal{O}_{L_1}(\mathfrak{Y}_j^{(m)})} Q)^{G_L,*} \\
  & = \ker \Big( \mathcal{O}_{L_2}(\mathfrak{Y}_j^{(m)}) \otimes_{\mathcal{O}_{L_1}(\mathfrak{Y}_j^{(m)})} Q \xrightarrow{\prod 1 - \sigma_i} \prod_{i=1}^r \mathcal{O}_{L_2}(\mathfrak{Y}_j^{(m)}) \otimes_{\mathcal{O}_{L_1}(\mathfrak{Y}_j^{(m)})} Q \Big).
\end{align*}
As a consequence of \cite{ST} Lemma 3.3 the ring homomorphism  $\varphi_L : \mathcal{O}_L(\mathfrak{Y}_j^{(m)}) \rightarrow \mathcal{O}_L(\mathfrak{Y}_j^{(m+1)})$ is flat. Hence applying the functor $\mathcal{O}_L(\mathfrak{Y}_j^{(m+1)}) \otimes_{\mathcal{O}_L(\mathfrak{Y}_j^{(m)}),\varphi_L}$ commutes with the formation of this kernel, and we obtain
\begin{align*}
  \mathcal{O}_L(\mathfrak{Y}_j^{(m+1)}) \otimes_{\mathcal{O}_L(\mathfrak{Y}_j^{(m)}),\varphi_L} X^{G_L,*} & = (\mathcal{O}_L(\mathfrak{Y}_j^{(m+1)}) \otimes_{\mathcal{O}_L(\mathfrak{Y}_j^{(m)}),\varphi_L} (\mathcal{O}_{L_2}(\mathfrak{Y}_j^{(m)}) \otimes_{\mathcal{O}_{L_1}(\mathfrak{Y}_j^{(m)})} Q))^{G_L,*} \\
  & = \big[ \mathcal{O}_{\Cp}(\mathfrak{Y}_j^{(m+1)}) \otimes_{\mathcal{O}_{\Cp}(\mathfrak{Y}_j^{(m)}),\varphi_L} X \big]^{G_L,*} \ .
\end{align*}
This finishes the first step.

Still fixing $m$ but varying $j$, what we have shown amounts to the statement that $\varphi_M$ induces an isomorphism
\begin{equation*}
  (\pi_L^*)^* \mathcal{M}_{\mathfrak{X} \setminus \mathfrak{X}_m} \xrightarrow{\;\cong\;} \mathcal{M}_{\mathfrak{X} \setminus \mathfrak{X}_{m+1}}
\end{equation*}
of vector bundles on  $\mathfrak{X} \setminus \mathfrak{X}_{m+1}$. We deduce that, on global sections, we have the isomorphism
\begin{equation*}
  \mathcal{O}_L(\mathfrak{X} \setminus \mathfrak{X}_{m+1}) \otimes_{\mathcal{O}_L(\mathfrak{X} \setminus \mathfrak{X}_m),\varphi_L} M_{m,\Cp}^{G_L,*} \xrightarrow{\;\cong\;} M_{m+1,\Cp}^{G_L,*}
\end{equation*}
By passing to the limit with respect to $m$ we obtain the assertion.
\end{proof}

The above construction is entirely symmetric in $\mathfrak{B}$ and $\mathfrak{X}$. Starting from a $(\varphi_L,\Gamma_L)$-module $N$ over $\mathscr{R}_L(\mathfrak{X})$ we obtain by the ``inverse'' twisting the $(\varphi_L,\Gamma_L)$-module
\begin{equation*}
      N_\mathfrak{B} := (\mathscr{R}_{\Cp}(\mathfrak{X}) \otimes_{\mathscr{R}_L(\mathfrak{X})} N)^{G_L,*}
\end{equation*}
over $\mathscr{R}_L(\mathfrak{B})$.

\begin{theorem}\label{equiv}
The functors $M \longmapsto M_\mathfrak{X}$ and $N \longmapsto N_\mathfrak{B}$ induce equivalences of categories
\begin{equation*}
    \Mod_{L,an}(\mathscr{R}_L(\mathfrak{B})) \simeq \Mod_{L,an}(\mathscr{R}_L(\mathfrak{X}))
\end{equation*}
which are quasi-inverse to each other.
\end{theorem}
\begin{proof}
We have
\begin{align*}
  (M_\mathfrak{X})_\mathfrak{B} & = (\mathscr{R}_{\Cp}(\mathfrak{X}) \otimes_{\mathscr{R}_L(\mathfrak{X})} M_\mathfrak{X})^{G_L,*} \cong  (\mathscr{R}_{\Cp}(\mathfrak{B}) \otimes_{\mathscr{R}_L(\mathfrak{B})} M)^{G_L \times \id}  \\ & = \mathscr{R}_{\Cp}(\mathfrak{B})^{G_L}  \otimes_{\mathscr{R}_L(\mathfrak{B})} M = M \ ,
\end{align*}
where the second last identity uses the fact that $M$ is free. Correspondingly we obtain $(N_\mathfrak{B})_\mathfrak{X} \cong N$ since we may first reduce to the case of a free $N$ by Remark \ref{summand-of-free}.
\end{proof}

\subsection{Etale $L$-analytic $(\varphi_L,\Gamma_L)$-modules}\label{sec:etalepgm}

In the previous section, we have constructed an equivalence of categories (Thm.\ \ref{equiv}) between $\Mod_{L,an}(\mathscr{R}_L(\mathfrak{B}))$ and $\Mod_{L,an}(\mathscr{R}_L(\mathfrak{X}))$. In this section, we define the slope of an $L$-analytic $(\varphi_L,\Gamma_L)$-module, as well as \'etale $L$-analytic $(\varphi_L,\Gamma_L)$-modules, and prove that the above equivalence respects the slopes as well as the condition of being \'etale. Before we do that, we show that our equivalence implies the following theorem.

\begin{theorem}\label{freeness}
Any $L$-analytic $(\varphi_L,\Gamma_L)$-module $M$ over $\mathscr{R}_L(\mathfrak{X})$ is free as an $\mathscr{R}_L(\mathfrak{X})$-module.
\end{theorem}

Before proving this theorem, we need a few preliminary results. Let $\delta : L^\times \to L^\times$ be a continuous character. We define a $(\varphi_L,\Gamma_L)$-module $\mathscr{R}_K(\mathfrak{X})(\delta)$ of rank $1$ over $\mathscr{R}_K(\mathfrak{X})$ as follows. We set $\mathscr{R}_K(\mathfrak{X})(\delta) = \mathscr{R}_K(\mathfrak{X})\cdot e_\delta$ where $\varphi_L(e_\delta) = \delta(\pi_L) \cdot e_\delta$ and $\gamma(e_\delta) = \delta(\gamma) \cdot e_\delta$ if $a \in \Gamma_L$. We likewise define a $(\varphi_L,\Gamma_L)$-module $\mathscr{R}_K(\mathfrak{B})(\delta)$ of rank $1$ over $\mathscr{R}_K(\mathfrak{B})$ by $\mathscr{R}_K(\mathfrak{B})(\delta) = \mathscr{R}_K(\mathfrak{B})\cdot e_\delta$ where $\varphi_L(e_\delta) = \delta(\pi_L) \cdot e_\delta$ and $\gamma(e_\delta) = \delta(\gamma) \cdot e_\delta$ if $\gamma \in \Gamma_L$. Note that $\mathscr{R}_K(\mathfrak{X})(\delta)$ is $L$-analytic if and only if $\delta$ is $L$-analytic, and likewise for $\mathscr{R}_K(\mathfrak{B})(\delta)$.

\begin{proposition}\label{rank1b}
If $M$ is a $(\varphi_L,\Gamma_L)$-module of rank $1$ over $\mathscr{R}_L(\mathfrak{B})$, then there exists a character $\delta : L^\times \to L^\times$ such that $M \simeq \mathscr{R}_L(\mathfrak{B})(\delta)$.
\end{proposition}
\begin{proof}
This is \cite{FX} Prop.\ 1.9.
\end{proof}

\begin{lemma}\label{perdeltau}
If $\delta : L^\times \to L^\times$ is a locally $L$-analytic character, then there exists $\alpha \in \Cp^\times$ such that $g(\alpha) = \alpha \cdot \delta(\tau(g))$ for all $g \in G_L$.
\end{lemma}
\begin{proof}
Recall that since $\tau$ has a Hodge-Tate weight equal to $0$, there exists $\beta \in \Cp^\times$ such that $g(\beta) = \beta \cdot \tau(g)$ for all $g \in G_L$. Since $\delta : L^\times \to L^\times$ is locally $L$-analytic, there exists $s \in L$ and an open subgroup $U$ of $o_L^\times$ such that $\delta(a)  = \exp(s \cdot \log(a))$ for any $a \in U$. By approximating $\beta$ multiplicatively by an element in $\overline{L}$ we find a finite extension $L'$ of $L$ and a $\beta' \in \Cp^\times$ such that
\begin{itemize}
  \item[(1)] $g(\beta') = \beta' \cdot \tau(g)$ for $g \in G_{L'}$, and
  \item[(2)] $s \cdot \log(\beta') \in p \cdot o_{\Cp}$.
\end{itemize}
Let $\alpha' = \exp(s \cdot \log(\beta'))$. We then have $g(\alpha') = \alpha' \cdot \delta(\tau(g))$ for all $g \in G_{L'}$. The map $g \mapsto \frac{g(\alpha')}{\alpha' \cdot \delta(\tau(g))}$ therefore is a 1-cocycle on $G_L$ which is trivial on $G_{L'}$. By inflation-restriction its class then contains a 1-cocycle $\Gal(L'/L) \to (\Cp^\times)^{G_{L'}} = L'^\times$. Due to Hilbert 90 this class actually is trivial, which implies the assertion.
\end{proof}

\begin{proposition}\label{r1b2x}
If $\delta : L^\times \to L^\times$ is a locally $L$-analytic character, then the equivalence in Thm.\ \ref{equiv} satisfies $\mathscr{R}_L(\mathfrak{B})(\delta)_{\mathfrak{X}} \cong \mathscr{R}_L(\mathfrak{X})(\delta)$.
\end{proposition}
\begin{proof}
Let $\alpha \in \Cp^\times$ be the element afforded by Lemma \ref{perdeltau}. Then $\alpha \otimes e_\delta \in (\mathscr{R}_{\Cp}(\mathfrak{B}) \otimes_{\mathscr{R}_L(\mathfrak{B})} M_\mathfrak{B})^{G_L,*}$. We now compute
\begin{equation*}
  \mathscr{R}_L(\mathfrak{B})(\delta)_{\mathfrak{X}} = (\mathscr{R}_{\Cp}(\mathfrak{B}) \cdot e_\delta)^{G_L,*} = (\mathscr{R}_{\Cp}(\mathfrak{B}) \cdot \alpha e_\delta)^{G_L,*} = \mathscr{R}_{\Cp}(\mathfrak{B})^{G_L,*} \cdot \alpha e_\delta = \mathscr{R}_L(\mathfrak{X}) \cdot \alpha e_\delta \ .
\end{equation*}
The $o_L \setminus\{0\}$-action on this is given by $a(f \otimes \alpha e_\delta) = a(f) \otimes \alpha \delta(a) e_\delta$ for $0 \neq a \in o_L$. Hence $M_\mathfrak{X} \cong \mathscr{R}_L(\mathfrak{X})(\delta)$.
\end{proof}

\begin{proposition}\label{rank1x}
If $M$ is an $L$-analytic $(\varphi_L,\Gamma_L)$-module of rank $1$ over $\mathscr{R}_L(\mathfrak{X})$, then there exists a locally $L$-analytic character $\delta : L^\times \to L^\times$ such that $M \cong \mathscr{R}_L(\mathfrak{X})(\delta)$.
\end{proposition}
\begin{proof}
If $M_{\mathfrak{B}}$ is the $(\varphi_L,\Gamma_L)$-module of rank $1$ over $\mathscr{R}_L(\mathfrak{B})$ that comes from applying Thm.\ \ref{equiv} to $M$, then $M_{\mathfrak{B}}$ is an $L$-analytic $(\varphi_L,\Gamma_L)$-module of rank $1$ over $\mathscr{R}_L(\mathfrak{B})$. By Prop.\ \ref{rank1b}, there exists a character $\delta : L^\times \to L^\times$ such that $M \cong \mathscr{R}_L(\mathfrak{B})(\delta)$. Prop.\ \ref{r1b2x} now implies that $M \cong \mathscr{R}_L(\mathfrak{X})(\delta)$.
\end{proof}

This proposition implies theorem \ref{freeness} for $L$-analytic $(\varphi_L,\Gamma_L)$-modules of rank $1$ over $\mathscr{R}_L(\mathfrak{X})$.

\begin{proof}[Proof of theorem \ref{freeness}]
Since $\mathscr{R}_L(\mathfrak{X})$ is an integral domain the projective module $M$ has a well defined rank $r$. We may assume that $r \neq 0$. We have seen in Cor.\ \ref{pruefer2}.ii that $\mathscr{R}_L(\mathfrak{X})$ is a $1 \frac{1}{2}$ generator Pr\"ufer domain. Hence $M$ is isomorphic to a finite direct sum of invertible ideals in $\mathscr{R}_L(\mathfrak{X})$ (cf.\ \cite{CE} Prop.\ I.6.1). By induction with respect to the number of these ideals one easily deduces from this, by using the analog of Cor.\ \ref{free-invertible}.ii, that
\begin{equation*}
  M \cong \mathscr{R}_L(\mathfrak{X})^{r-1} \oplus I
\end{equation*}
for some invertible ideal $I \subseteq \mathscr{R}_L(\mathfrak{X})$. Moreover, $I$ is isomorphic to the exterior power $\bigwedge^r M$ of $M$ (cf.\ \cite{B-A} III \S7). The latter (and therefore also $I$) inherits from $M$ the structure of an $L$-analytic $(\varphi_L,\Gamma_L)$-module $M$ over $\mathscr{R}_L(\mathfrak{X})$. This reduces the proof of our assertion to the case that the $(\varphi_L,\Gamma_L)$-module $M$ is of rank $1$, which follows from Prop.\ \ref{rank1x}.
\end{proof}

We consider now a $(\varphi_L,\Gamma_L)$-module $M$ over $\mathscr{R}_K(\mathfrak{X})$ that is free of $\mathrm{rank}(M) = r$. If we pick a basis $e_1, \ldots, e_r$ of $M$ then the isomorphism in Def.\ \ref{def:modR} guarantees that $\varphi_M(e_1), \ldots, \varphi_M(e_r)$ again is a basis of $M$. The matrix
\begin{equation*}
  A_M = (a_{ij}) \quad\text{where $\varphi_M(e_j) = \sum_{i=1}^r a_{ij} e_i$}
\end{equation*}
therefore is invertible over $\mathscr{R}_K(\mathfrak{X})$. If we change the given basis by applying an invertible matrix $B$ then $A_M$ gets replaced by $B^{-1} A_M \varphi_L(B)$. By the above Prop.\ \ref{unitsR} we have $\det(A_M)$, $\det(B) \in \mathscr{E}_K^\dagger(\mathfrak{X})^\times$. We compute
\begin{align*}
  \|\det(B^{-1} A_M \varphi_L (B))\|_1 & = \|\det(B)\|_1^{-1} \cdot \|\det(A_M)\|_1 \cdot \|\varphi_L(\det(B))\|_1  \\
                                       & = \|\det(B)\|_1^{-1} \cdot \|\det(A_M)\|_1 \cdot \|\det(B)\|_1  \\
                                       & = \|\det(A_M)\|_1 \ ,
\end{align*}
where the first, resp.\ second, identity uses that the norm $\|\ \|_1$ is multiplicative, resp.\ that $\varphi_L$ is $\|\ \|_1$-isometric. This leads to the following definitions.

\begin{definition}\phantomsection\label{def:deg}
Let $M$ be a free $(\varphi_L,\Gamma_L)$-module over $\mathscr{R}_K(\mathfrak{X})$.
\begin{enumerate}
\item The degree of $M$ is the rational number $\deg(M)$ such that
\begin{equation*}
  p^{\deg(M)} = \|{\det(A_M)}\|_1 \ .
\end{equation*}
\item The slope of $M$ is $\mu(M) := \deg(M)/\mathrm{rk}(M)$.
\end{enumerate}
\end{definition}

\begin{definition}\label{def:etaleR}
An $L$-analytic $(\varphi_L,\Gamma_L)$-module $M$ over $\mathscr{R}_L(\mathfrak{X})$ is called \'etale if $M$ has degree zero and if every sub-$(\varphi_L,\Gamma_L)$-module $N$ of $M$ has degree $\geq 0$. Let $\Mod^{et}_{L,an}(\mathscr{R}_L(\mathfrak{X}))$ denote the full subcategory of all \'etale ($L$-analytic)  $(\varphi_L,\Gamma_L)$-modules over $\mathscr{R}_L(\mathfrak{X})$.
\end{definition}

\begin{remark}
\label{remetale}
Every sub-$(\varphi_L,\Gamma_L)$-module $N$ of an $L$-analytic $(\varphi_L,\Gamma_L)$-module $M$ is $L$-analytic and  hence is free by Thm.\ \ref{freeness}, so that we can define its degree.
\end{remark}

As before this makes equally sense for $\mathfrak{B}$ instead of $\mathfrak{X}$ so that we also have the category $\Mod^{et}_{L,an}(\mathscr{R}_K(\mathfrak{B}))$. By the subsequent remark, the definition of \'etaleness  coincides with the usual definition in the case of $\mathfrak{B}$.

\begin{remark}\phantomsection\label{B-proj-free}
\begin{itemize}
  \item[i.] Any finitely generated submodule of a finitely generated projective module over $\mathscr{R}_L(\mathfrak{B})$ is free.
  \item[ii.] Let $M$ be an \'etale $(\varphi_L,\Gamma_L)$-module over $\mathscr{R}_L(\mathfrak{B})$; then any sub-$\varphi_L$-module $N$ of $M$ has degree $\geq 0$.
\end{itemize}
\end{remark}
\begin{proof}
i. It is well known to follow from Lazard's work in \cite{Laz} that $\mathscr{R}_L(\mathfrak{B})$ is a Bezout domain (compare \cite{TdA} Satz 10.1). But it is a general fact that finitely generated projective modules over Bezout domains are free (cf.\ \cite{Lam} Thm.\ 2.29). Moreover, even over a Pr\"ufer domain, finitely generated submodules of finitely generated projective modules are projective.

ii. For sake of clarity: A sub-$\varphi_L$-module $N$ of $M$ is a finitely generated $\varphi_L$-invariant submodule $N \subseteq M$ such that $\mathscr{R}_L(\mathfrak{B}) \otimes_{\mathscr{R}_L(\mathfrak{B}),\varphi_L} N \xrightarrow{\cong} N$. By i. any such $N$ is free of rank $\leq \mathrm{rank}(M)$.

Let now $N \subseteq M$ be an arbitrary but fixed nonzero sub-$\varphi_L$-module. We have to show that $\mu(N) \geq 0$. By \cite{KedAst} Lemma 1.4.12 it suffices to consider the unique largest sub-$\varphi_L$-module $M_1 \subseteq M$ of least slope. We claim that $M_1$ is $\Gamma_L$-invariant. Let $g \in \Gamma_L$ and let $e_1, \ldots, e_r$ be a basis of $M_1$. Then $ge_1, \ldots, ge_r$ is a basis of the sub-$\varphi_L$-module $gM_1$. By the $\Gamma_L$-invariance of the norm $\|\ \|_1$ we obtain
\begin{equation*}
  \|\det(A_{gM_1})\|_1 = \|g \det(A_{M_1})\|_1 = \|\det(A_{M_1})\|_1 \quad\text{and hence} \quad \mu(gM_1) = \mu(M_1) \ .
\end{equation*}
It follows that $gM_1 \subseteq M_1$. This shows that $M_1$ is a sub-$(\varphi_L,\Gamma_L)$-module. By the \'etaleness of $M$ we then must have $\mu(M_1) \geq 0$.
\end{proof}

\begin{proposition}
\label{mxet}
If $M \in \Mod_{L,an}(\mathscr{R}_L(\mathfrak{B}))$, then $\deg(M_\mathfrak{X}) = \deg(M)$.
\end{proposition}
\begin{proof}
The degree of $M_\mathfrak{X}$ depends only on $\mathscr{R}_{\Cp}(\mathfrak{X}) \otimes_{\mathscr{R}_L(\mathfrak{X})} M_\mathfrak{X}$, and we have
\begin{equation*}
  \mathscr{R}_{\Cp}(\mathfrak{X}) \otimes_{\mathscr{R}_L(\mathfrak{X})} M_\mathfrak{X} =
\mathscr{R}_{\Cp}(\mathfrak{B}) \otimes_{\mathscr{R}_L(\mathfrak{B})} M \ .
\end{equation*}
\end{proof}

\begin{theorem}\label{etalequiv}
The functors $M \longmapsto M_\mathfrak{X}$ and $N \longmapsto N_\mathfrak{B}$ induce an equivalence of categories
\begin{equation*}
   \Mod_{L,an}^{et}(\mathscr{R}_L(\mathfrak{B})) \simeq \Mod_{L,an}^{et}(\mathscr{R}_L(\mathfrak{X}))
\end{equation*}
\end{theorem}
\begin{proof}
By Prop.\ \ref{mxet}, for every sub-$(\varphi_L,\Gamma_L)$-module $N$ of $M$, we have $\deg(N_\mathfrak{X}) = \deg(N)$. If $M$ is \'etale, then this shows that $\deg(M_\mathfrak{X}) = 0$ and that $\deg(N_\mathfrak{X}) \geq 0$ for every sub-$(\varphi_L,\Gamma_L)$-module $N$ of $M$. As a consequence of Thm.\ \ref{equiv}, if $N$ runs over all sub-$(\varphi_L,\Gamma_L)$-modules of $M$, then $N_\mathfrak{X}$ runs over all sub-$(\varphi_L,\Gamma_L)$-modules of $M_\mathfrak{X}$. It follows that $M_\mathfrak{X}$ is \'etale. By symmetry, the same holds with $N$ and $N_\mathfrak{B}$. Hence the equivalence in Thm.\ \ref{equiv} restricts to the asserted equivalence.
\end{proof}

\begin{corollary}
\label{repequiv}
There is an equivalence of categories:
\[  \Mod_{L,an}^{et}(\mathscr{R}_L(\mathfrak{X})) \simeq \{ \text{$L$-analytic representations of of $G_L$} \}. \]
\end{corollary}
\begin{proof}
This follows from Thm.\ \ref{etalequiv}, and Thm.\ D of \cite{PGMLAV} according to which there
is an equivalence of categories $\Mod_{L,an}^{et}(\mathscr{R}_L(\mathfrak{B})) \simeq \{ L$-analytic representations of of $G_L\}$.
\end{proof}

\subsection{Crystalline $(\varphi_L,\Gamma_L)$-modules over $\mathcal{O}_L(\mathfrak{X})$}
\label{sec:fil-to-R}

In \cite{KR} and \cite{KFC} they attach to every filtered $\varphi_L$-module $D$ over $L$
a $(\varphi_L,\Gamma_L)$-module $\mathcal{M}(D)$ over $\mathcal{O}_L(\mathfrak{B})$.
We carry out the corresponding construction over $\mathcal{O}_L(\mathfrak{X})$. In the next section,
we show that the two generalizations are compatible with the equivalence of categories of Theorem A, after extending scalars to $\mathscr{R}_L(\mathfrak{B})$ and $\mathscr{R}_L(\mathfrak{X})$.
We now give the analogue over $\mathfrak{X}$ of Kisin's construction over $\mathfrak{B}$. It is possible to replace everywhere in this section $\mathfrak{X}$ by $\mathfrak{B}$, and we then recover the results of \cite{KR}. If $Y = \supp(\Delta)$ is the support of an effective divisor $\Delta$ on $\mathfrak{X}$ of the form
\begin{equation*}
    \Delta(x) :=
    \begin{cases}
    1 & \text{if $x \in Y$}, \\
    0 & \text{otherwise},
    \end{cases}
\end{equation*}
then we simply write $I_Y := I_\Delta = \{f \in \mathcal{O}_L(\mathfrak{X}) : f(x) = 0 \ \text{for any $x \in Y$}\}$ for the corresponding closed ideal, and let $I_Y^{-1}$ denote its inverse fractional ideal. We introduce the $\mathcal{O}_L(\mathfrak{X})$-algebra $\mathcal{O}_L(\mathfrak{X})[Y^{-1}] := \bigcup_{m \geq 1} I_{Y}^{-m}$ and, for any $\mathcal{O}_L(\mathfrak{X})$-module $M$, we put
\begin{equation*}
  M[Y^{-1}] := \mathcal{O}_L(\mathfrak{X})[Y^{-1}] \otimes_{\mathcal{O}_L(\mathfrak{X})} M \ .
\end{equation*}

\begin{lemma}\label{Y-1-flat}
The ring homomorphism $\mathcal{O}_L(\mathfrak{X}) \longrightarrow \mathcal{O}_L(\mathfrak{X})[Y^{-1}]$ is flat.
\end{lemma}
\begin{proof}
The assertion is a special case of the following fact. Let $I \subseteq A$ be an invertible ideal in the integral domain $A$ and consider the ring homomorphism $A \longrightarrow A_I := \bigcup_{m \geq 1} I^{-m}$. We have $\sum_{i=1}^n f_i g_i = 1$ for appropriate elements $f_i \in I$ and $g_i \in I^{-1}$. The $f_i$ then generate the unit ideal in the ring $A_I$. Hence $\Spec(A_I) = \bigcup_i \Spec((A_I)_{f_i})$. But $Af_i \subseteq I$ implies $I^{-1} \subseteq Af_i^{-1}$, hence $A_I \subseteq A_{f_i}$, and therefore $(A_I)_{f_i} = A_{f_i}$ and $\Spec(A_I) = \bigcup_i \Spec(A_{f_i})$. Since the localizations $A \longrightarrow A_{f_i}$ are flat it follows that $A \longrightarrow A_I$ is flat.
\end{proof}

In this section we will construct $(\varphi_L,\Gamma_L)$-modules over $\mathscr{R}_L(\mathfrak{X})$ from the following kind of data.

\begin{definition}\label{def:fil-vs}
A filtered $\varphi_L$-module $D$ is a finite dimensional $L$-vector space equipped with
\begin{itemize}
  \item[--] a linear automorphism $\varphi_D : D \longrightarrow D$
  \item[--] and a descending, separated, and exhaustive $\ZZ$-filtration $\Fil^\bullet D$ by vector subspaces.
\end{itemize}
\end{definition}

In the following we let $Z \subseteq \mathfrak{X}$ denote the subset of all torsion points different from $1$, i.e., $Z = (\bigcup_{n \geq 1} \mathfrak{X}[\pi_L^n]) \setminus \{1\}$. We define the function $n : Z \longrightarrow \ZZ_{\geq 0}$ by $x \in \mathfrak{X}[\pi_L^{n(x)+1}]) \setminus \mathfrak{X}[\pi_L^{n(x)}])$. We keep the notations introduced in the previous section \ref{sec:globalring} (with $K = L$). We always equip the field of fractions $\Fr(\mathcal{O}_x)$ of the local ring $\mathcal{O}_x$, for any point $x \in Z$, with the $\mathfrak{m}_x$-adic filtration.

Let $D$ be a filtered $\varphi_L$-module of dimension $d_D$. We introduce the $\mathcal{O}_L(\mathfrak{X})$-module
\begin{equation*}
    \mathcal{M}(D) := \{ s \in \mathcal{O}_L(\mathfrak{X})[Z^{-1}] \otimes_L D : (\id \otimes \varphi_D^{-n(x)})(s) \in \Fil^0(\Fr(\mathcal{O}_x) \otimes_L D) \ \text{for any $x \in Z$} \} \ ,
\end{equation*}
where the $\Fil^0$ refers to the tensor product filtration on $\Fr(\mathcal{O}_x) \otimes_L D$. Suppose that $a \leq b$ are integers such that $\Fil^a D = D$ and $\Fil^{b+1} D = 0$. Then
\begin{equation}\label{f:bounds}
    I_Z^{-a} \otimes_L D \subseteq \mathcal{M}(D) \subseteq I_Z^{-b} \otimes_L D \ .
\end{equation}

\begin{lemma}\label{fg-proj}
$\mathcal{M}(D)$ is a finitely generated projective $\mathcal{O}_L(\mathfrak{X})$-module of rank $d_D$. If $\widetilde{\mathcal{M}}(D)$ denotes the corresponding coherent $\mathcal{O}$-module sheaf and $\widetilde{\mathcal{M}}_x(D)$ its stalk in any point $x \in \mathfrak{X}$, then $\widetilde{\mathcal{M}}_x(D) = \mathcal{O}_x \otimes_L D$ for any $x \not\in Z$.
\end{lemma}
\begin{proof}
The definition of $\mathcal{M}(D)$ shows that
\begin{align*}
    \mathcal{M}(D) = \ker \big(I_Z^{-b} \otimes_L D \longrightarrow \prod\limits_r I_Z^{-b} \mathcal{O}_L(\mathfrak{X}(r)) \otimes_L D/\mathcal{O}_L(\mathfrak{X}(r)) \otimes_{\mathcal{O}_L(\mathfrak{X})} \mathcal{M}(D)  \big).
\end{align*}
Since the above map is continuous and since over the noetherian Banach algebra $\mathcal{O}_L (\mathfrak{X}(r))$ any submodule of a finitely generated module is closed it follows that $\mathcal{M}(D)$ is closed in the finitely generated projective $\mathcal{O}_L(\mathfrak{X})$-module $I_Z^{-b} \otimes_L D$. Hence Lemma \ref{closed-proj} implies that $\mathcal{M}(D)$ is finitely generated projective. The second part of the assertion is an immediate consequence of  \eqref{f:bounds}.
\end{proof}

We may also compute the stalks of $\widetilde{\mathcal{M}}(D)$ in the points in $Z$.

\begin{lemma}\label{stalks}
For any $x \in Z$ we have
\begin{align*}
     & \widetilde{\mathcal{M}}_x(D) = \mathcal{O}_x  \otimes_{\mathcal{O}_L(\mathfrak{X})} \mathcal{M}(D)
     =  \{ s \in \Fr(\mathcal{O}_x) \otimes_L  D : (\id \otimes \varphi_D^{-n(x)}) (s) \in \Fil^0 ( \Fr (\mathcal{O}_x) \otimes_L D ) \}   \\
     & \qquad \xrightarrow [\id \otimes \varphi_D^{-n(x)}] { \cong } \Fil^0 ( \Fr(\mathcal{O}_x) \otimes_L D )
\end{align*}
\end{lemma}
\begin{proof}
We temporarily define
\begin{equation*}
    \mathcal{M}_x(D):  = \{ s \in \mathcal{O}_L(\mathfrak{X})[Z^{-1}] \otimes_L  D : ( \id \otimes \varphi _D^{-n(x)} ) (s) \in \Fil^0 \left( \Fr(\mathcal{O}_x) \otimes_L D \right) \}.
\end{equation*}
We pick a function $f \in \mathcal{O}_L(\mathfrak{X})$ whose zero set contains $Z \setminus \{x\}$ but not $x$ (Lemma \ref{prescribe-divisor}). Then
\begin{equation*}
    \mathcal{M}(D) \subseteq \mathcal{M}_x (D) \subseteq f^{-(b-a)} \mathcal{M}(D) \ .
\end{equation*}
This shows that $\mathcal{O}_x \otimes_{\mathcal{O}_L(\mathfrak{X})} \mathcal{M}(D)  =  \mathcal{O}_x \otimes_{\mathcal{O}_L(\mathfrak{X})} \mathcal{M}_x(D)$. Due to \eqref{f:bounds} we have the commutative diagram
\begin{equation*}
    \xymatrix{
      \Fil^{-b}(\Fr(\mathcal{O}_x)) \otimes_L D  \ar[r]^{=} & \Fil^{-b}(\Fr(\mathcal{O}_x)) \otimes_L D   \\
      \mathcal{O}_x \otimes_{\mathcal{O}_L(\mathfrak{X})} \mathcal{M}_x(D) \ar[u]_{\subseteq} \ar[r] & \{ s \in \Fr(\mathcal{O}_x) \otimes_L  D : (\id \otimes \varphi_D^{-n(x)}) (s) \in \Fil^0 ( \Fr (\mathcal{O}_x) \otimes_L D ) \} \ar[u]^{\subseteq} \\
       \Fil^{-a}(\Fr(\mathcal{O}_x)) \otimes_L D \ar[u]_{\subseteq} \ar[r]^{=} & \Fil^{-a}(\Fr(\mathcal{O}_x)) \otimes_L D  \ar[u]_{\subseteq} }
\end{equation*}
In particular, the middle horizontal arrow is injective. But it also is surjective as follows easily from the surjectivity of the map $I_Z^{-b} \longrightarrow \Fil^{-b}(\Fr(\mathcal{O}_x)) / \Fil^{-a}(\Fr(\mathcal{O}_x))$ (compare the proof of Lemma \ref{germ}).
\end{proof}

The injective ring endomorphism $\varphi_L$ of $\mathcal{O}_L(\mathfrak{X})$ extends, of course, to its field of fractions. Using Lemma \ref{unramified}.iii one checks that $\varphi_L(I_Z^{-1}) \subseteq I_Z^{-1}$. Hence $\mathcal{O}_L(\mathfrak{X})[Z^{-1}] \otimes_L D$ carries the injective $\varphi_L$-linear endomorphism $\varphi_L \otimes \varphi_D$. We define $Z_0 := \{x \in Z : n(x) = 0\} = \mathfrak{X}[\pi_L] \setminus \{1\}$.

\begin{lemma}\label{varphi}
$\varphi_L \otimes \varphi_D$ induces an injective $\varphi_L$-linear homomorphism
\begin{equation*}
    \varphi_{\mathcal{M}(D)} : \mathcal{M} (D) \longrightarrow I_{Z_0}^a \mathcal{M}(D) \ .
\end{equation*}
\end{lemma}

\begin{proof}
Let $s \in \mathcal{M}(D)$. We show that $I_{Z_0}^{-a} (\varphi_L \otimes \varphi_D) (s)$ is contained in $\mathcal{M}(D)$. Since $I_{Z_0}^{-1} \subseteq I_Z^{-1}$ we certainly have $I_{Z_0}^{-a} (\varphi_L \otimes \varphi_D) (s) \subseteq \mathcal{O}_L(\mathfrak{X})[Z^{-1}] \otimes_L D$. Hence we have to check the local condition for any $x \in Z$.  First we assume that $n(x) \geq 1$ or, equivalently, that $x \not\in Z_0$.  Then
\begin{align*}
    (\id \otimes \varphi_D^{-n(x)})( I_{Z_0}^{-a} (\varphi_L \otimes \varphi_D) (s))
     & = I_{Z_0}^{-a} (\varphi_L \otimes \id) (\id \otimes \varphi_D^{-n(x)+1} ) (s) \\
& = I_{Z_0}^{-a} (\varphi_L \otimes \id) (\id \otimes \varphi_D^{-n(\pi_L^* (x))} ) (s) \\
& \subseteq I_{Z_0}^{-a} (\varphi_L \otimes \id) (\Fil^0 ( \Fr(\mathcal{O}_{\pi_L^*(x)}) \otimes_L D)) \\
& \subseteq I_{Z_0}^{-a} \Fil^0 (\Fr(\mathcal{O}_x) \otimes_L D) \\
& = \Fil^0 (\Fr(\mathcal{O}_x) \otimes_L D) \ .
\end{align*}
If $n(x) = 0$ then, using that $\pi_L^* (x) = 1 \not\in Z$, we have
\begin{align*}
I_{Z_0}^{-a} (\varphi_L \otimes \varphi_D) (s) & = I_{Z_0}^{-a} (\varphi_L \otimes \id) (\id \otimes \varphi_D) (s) \subseteq I_{Z_0}^{-a} (\varphi_L \otimes \id) (\mathcal{O}_{\pi_L^*(x)} \otimes_L D) \\
    & \subseteq I_{Z_0}^{-a} \mathcal{O}_x \otimes_L D = \mathfrak{m}_x^{-a} \otimes_L D \subseteq \Fil^0 ( \Fr(\mathcal{O}_x) \otimes_L D).
\end{align*}
\end{proof}

Let $\mathfrak{n}(1) \subseteq \mathcal{O}_L(\mathfrak{X})$ denote the (closed) maximal ideal of functions vanishing in the point $1$.

\begin{lemma}\label{n(1)}
We have $I_{Z_0} = \mathcal{O}_L(\mathfrak{X}) \varphi_L(\mathfrak{n}(1)) \mathfrak{n}(1)^{-1}$.
\end{lemma}
\begin{proof}
Since $\varphi_L(\mathfrak{n}(1)) \subseteq \mathfrak{n}(1)$ we have $\mathcal{O}_L(\mathfrak{X}) \varphi_L(\mathfrak{n}(1)) \mathfrak{n}(1)^{-1} \subseteq \mathfrak{n}(1) \mathfrak{n}(1)^{-1} = \mathcal{O}_L(\mathfrak{X})$. Hence the right hand side of the asserted identity is a finitely
generated and therefore (Prop.\ \ref{closed-fingen}) closed ideal in
$\mathcal{O}_L(\mathfrak{X})$ as is the left hand side. By Prop. \ref{divisor-theory} this reduces us
to checking that both ideals have the same divisor. But this follows immediately from
Lemma \ref{unramified}.iii.
\end{proof}

In view of this Lemma \ref{n(1)} the Lemma \ref{varphi} can be restated as saying that $\varphi_L \otimes \varphi_D$ induces a $\varphi_L$-linear endomorphism
\begin{equation*}
    \varphi_{\mathcal{M}(D)} : \mathfrak{n}(1)^{-a} \mathcal{M} (D) \longrightarrow \mathfrak{n}(1)^{-a} \mathcal{M}(D) \ ,
\end{equation*}
and we therefore may define the linear map
\begin{align*}
   \overline{\varphi}_{\mathcal{M}(D)} : \mathcal{O}_L(\mathfrak{X}) \otimes_{\mathcal{O}_L(\mathfrak{X}),\varphi_L} \mathfrak{n}(1)^{-a} \mathcal{M}(D) & \longrightarrow  \mathfrak{n}(1)^{-a} \mathcal{M}(D) \\
   f \otimes s & \longmapsto f \varphi_{\mathcal{M}(D)} (s) \ .
\end{align*}
Note that $\mathfrak{n}(1)^{-a} \mathcal{M}(D)$ is a finitely generated projective $\mathcal{O}_L(\mathfrak{X})$-module by Lemma \ref{fg-proj} and Cor.\ \ref{pruefer}.

\begin{remark}\label{phi-pullback}
Let $\phi := \pi_L^* : \mathfrak{X} \longrightarrow \mathfrak{X}$; for any coherent $\mathcal{O}$-module $\mathfrak{M}$ such that $\mathfrak{M}(\mathfrak{X})$ is a finitely generated projective $\mathcal{O}_L(\mathfrak{X})$-module the coherent $\mathcal{O}$-module $\phi^*\mathfrak{M}$ has global sections  $\mathcal{O}_L(\mathfrak{X}) \otimes_{\mathcal{O}_L(\mathfrak{X}), \varphi_L} \mathfrak{M}(\mathfrak{X})$.
\end{remark}
\begin{proof}
Let $\phi_r$ denote the restriction of $\phi$ to $\mathfrak{X}(r)$. We compute
\begin{align*}
    (\phi^*\mathfrak{M})(\mathfrak{X}) & = \varprojlim (\phi^*\mathfrak{M})(\mathfrak{X}(r)) = \varprojlim (\phi_r^*(\mathfrak{M}_{|\mathfrak{X}(r)}))(\mathfrak{X}(r)) \\
    & = \varprojlim \mathcal{O}_L(\mathfrak{X}(r)) \otimes_{\mathcal{O}_L(\mathfrak{X}(r)), \varphi_L} \mathfrak{M}(\mathfrak{X}(r)) \\
    & = \varprojlim \mathcal{O}_L(\mathfrak{X}(r)) \otimes_{\mathcal{O}_L(\mathfrak{X}), \varphi_L} \mathfrak{M}(\mathfrak{X}) \\
    & = \big( \varprojlim \mathcal{O}_L(\mathfrak{X}(r)) \big) \otimes_{\mathcal{O}_L(\mathfrak{X}), \varphi_L} \mathfrak{M}(\mathfrak{X}) \\
    & = \mathcal{O}_L(\mathfrak{X}) \otimes_{\mathcal{O}_L(\mathfrak{X}), \varphi_L} \mathfrak{M}(\mathfrak{X}) \ .
\end{align*}
Here the fourth identity comes from Remark \ref{coherent}.i, and the fifth identity uses the fact that the tensor product by a finitely generated projective module commutes with projective limits.
\end{proof}

Let $\mathcal{N}$ denote the coherent $\mathcal{O}$-module with global sections $\mathcal{N}(\mathfrak{X}) = \mathfrak{n}(1)^{-a} \mathcal{M} (D)$. Using Remarks \ref{coherent} and \ref{phi-pullback} we see that $\widetilde{\varphi}_{\mathcal{M}(D)}$ is the map induced on global sections by a homomorphism $\Phi : \phi^* \mathcal{N} \longrightarrow \mathcal{N}$ of coherent $\mathcal{O}$-modules (where $\phi := \pi_L^*$).

\begin{proposition}\phantomsection\label{varphi-iso}
\begin{itemize}
  \item[i.] On stalks in a point $x \in \mathfrak{X}$ the map $\Phi$ is an isomorphism if $x \not\in Z_0$ and identifies with
\begin{equation*}
    \mathfrak{m}_x^{-a} \otimes_L D \xrightarrow{\; \subseteq \;} \sum_j \mathfrak{m}_x^{-j} \otimes_L \Fil^j D \ ,
\end{equation*}
whose cokernel is $\mathcal{O}_x$-isomorphic to $\sum_{j = 0}^{b-a} \mathcal{O}_x / \mathfrak{m}_x^j \otimes_L \gr^{a+j} D$, if $x \in Z_0$.
  \item[ii.] The map $\overline{\varphi}_{\mathcal{M}(D)}$ is injective and its cokernel is annihilated by $I_{Z_0}^{b-a}$.
\end{itemize}
\end{proposition}
\begin{proof}
i. The map under consideration between the stalks in a point $x$ is
\begin{equation*}
    \mathcal{O}_x \otimes_{\mathcal{O}_{\pi_L^*(x),\varphi_L}} \widetilde{\mathcal{M}}_{\pi_L^*(x)}(D)  \xrightarrow[\mathcal{O}_x \otimes (\varphi_{\mathcal{M}(D)})_x]{}
    \widetilde{\mathcal{M}}_x(D)
\end{equation*}
if $x \not\in Z_0 \cup \{1\}$, resp.\
\begin{equation*}
    \mathcal{O}_x \otimes_{\mathcal{O}_{1,\varphi_L}} \mathfrak{m}_1^{-a} \widetilde{\mathcal{M}}_{1}(D)  \xrightarrow[\mathcal{O}_x \otimes (\varphi_{\mathcal{M}(D)})_x]{}
    \widetilde{\mathcal{M}}_x(D)
\end{equation*}
if $x \in Z_0$, resp.\
\begin{equation*}
    \mathcal{O}_1 \otimes_{\mathcal{O}_{1,\varphi_L}} \mathfrak{m}_1^{-a} \widetilde{\mathcal{M}}_{1}(D)  \xrightarrow[\mathcal{O}_1 \otimes (\varphi_{\mathcal{M}(D)})_1]{}
    \mathfrak{m}_1^{-a} \widetilde{\mathcal{M}}_1(D)
\end{equation*}
if $x = 1$. If $x \not\in Z$ then also $\pi_L^*(z) \not\in Z$, and this map identifies with the isomorphism
\begin{equation*}
    \mathcal{O}_x \otimes_L D \xrightarrow[\id \otimes \varphi_D]{\cong} \mathcal{O}_x \otimes_L D \ .
\end{equation*}
if $x \neq 1$, resp.\
\begin{equation*}
    \mathcal{O}_1 \otimes_{\mathcal{O}_{1,\varphi_L}} \mathfrak{m}_1^{-a} \otimes_L D = \mathcal{O}_1 \varphi_L(\mathfrak{m}_1)^{-a} \otimes_L D  \xrightarrow[\id \otimes \varphi_D]{\cong}
    \mathfrak{m}_1^{-a} \otimes_L D
\end{equation*}
if $x= 1$, where the identity comes from the flatness of $\varphi_L$ (Lemma \ref{unramified}.i).

Next suppose that $x \in Z \setminus Z_0$ so that $\pi_L^*(x) \in Z$ as well.  Then, by Lemma \ref{stalks}, we have the commutative diagram
\begin{equation*}
\begin{xymatrix}{
 \mathcal{O}_x \otimes_{\mathcal{O}_{\pi_L^*(x),\varphi_L}} \widetilde{\mathcal{M}}_{\pi_L^*(x)}(D)
 \ar[d]^{\cong}_{\id \otimes \varphi_D^{-n(x)+1}} \ar[r]^-{\mathcal{O}_x \otimes (\varphi_{\mathcal{M}(D)})_x} & \widetilde{\mathcal{M}}_x(D) \ar[d]_{\cong}^{\ \id \otimes \varphi_D^{-n(x)}} \\
 \mathcal{O}_x \otimes_{\mathcal{O}_{\pi_L^*(x),\varphi_L}} \Fil^0 (\Fr(\mathcal{O}_{\pi_L^*(x)}) \otimes_L D) \ar[d]^=   & \Fil^0 (\Fr(\mathcal{O}_x) \otimes_L D) \ar[dd]_= \\
\sum\limits_{j} \mathcal{O}_x \otimes_{\mathcal{O}_{\pi_L^*(x),\varphi_L}} \mathfrak{m}^{-j}_{\pi_L^*(x)} \otimes_L \Fil^j D \ar[d]_{\id \otimes \varphi_L \otimes \id}^{\cong}        & \\
\sum\limits_{j}(\mathcal{O}_x \varphi_L(\mathfrak{m}_{\pi_L^* (x)} ) )^{-j} \otimes_L \Fil^j D \ar[r]^-{=}        & \sum\limits_{j} \mathfrak{m}^{-j}_x \otimes \Fil^j D       }
\end{xymatrix}
\end{equation*}
where the lower left vertical map is an isomorphism again by the flatness of $\varphi_L$. For $x \in Z_0$ the corresponding diagram is
\begin{equation*}
\begin{xymatrix}{
 \mathcal{O}_x \otimes_{\mathcal{O}_{1,\varphi_L}} \mathfrak{m}_1^{-a} \widetilde{\mathcal{M}}_1(D)
 \ar[d]^{\cong} \ar[r]^-{\mathcal{O}_x \otimes (\varphi_{\mathcal{M}(D)})_x} & \widetilde{\mathcal{M}}_x(D) \ar[d]_{\cong}^{\ \id \otimes \varphi_D^{-1}} \\
 \mathcal{O}_x \otimes_{\mathcal{O}_{1,\varphi_L}} \mathfrak{m}_1^{-a} \otimes_L D \ar[d]^=   & \Fil^0 (\Fr(\mathcal{O}_x) \otimes_L D) \ar[d]_= \\
 \mathfrak{m}_x^{-a} \otimes_L D \ar[r]^-{\subseteq}        & \sum\limits_{j} \mathfrak{m}^{-j}_x \otimes \Fil^j D.       }
\end{xymatrix}
\end{equation*}
(Note that we have used several times implicitly that $\varphi_L$ is unramified by Lemma \ref{unramified}.ii.)

ii. This follows from i. by passing to global sections.
\end{proof}

\begin{corollary}\phantomsection\label{iso-Z0invert}
\begin{itemize}
  \item[i.] We have $I_{Z_0}^b \mathcal{M}(D) \subseteq \mathcal{O}_L(\mathfrak{X}) \varphi_{\mathcal{M}(D)}(\mathcal{M}(D)) \subseteq I_{Z_0}^a \mathcal{M}(D)$.
  \item[ii.] The $\mathcal{O}_L(\mathfrak{X})$-linear map $\mathcal{O}_L(\mathfrak{X}) \otimes_{\mathcal{O}_L(\mathfrak{X}),\varphi_L} \mathcal{M}(D) \longrightarrow  I_{Z_0}^a \mathcal{M}(D)$
      induced by $\varphi_{\mathcal{M}(D)}$ is injective with cokernel annihilated by $I_{Z_0}^{b-a}$.
  \item[iii.] The $\mathcal{O}_L(\mathfrak{X})[Z_0^{-1}]$-linear map $\mathcal{O}_L(\mathfrak{X})[Z_0^{-1}] \otimes_{\mathcal{O}_L(\mathfrak{X}),\varphi_L} \mathcal{M}(D) \xrightarrow{\;\cong\;}  \mathcal{M}(D)[Z_0^{-1}]$
      induced by $\varphi_{\mathcal{M}(D)}$ is an isomorphism.
\end{itemize}
\end{corollary}
\begin{proof}
i. The right hand inclusion was shown in Lemma \ref{varphi}. For the left hand inclusion
we observe that Prop.\ \ref{varphi-iso}.ii implies that
\begin{equation*}
  I_{Z_0}^{b-a} \mathfrak{n}(1)^{-a} \mathcal{M}(D) \subseteq \mathcal{O}_L(\mathfrak{X})
  \varphi_L(\mathfrak{n}(1)^{-a}) \varphi_{\mathcal{M}(D)}(\mathcal{M}(D)) \ .
\end{equation*}
Using Lemma \ref{n(1)} we deduce that
\begin{equation*}
  I_{Z_0}^b \mathcal{M}(D) \subseteq (\mathcal{O}_L(\mathfrak{X})
  \varphi_L(\mathfrak{n}(1))^a \mathcal{O}_L(\mathfrak{X})
  \varphi_L(\mathfrak{n}(1)^{-a}) \varphi_{\mathcal{M}(D)}(\mathcal{M}(D)) \ .
\end{equation*}
But one easily checks that $(\mathcal{O}_L(\mathfrak{X})
  \varphi_L(\mathfrak{n}(1))^a \mathcal{O}_L(\mathfrak{X})
  \varphi_L(\mathfrak{n}(1)^{-a}) = \mathcal{O}_L(\mathfrak{X})$.

ii. The annihilation property of the cokernel is clear from i. To establish the injectivity we may assume without loss of generality that $a \leq 0$. We consider the commutative diagram
\begin{equation*}
  \xymatrix{
    \mathcal{O}_L(\mathfrak{X}) \otimes_{\mathcal{O}_L(\mathfrak{X}),\varphi_L} \mathfrak{n}(1)^{-a} \mathcal{M}(D)  \ar[d] \ar[r]^-{\overline{\varphi}_{\mathcal{M}(D)}} & \mathfrak{n}(1)^{-a} \mathcal{M}(D) \ar[d]^{\subseteq} \\
    \mathcal{O}_L(\mathfrak{X}) \otimes_{\mathcal{O}_L(\mathfrak{X}),\varphi_L}  \mathcal{M}(D) \ar[r] & I_{Z_0}^a \mathcal{M}(D), }
\end{equation*}
and we pick a function $0 \neq f \in \mathfrak{n}(1)$. Then $\varphi_L(f^{-a}) \neq 0$ in $\mathcal{O}_L(\mathfrak{X})$. Suppose that $s$ is a nonzero element in the kernel of the lower horizontal map. Then $\varphi_L(f^{-a})s$ is a nonzero element in this kernel as well since $\mathcal{M}(D)$ is a projective $\mathcal{O}_L(\mathfrak{X})$-module. But $\varphi_L(f^{-a})s$ lifts to an element in $\mathcal{O}_L(\mathfrak{X}) \otimes_{\mathcal{O}_L(\mathfrak{X}),\varphi_L} \mathfrak{n}(1)^{-a} \mathcal{M}(D)$ which necessarily lies in the kernel of $\overline{\varphi}_{\mathcal{M}(D)}$. This contradicts the injectivity assertion in Prop.\ \ref{varphi-iso}.ii.

iii. This immediately follows from ii. since $\mathcal{O}_L(\mathfrak{X})[Z_0^{-1}]$ is flat over $\mathcal{O}_L(\mathfrak{X})$ by Lemma \ref{Y-1-flat}.
\end{proof}

For any $c \in o_L^\times$ the map $c^*$ is an automorphism of $\mathfrak{X}$ which respects the subset $Z$ as well as the function $n : Z \longrightarrow \ZZ_{\geq 0}$. It is straightforward to conclude from this that $c_* \otimes \id$ induces a semilinear automorphism $c_{\mathcal{M}(D)}$ of $\mathcal{M}(D)$ such that the diagram
\begin{equation}\label{f:commute}
    \xymatrix{
      \mathcal{M}(D) \ar[d]_{c_{\mathcal{M}(D)}} \ar[rr]^{\varphi_{\mathcal{M}(D)}} && I_{Z_0}^a \mathcal{M}(D) \ar[d]^{c_* \otimes c_{\mathcal{M}(D)}} \\
      \mathcal{M}(D) \ar[rr]^{\varphi_{\mathcal{M}(D)}} && I_{Z_0}^a \mathcal{M}(D)   }
\end{equation}
is commutative. Clearly these constructions are functorial in $D$.

Recall that $\Ve(\Fil,\varphi_L)$ denotes the exact category of filtered $\varphi_L$-modules as introduced in Definition \ref{def:fil-vs}. For any $D$ in $\Ve(\Fil,\varphi_L)$ we have constructed a finitely generated projective $\mathcal{O}_L(\mathfrak{X})$-module $\mathcal{M}(D)$ together with an injective $\varphi_L$-linear map
\begin{equation*}
  \varphi_{\mathcal{M}(D)} : \mathcal{M}(D) \longrightarrow \mathcal{M}(D)[Z_0^{-1}]
\end{equation*}
(cf.\ Lemmas \ref{fg-proj} and \ref{varphi}) such that the induced $\mathcal{O}_L(\mathfrak{X})[Z_0^{-1}]$-linear map
\begin{align*}
    \widetilde{\varphi}_{\mathcal{M}(D)} : \mathcal{O}_L(\mathfrak{X})[Z_0^{-1}] \otimes_{\mathcal{O}_L(\mathfrak{X}),\varphi_L} \mathcal{M}(D) & \xrightarrow{\; \cong \;} \mathcal{M}(D)[Z_0^{-1}] \\
    f \otimes s & \longmapsto f \varphi_{\mathcal{M}(D)} (s)
\end{align*}
is an isomorphism (Cor.\ \ref{iso-Z0invert}.iii). We also have the group $\Gamma_L$ acting on $\mathcal{M}(D)$ by semilinear automorphisms $\gamma_{\mathcal{M}(D)}$ such that the diagrams
\begin{equation*}
    \xymatrix{
      \mathcal{M}(D) \ar[d]_{\gamma_{\mathcal{M}(D)}} \ar[rr]^-{\varphi_{\mathcal{M}(D)}} && \mathcal{M}(D)[Z_0^{-1}] \ar[d]^{\gamma_* \otimes \gamma_{\mathcal{M}(D)}} \\
      \mathcal{M}(D) \ar[rr]^-{\varphi_{\mathcal{M}(D)}} &&  \mathcal{M}(D)[Z_0^{-1}]   }
\end{equation*}
are commutative (cf.\ \eqref{f:commute}). Being a finitely generated module $\mathcal{M}(D)$ carries a natural Fr\'echet topology.

\begin{lemma}\label{action-MD}
The $\Gamma_L$-action on $\mathcal{M}(D)$ is continuous and differentiable, and the derived $\Lie(\Gamma_L)$-action is $L$-bilinear.
\end{lemma}
\begin{proof}
By the proof of Lemma \ref{fg-proj} the module $\mathcal{M}(D)$ is a closed submodule of the finitely generated projective module $I_Z^{-b} \otimes_L D$ for some $b \geq 0$. This reduces us to showing that the $\Gamma_L$-action on $I_Z^{-b}$ has all the asserted properties. Multiplication by the function $(\log_{\mathfrak{X}})^b$ induces a continuous bijection $I_Z^{-b} \longrightarrow \mathfrak{n}(1)^b$. By the open mapping theorem it has to be a topological isomorphism. By Lemma \ref{zeros}.ii the $\Gamma_L$-action on $I_Z^{-b}$ corresponds to the natural $\Gamma_L$-action on the closed ideal $\mathfrak{n}(1)^b$ in $\mathcal{O}_L(\mathfrak{X})$ twisted by the locally $L$-analytic character $\Gamma_L = o_L^\times \longrightarrow L^\times$ sending $\gamma$ to $\gamma^{-b}$. By the discussion after Lemma \ref{smalldisk-locan} the $\Gamma_L$-action on $\mathcal{O}_L(\mathfrak{X})$ and a fortiori the natural action on the closed ideal $\mathfrak{n}(1)^b$ and hence the twisted action all have the wanted properties.
\end{proof}

There is one additional important property. By Lemma \ref{stalks} we have $\widetilde{\mathcal{M}}_1(D) = \mathcal{O}_1 \otimes_L D$. We see that the induced $\Gamma_L$-action (note that the monoid action on $\mathfrak{X}$ fixes the point $1$) on $\mathcal{M}(D)/\mathfrak{n}(1)\mathcal{M}(D) = \widetilde{\mathcal{M}}_1(D)/\mathfrak{m}_1 \widetilde{\mathcal{M}}_1(D) = D$ is trivial.

\begin{definition}
\label{defmodw}
We therefore introduce the category $\Mod^{\varphi_L,\Gamma_L,\mathrm{an}}_{/\mathfrak{X}}$ of finitely generated projective $\mathcal{O}_L(\mathfrak{X})$-modules
$M$ equipped with an injective $\varphi_L$-linear map
\begin{equation*}
  \varphi_M: M \longrightarrow M[Z_0^{-1}]
\end{equation*}
as well as a continuous (w.r.t.\ the natural Fr\'echet topology on $M$) action of $\Gamma_L$ by semilinear automorphisms $\gamma_M$ such that the following properties are satisfied:
\begin{itemize}
  \item[a)] The induced $\mathcal{O}_L(\mathfrak{X})[Z_0^{-1}]$-linear map
\begin{align*}
    \widetilde{\varphi}_M : \mathcal{O}_L(\mathfrak{X})[Z_0^{-1}] \otimes_{\mathcal{O}_L(\mathfrak{X}),\varphi_L} M & \xrightarrow{\; \cong \;} M[Z_0^{-1}] \\
    f \otimes s & \longmapsto f \varphi_M (s)
\end{align*}
is an isomorphism.
  \item[b)] For any $\gamma \in \Gamma_L$ the diagram
  \begin{equation*}
    \xymatrix{
       M\ar[d]_{\gamma_M} \ar[rr]^-{\varphi_M} && M[Z_0^{-1}] \ar[d]^{\gamma_* \otimes \gamma_M} \\
       M \ar[rr]^-{\varphi_M} &&  M[Z_0^{-1}]   }
\end{equation*}
is commutative.
  \item[c)] The induced $\Gamma_L$-action on $M/\mathfrak{n}(1)M$ is trivial.
\end{itemize}
\end{definition}

\begin{remark}\label{phi-M-bounded}
There are integers $a \leq 0 \leq c$ such that the map $\widetilde{\varphi}_M$ restricts to an injective map $\mathcal{O}_L(\mathfrak{X}) \otimes_{\mathcal{O}_L(\mathfrak{X}),\varphi_L} M \longrightarrow I_{Z_0}^a M$ whose cokernel is annihilated by $I_{Z_0}^c$.
\end{remark}
\begin{proof}
Since $M$ is finitely generated we find an $a \leq 0$ such that $\varphi_M(M) \subseteq I_{Z_0}^a M$. Since $M$ is projective we have $\mathcal{O}_L(\mathfrak{X}) \otimes_{\mathcal{O}_L(\mathfrak{X}),\varphi_L} M \subseteq \mathcal{O}_L(\mathfrak{X})[Z_0^{-1}] \otimes_{\mathcal{O}_L(\mathfrak{X}),\varphi_L} M$. Hence the injectivity follows from condition a). Tensoring the resulting map $\mathcal{O}_L(\mathfrak{X}) \otimes_{\mathcal{O}_L(\mathfrak{X}),\varphi_L} M \longrightarrow I_{Z_0}^a M$ with $\mathcal{O}_L(\mathfrak{X})[Z_0^{-1}]$ gives back the isomorphism $\widetilde{\varphi}_M$ (recall that $\mathcal{O}_L(\mathfrak{X})[Z_0^{-1}]$ is flat over $\mathcal{O}_L(\mathfrak{X})$ by Lemma \ref{Y-1-flat}). Hence the cokernel $N$ of this map satisfies $N[Z_0^{-1}] = 0$. For any function $0 \neq f \in I_{Z_0}$ we have $I_{Z_0}^{-1} \subseteq \mathcal{O}_L(\mathfrak{X}) f^{-1}$ and consequently that the localization $N_f$ vanishes. Since $N$ is finitely generated we find a $c(f) \geq 0$ such that $f^{c(f)} N = 0$. Finally, since $I_{Z_0}$ is finitely generated, we must have a $c \geq 0$ such that $I_{Z_0}^c N = 0$.
\end{proof}

\begin{lemma}\label{M-diff}
The $\Gamma_L$-action on $M$ is differentiable.
\end{lemma}
\begin{proof}
We put $M_n := \mathcal{O}_L(\mathfrak{X}_n) \otimes_{\mathcal{O}_L(\mathfrak{X})} M$. The $\Gamma_L$-action extends semilinearly to an action on each $M_n$ by automorphisms $\gamma_{M_n} := \gamma_* \otimes \gamma_M$. Since $M_n$ is a finitely generated projective $\mathcal{O}_L(\mathfrak{X}_n)$-module (and $\Gamma_L$ acts continuously on $\mathcal{O}_L(\mathfrak{X}_n)$) any individual $\gamma_{M_n}$ is a continuous automorphism of the Banach module $M_n$. On the other hand, the orbit maps for $M_n$ are the composite of the, by assumption, continuous orbit maps for $M$ and the continuous map $M \longrightarrow M_n$ and hence are continuous. It therefore follows from the nonarchimedean Banach-Steinhaus theorem that the $\Gamma_L$-action on each $M_n$ is (jointly) continuous. From the discussion after Lemma \ref{smalldisk-locan} we know that the $\Gamma_L$-action on $\mathcal{O}_L(\mathfrak{X}_n)$ satisfies the condition \eqref{f:condition-locan}. We then conclude from Prop.\ \ref{locan3} that the $\Gamma_L$-action on $M_n$ is locally $\Qp$-analytic. The asserted differentiability now follows from the fact that the Fr\'echet module $M$ is the projective limit of the Banach modules $M_n$.
\end{proof}

As we have seen above we then have the well defined functor
\begin{align*}
  \mathcal{M} : \Ve(\Fil,\varphi_L) & \longrightarrow \Mod^{\varphi_L,\Gamma_L,\mathrm{an}}_{/\mathfrak{X}} \\
  D & \longmapsto \mathcal{M}(D) \ .
\end{align*}

On the other hand, for any $M$ in $\Mod^{\varphi_L,\Gamma_L,\mathrm{an}}_{/\mathfrak{X}}$ we have the finite dimensional $L$-vector space
\begin{equation*}
  \mathcal{D}(M) := M/\mathfrak{n}(1)M \ .
\end{equation*}
Since $M/\mathfrak{n}(1)M =
M[Z_0^{-1}]/\mathfrak{n}(1)M[Z_0^{-1}]$ the map $\varphi_M$ induces an $L$-linear endomorphism of $M/\mathfrak{n}(1)M$ denoted by $\varphi_{\mathcal{D}(M)}$. The condition a) implies that $\varphi_{\mathcal{D}(M)}$ is surjective and hence is an automorphism.

By Lemma \ref{M-diff} we have available the operator $\nabla_M$ on $M$ corresponding to the derived action of the element $1 \in \Lie(\Gamma_L)$.

In $\mathfrak{X}$ we have the Zariski and hence admissible open subvariety $\mathfrak{X} \setminus Z$, which is preserved by $\pi_L^*$ and by the action of $\Gamma_L$. Therefore $\varphi_M$ and the action of $\Gamma_L$ extend semilinearly to $M\{Z^{-1}\} := \mathcal{O}_L(\mathfrak{X} \setminus Z) \otimes_{\mathcal{O}_L(\mathfrak{X})} M$. The same argument as for \eqref{f:tilde-varphi-iso} shows that $\mathcal{O}_L(\mathfrak{X})[Z^{-1}] \subseteq \mathcal{O}_L(\mathfrak{X} \setminus Z)$. Hence the condition a) implies that the map
\begin{equation*}
\id \otimes \varphi_M : \mathcal{O}_L(\mathfrak{X} \setminus Z) \otimes_{\mathcal{O}_L(\mathfrak{X}),\varphi_L} M  = \mathcal{O}_L(\mathfrak{X} \setminus Z) \otimes_{\mathcal{O}_L(\mathfrak{X} \setminus Z),\varphi_L} M\{Z^{-1}\} \xrightarrow{\; \cong \;} M\{Z^{-1}\}
\end{equation*}
is an isomorphism.

\begin{proposition}\label{flat}
We have $\mathcal{O}_L(\mathfrak{X} \setminus Z) \otimes_L M\{Z^{-1}\}^{\Gamma_L} = M\{Z^{-1}\}$. In particular, the projection map $M\{Z^{-1}\} \longrightarrow \mathcal{D}(M)$ restricts to an isomorphism $M\{Z^{-1}\}^{\Gamma_L} \xrightarrow{\cong} \mathcal{D}(M)$ such that the diagram
\begin{equation*}
    \xymatrix{
      M\{Z^{-1}\}^{\Gamma_L} \ar[d]_{\varphi_M} \ar[r]^-{\cong} & \mathcal{D}(M) \ar[d]^{\varphi_{\mathcal{D}(M)}} \\
      M\{Z^{-1}\}^{\Gamma_L} \ar[r]^-{\cong} & \mathcal{D}(M)   }
\end{equation*}
is commutative. In fact, we have
\begin{equation*}
  M[Z^{-1}]^{\Gamma_L} = M\{Z^{-1}\}^{\Gamma_L} \qquad\text{and}\qquad \mathcal{O}_L(\mathfrak{X})[Z^{-1}] \otimes_L M[Z^{-1}]^{\Gamma_L} = M[Z^{-1}] \ .
\end{equation*}
\end{proposition}
\begin{proof}
Let $r \in (0,1) \cap p^\QQ$. We put
\begin{align*}
  & M(r) := \mathcal{O}_L(\mathfrak{X}(r)) \otimes_{\mathcal{O}_L(\mathfrak{X})} M \; \text{and} \\
  & M(r)\{Z^{-1}\} := \mathcal{O}_L(\mathfrak{X}(r) \setminus Z) \otimes_{\mathcal{O}_L(\mathfrak{X})} M = \mathcal{O}_L(\mathfrak{X}(r) \setminus Z) \otimes_{\mathcal{O}_L(\mathfrak{X}(r))} M(r) \ .
\end{align*}
We semilinearly extend the $\Gamma_L$-action to $M(r)$ and $M(r)\{Z^{-1}\}$. At first we suppose that $r < p^{-\frac{1}{p-1}}$. Using Lemma \ref{smalldisk-locan} the same reasoning as in the proof of Lemma \ref{M-diff} shows that this $\Gamma_L$-action on $M(r)$ is locally $\Qp$-analytic. In particular, we have the operator $\nabla_{M(r)}$. By possibly passing to a smaller $r$ we may assume that $M(r)$ is free over $\mathcal{O}_L(\mathfrak{X}(r))$. By Lemma \ref{small-disk} we may view $M(r)$ as a finitely generated free differential module over $\mathcal{O}_L(\mathfrak{B}(r))$ for the derivation $y \frac{d}{dy}$. The $p$-adic Fuchs theorem for disks (\cite{Ked} Thm.\ 13.2.2) then implies, again by possibly passing to a smaller $r$, that the $L$-vector space $M(r)^{\nabla_{M(r)}=0}$ contains a basis of the $\mathcal{O}_L(\mathfrak{X}(r))$-module $M(r)$ (observe that, as a consequence of the condition c), the constant matrix of this differential module is the zero matrix). By \cite{Ked} Lemma 5.1.5 (applied to the field of fractions of $\mathcal{O}_L(\mathfrak{X}(r))$) we, in fact, obtain that
\begin{equation*}
    \mathcal{O}_L(\mathfrak{X}(r)) \otimes_L M(r)^{\nabla_{M(r)}=0} = M(r) \ .
\end{equation*}
It follows that the projection map restricts to an isomorphism
\begin{equation*}
    M(r)^{\nabla_{M(r)}=0} \xrightarrow{\;\cong\;} M(r)/\mathfrak{n}(1) M(r) = M/\mathfrak{n}(1) M \ .
\end{equation*}
Since, on the one hand, $\Gamma_L$ acts trivially on $M/\mathfrak{n}(1) M$ by condition c) and, on the other hand, the $\Gamma_L$-action commutes with $\nabla_{M(r)}$ it further follows that $M(r)^{\Gamma_L} = M(r)^{\nabla_{M(r)}=0}$ and hence that
\begin{equation*}
    \mathcal{O}_L(\mathfrak{X}(r)) \otimes_L M(r)^{\Gamma_L} = M(r) \qquad\text{and}\qquad   M(r)^{\Gamma_L} \xrightarrow{\;\cong\;} M/ \mathfrak{n}(1) M \ .
\end{equation*}
Observe that $\mathfrak{X}(r) \cap Z = \emptyset$ and hence $M(r) = M(r)\{Z^{-1}\}$.

We now choose a sequence $r = r_0 < r_1 < \ldots < r_m < \ldots < 1$ in $p^\QQ$ converging to $1$ such that $p^* (\mathfrak{X}(r_{m+1})) \subseteq \mathfrak{X}(r_m)$ for any $m \geq 0$. By induction we assume that
\begin{equation*}
    \mathcal{O}_L(\mathfrak{X}(r_m) \setminus Z) \otimes_L M(r_m)\{Z^{-1}\}^{\Gamma_L} = M(r_m)\{Z^{-1}\}
\end{equation*}
holds true. We temporarily denote by $\phi$ the ring endomorphism of $\mathcal{O}_L(\mathfrak{X} \setminus Z)$ induced by the morphism $p^* : \mathfrak{X} \longrightarrow \mathfrak{X}$. It extends to a ring homomorphism $\phi : \mathcal{O}_L(\mathfrak{X}(r_m) \setminus Z) \longrightarrow \mathcal{O}_L(\mathfrak{X}(r_{m+1}) \setminus Z)$. As a consequence of the condition a) we have the $\mathcal{O}_L(\mathfrak{X} \setminus Z)$-linear isomorphism
\begin{align*}
    \mathcal{O}_L(\mathfrak{X} \setminus Z) \otimes_{\mathcal{O}_L(\mathfrak{X}),\phi} M & \xrightarrow{\;\cong\;} M\{Z^{-1}\} \\
    f \otimes s & \longmapsto f p(s)
\end{align*}
which base changes to the $\mathcal{O}_L(\mathfrak{X}(r_{m+1}) \setminus Z)$-linear isomorphism
\begin{multline*}
    \mathcal{O}_L(\mathfrak{X}(r_{m+1}) \setminus Z) \otimes_L M(r_m)\{Z^{-1}\}^{\Gamma_L} =
    \mathcal{O}_L(\mathfrak{X}(r_{m+1}) \setminus Z) \otimes_{\mathcal{O}_L(\mathfrak{X}(r_m) \setminus Z),\phi} M(r_m)\{Z^{-1}\} \\ = \mathcal{O}_L(\mathfrak{X}(r_{m+1}) \setminus Z) \otimes_{\mathcal{O}_L(\mathfrak{X}),\phi} M  \xrightarrow{\;\cong\;}   M(r_{m+1})\{Z^{-1}\},
\end{multline*}
where the first identity is our induction hypothesis, denoted by $\alpha_m$. By restricting to $\Gamma_L$-invariants the latter map fits into the diagram
\begin{equation*}
  \xymatrix{
     M(r_m)\{Z^{-1}\}^{\Gamma_L} \ar[d]_{\cong} \ar[r]^-{\alpha_m} & M(r_{m+1})\{Z^{-1}\}^{\Gamma_L} \ar[d] \ar@{^{(}->}[r]^{\mathrm{res}} & M(r_m)\{Z^{-1}\}^{\Gamma_L} \ar[dl]^{\cong} \\
    M/\mathfrak{n}(1) M \ar[r]^{\cong}_{\varphi_{\mathcal{D}(M)}^e} &  M/\mathfrak{n}(1) M &    }
\end{equation*}
where $e$ denotes the ramification index of $L/\Qp$. The left square is commutative as a consequence of conditions b) and c). The right upper horizontal arrow is given by restriction from $\mathfrak{X}(r_{m+1})$ to $\mathfrak{X}(r_m)$; it is injective since $M$ is projective. The right part of the diagram is trivially commutative. The indicated perpendicular isomorphisms hold by the induction hypothesis. It follows that all arrows in the diagram are isomorphisms. We see that the original isomorphism $\alpha_m$ sends $M(r_m)\{Z^{-1}\}^{\Gamma_L}$ isomorphically onto $M(r_{m+1})\{Z^{-1}\}^{\Gamma_L}$ which implies that
\begin{equation*}
    \mathcal{O}_L(\mathfrak{X}(r_{m+1}) \setminus Z) \otimes_L M(r_{m+1})\{Z^{-1}\}^{\Gamma_L} = M(r_{m+1})\{Z^{-1}\} \ .
\end{equation*}
Moreover, the restriction maps induce isomorphisms $M(r_{m+1})\{Z^{-1}\}^{\Gamma_L} \cong M(r_m)\{Z^{-1}\}^{\Gamma_L}$ as well.

Let $\mathfrak{M}$ denote the coherent $\mathcal{O}_\mathfrak{X}$-module with global sections $M$. According to Remark \ref{coherent}.iv we have $\mathfrak{M}(\mathfrak{X} \setminus Z) = M\{Z^{-1}\}$ and $\mathfrak{M}(\mathfrak{X}(r_m) \setminus Z) = M(r_m)\{Z^{-1}\}$. Since the $\mathfrak{X}(r_m) \setminus Z$ form an admissible covering of $\mathfrak{X} \setminus Z$ we have $\mathfrak{M}(\mathfrak{X} \setminus Z) = \varprojlim_m \mathfrak{M}(\mathfrak{X}(r_m) \setminus Z)$. It follows first that $M\{Z^{-1}\}^{\Gamma_L} = \varprojlim_m M(r_m)\{Z^{-1}\}^{\Gamma_L} \cong M(r_n)\{Z^{-1}\}^{\Gamma_L}$ for any $n$, hence $\mathcal{O}_L(\mathfrak{X}(r_m) \setminus Z) \otimes_L M\{Z^{-1}\}^{\Gamma_L} = M(r_m)\{Z^{-1}\}$, and finally, by passing to the projective limit in the latter identity, that $\mathcal{O}_L(\mathfrak{X} \setminus Z) \otimes_L M\{Z^{-1}\}^{\Gamma_L} = M\{Z^{-1}\}$.

For any $j \geq 0$ we set $Z_j := \{x \in Z : n(x) \leq j\} = \mathfrak{X}[\pi_L^{j+1}]$, and $Z_{-1} := \emptyset$. Lemma \ref{unramified} implies that
\begin{equation*}
  \mathcal{O}_L(\mathfrak{X}) \varphi_L(I_{Z_j}) = I_{Z_{j+1} \setminus Z_0} \qquad\text{and}\qquad \mathcal{O}_L(\mathfrak{X}) \varphi_L(I_{Z \setminus Z_j}) = I_{Z \setminus Z_{j+1}} \ .
\end{equation*}
If $e$ denotes the ramification index of $L/\Qp$ then we have $p = \pi_L^e u$ for some $u \in o_L^\times$. We deduce inductively that
\begin{equation*}
  \mathcal{O}_L(\mathfrak{X}) \phi(I_Z) = \mathcal{O}_L(\mathfrak{X}) \varphi_L^e(I_Z) = I_{Z \setminus Z_{e-1}} \ .
\end{equation*}
For the more precise form of the assertion we use Remark \ref{phi-M-bounded} with the integers $a \leq 0 \leq c$. Again we deduce inductively that
\begin{equation*}
  p(M) = \varphi_M^e(M) \subseteq \varphi_M(I_{Z_{e-2}}^a M) = \varphi_L(I_{Z_{e-2}}^a)\varphi_M(M) \subseteq I_{Z_{e-1} \setminus Z_0}^a I_{Z_0}^a M = I_{Z_{e-1}}^a M
\end{equation*}
and
\begin{align*}
  \mathcal{O}_L(\mathfrak{X})p(M) & = \mathcal{O}_L(\mathfrak{X}) \varphi_M^e(M) = \mathcal{O}_L(\mathfrak{X}) \varphi_M(\mathcal{O}_L(\mathfrak{X}) \varphi_M^{e-1}(M)) \\
   & \supseteq \mathcal{O}_L(\mathfrak{X}) \varphi_M(I_{Z_{e-2}}^{a+c} M) = I_{Z_{e-1} \setminus Z_0}^{a+c} \varphi_M(M) \supseteq I_{Z_{e-1} \setminus Z_0}^{a+c} I_{Z_0}^{a+c} M \\
   & = I_{Z_{e-1}}^{a+c} M \ .
\end{align*}
For the identity $M[Z^{-1}]^{\Gamma_L} = M\{Z^{-1}\}^{\Gamma_L}$ it suffices to show that $I_Z^{-a} \cdot M\{Z^{-1}\}^{\Gamma_L} \subseteq M$ which further reduces to the claim that $I_Z^{-a} \cdot M(r_m)\{Z^{-1}\}^{\Gamma_L} \subseteq M(r_m)$ for any $m \geq 0$. This holds trivially true for $m = 0$. Any element in $M(r_{m+1})\{Z^{-1}\}^{\Gamma_L}$ can be written as $\alpha_m(s)$ for some $s \in M(r_m)\{Z^{-1}\}^{\Gamma_L}$. By induction we may assume that $I_Z^{-a}s \subseteq M(r_m)$. Then $\phi(I_Z^{-a})\alpha_m(s) = p(I_Z^{-a}s) \subseteq p(M(r_m)) \subseteq I_{Z_{e-1}}^a M(r_{m+1})$ and hence
\begin{equation*}
  M(r_{m+1}) \supseteq I_{Z_{e-1}}^{-a} \phi(I_Z^{-a})\alpha_m(s) = I_{Z_{e-1}}^{-a} I_{Z \setminus Z_{e-1}}^{-a} \alpha_m(s) = I_Z^{-a} \alpha_m(s) \ .
\end{equation*}
Next we consider the commutative diagram
\begin{equation*}
  \xymatrix{
    \mathcal{O}_L(\mathfrak{X})[Z^{-1}] \otimes_L M[Z^{-1}]^{\Gamma_L} \ar[d]_{\subseteq\; \otimes\; =} \ar[r]^-{\subseteq} & M[Z^{-1}] \ar[d]^{\subseteq} \\
    \mathcal{O}_L(\mathfrak{X} \setminus Z) \otimes_L M\{Z^{-1}\}^{\Gamma_L} \ar[r]^-{=} & M\{Z^{-1}\}   }
\end{equation*}
where the right vertical inclusion comes from the projectivity of $M$. We have seen that $\mathcal{O}_L(\mathfrak{X}) \otimes_L M[Z^{-1}]^{\Gamma_L} \subseteq I_Z^a M$. In order to recognize the upper horizontal arrow as an identity it suffices to show that $I_Z^{a+c} M \subseteq \mathcal{O}_L(\mathfrak{X}) \otimes_L M[Z^{-1}]^{\Gamma_L}$. This reduces to the claim that $I_Z^{a+c} M(r_m) \subseteq \mathcal{O}_L(\mathfrak{X}(r_m)) \otimes_L M(r_m)[Z^{-1}]^{\Gamma_L}$ for any $m \geq 0$. This is clear for $m=0$. By induction we compute
\begin{align*}
  \mathcal{O}_L(\mathfrak{X}(r_{m+1})) \otimes_L M(r_{m+1})[Z^{-1}]^{\Gamma_L} & = \mathcal{O}_L(\mathfrak{X}(r_{m+1})) \otimes_L \alpha_m(M(r_m)[Z^{-1}]^{\Gamma_L}) \\
   & = \mathcal{O}_L(\mathfrak{X}(r_{m+1})) p(\mathcal{O}_L(\mathfrak{X}(r_m)) \otimes_L M(r_m)[Z^{-1}]^{\Gamma_L}) \\
   & \supseteq \mathcal{O}_L(\mathfrak{X}(r_{m+1})) p(I_Z^{a+c} M(r_m)) \\
   & = \mathcal{O}_L(\mathfrak{X}(r_{m+1})) \phi(I_Z^{a+c}) p(M(r_m))   \\
   & = \mathcal{O}_L(\mathfrak{X}(r_{m+1})) I_{Z \setminus Z_{e-1}}^{a+c} p(M(r_m)) \\
   & \supseteq I_{Z \setminus Z_{e-1}}^{a+c} I_{Z_{e-1}}^{a+c} M(r_{m+1}) = I_Z^{a+c} M(r_{m+1}) \ .
\end{align*}
\end{proof}

\begin{remark}\label{L-bilinear}
The derived $\Lie(\Gamma_L)$-action on any $M$ in $\Mod^{\varphi_L,\Gamma_L,\mathrm{an}}_{/\mathfrak{X}}$  is $L$-bilinear.
\end{remark}
\begin{proof}
We use the notations in the proof of Prop.\ \ref{flat}, in particular, the radius $r = r_0$. Since $M \subseteq M(r)$ it suffices to prove the analogous assertion for $M(r)$. But in that proof we had seen that $M(r) = \mathcal{O}_L(\mathfrak{X}(r)) \otimes_L M(r)^{\Gamma_L}$ which further reduces us to the case of the derived action on $\mathcal{O}_L(\mathfrak{X}(r))$. This case was treated in Lemma \ref{smalldisk-locan}.
\end{proof}

We now define a filtration on $\mathcal{D}(M)$. In fact, using the isomorphism in Prop.\ \ref{flat} we define the filtration on the isomorphic $M[Z^{-1}]^{\Gamma_L}$. We pick a point $x_0 \in Z_0$. Passing to the germ in $x_0$ gives an injective homomorphism $M \longrightarrow \mathcal{O}_{x_0} \otimes_{\mathcal{O}_L(\mathfrak{X})} M$. This map extends to a homomorphism $M[Z^{-1}] \longrightarrow \Fr(\mathcal{O}_{x_0}) \otimes_{\mathcal{O}_L(\mathfrak{X})} M$, which still is injective, leading to the commutative diagram
\begin{equation*}
    \xymatrix{
                   & M \ar[d] \ar[r] & \mathcal{O}_{x_0} \otimes_{\mathcal{O}_L(\mathfrak{X})} M \ar[d]^{\subseteq} \\
      M[Z^{-1}]^{\Gamma_L} \ar[r]^{\subseteq} & M[Z^{-1}] \ar[r] & \Fr(\mathcal{O}_{x_0}) \otimes_{\mathcal{O}_L(\mathfrak{X})} M.   }
\end{equation*}
The $\mathfrak{m}_{x_0}$-adic filtration on $\Fr(\mathcal{O}_{x_0}) \otimes_{\mathcal{O}_L(\mathfrak{X})} M$ induces, via the injective composition of the lower horizontal arrows, an exhaustive and separated filtration $\Fil^\bullet M[Z^{-1}]^{\Gamma_L}$ on $M[Z^{-1}]^{\Gamma_L}$ which we transport to a filtration on $\mathcal{D}(M)$. Since $\Gamma_L$ acts transitively on $Z_0$ this filtration is independent of the choice of $x_0$. In this way we obtain a functor
\begin{align*}
  \mathcal{D} : \Mod^{\varphi_L,\Gamma_L,\mathrm{an}}_{/\mathfrak{X}}  & \longrightarrow \Ve(\Fil,\varphi_L)  \\
  M & \longmapsto \mathcal{D}(M) \ .
\end{align*}

\begin{theorem}\label{Wach-equiv}
The functors $\mathcal{M}$ and $\mathcal{D}$ are quasi-inverse equivalences between the categories $\Ve(\Fil,\varphi_L)$  and $\Mod^{\varphi_L,\Gamma_L,\mathrm{an}}_{/\mathfrak{X}}$.
\end{theorem}
\begin{proof}
Given any $D$ in $\Ve(\Fil,\varphi_L)$ we put $M := \mathcal{M}(D)$. It follows from \eqref{f:bounds} and Lemma \ref{Y-1-flat} that
\begin{equation*}
    D \subseteq \mathcal{M}(D)[Z^{-1}]^{\Gamma_L} = M[Z^{-1}]^{\Gamma_L} \subseteq M\{Z^{-1}\}^{\Gamma_L} \cong \mathcal{D}(M) \ .
\end{equation*}
By Lemma \ref{fg-proj} the rank of $M$ is equal to the dimension of $D$. Hence the outer terms in the above chain have the same dimension, and so we have to have equality everywhere. Obviously, $\varphi_M$ induces $\varphi_D$. As a consequence of Lemma \ref{stalks} we have
\begin{equation*}
    \mathfrak{m}_{x_0}^i \otimes_{\mathcal{O}_L(\mathfrak{X})} M  = \mathfrak{m}_{x_0}^i (\mathcal{O}_{x_0} \otimes_{\mathcal{O}_L(\mathfrak{X})} M) = \mathfrak{m}_{x_0}^i \Fil^0(\Fr(\mathcal{O}_{x_0}) \otimes_L D) = \Fil^i(\Fr(\mathcal{O}_{x_0}) \otimes_L D)
\end{equation*}
for any $i \in \ZZ$. This shows that the $\mathfrak{m}_{x_0}^i$-adic filtration on $\Fr (\mathcal{O}_{x_0}) \otimes_{\mathcal{O}_L(\mathfrak{X})} M$ coincides with the tensor product filtration on $\Fr(\mathcal{O}_{x_0}) \otimes_L D$ and therefore induces the original filtration on $D$. We conclude that the above isomorphism $D \cong \mathcal{D}(\mathcal{M}(D))$ is a natural isomorphism in the category
$\Ve(\Fil,\varphi_L)$.

Now consider any $M$ in $\Mod^{\varphi_L,\Gamma_L,\mathrm{an}}_{/\mathfrak{X}}$ and $D' := M\{Z^{-1}\}^{\Gamma_L} = M[Z^{-1}]^{\Gamma_L}$ in $\Ve(\Fil,\varphi_L)$. Using Prop.\ \ref{flat} we see that
\begin{equation*}
  \mathcal{M}(D') = \{s \in M[Z^{-1}] : (\id \otimes \varphi_{D'}^{-n(x)})(s) \in \Fil^0(\Fr(\mathcal{O}_x) \otimes_L D')\ \text{for any $x \in Z$}\} \ .
\end{equation*}
The identity $\mathcal{O}_L(\mathfrak{X})[Z^{-1}] \otimes_L D' = \mathcal{O}_L(\mathfrak{X})[Z^{-1}] \otimes_{\mathcal{O}_L(\mathfrak{X})} M$ gives rise, for any $x \in Z$, to the identity $\Fr(\mathcal{O}_x) \otimes_L D' = \Fr(\mathcal{O}_x) \otimes_{\mathcal{O}_L(\mathfrak{X})} M$. We claim that
\begin{equation*}
  (\id \otimes \varphi_{D'}^{-n(x)})(s) \in \Fil^0(\Fr(\mathcal{O}_x) \otimes_L D') \qquad\text{if and only if}\qquad s \in \mathcal{O}_x \otimes_{\mathcal{O}_L(\mathfrak{X})} M \ .
\end{equation*}
This shows that $\mathcal{M}(D') = M$. It is straightforward to see that this identity is compatible with the monoid actions on both sides. Hence we obtain a natural isomorphism $M \cong \mathcal{M}(\mathcal{D}(M))$ in the category $\Mod^{\varphi_L,\Gamma_L,\mathrm{an}}_{/\mathfrak{X}}$. To establish this claim we first consider the case $x \in Z_0$. We need to verify that $\Fil^0(\Fr(\mathcal{O}_x) \otimes_L D') = \mathcal{O}_x \otimes_{\mathcal{O}_L(\mathfrak{X})} M$. But this is a special case of the subsequent Lemma \ref{ED}. For a general $x \in Z$ we put $x_0 := (\pi_L^*)^{n(x)}(x) \in Z_0$. Then $(\id \otimes \varphi_{D'}^{-n(x)})(s) \in \Fil^0(\Fr(\mathcal{O}_x) \otimes_L D')$ if and only if $s$ lies in
\begin{align*}
  (\id \otimes \varphi_{D'}^{n(x)}) \Fil^0(\Fr(\mathcal{O}_x) \otimes_L D') & = \sum_i \mathfrak{m}_x^i \otimes_L \varphi_{D'}^{n(x)} (\Fil^{-i} D')  \\
   & = \sum_i \mathcal{O}_x \varphi_L^{n(x)}(\mathfrak{m}_{x_0}^i) \otimes_L \varphi_{D'}^{n(x)} (\Fil^{-i} D') \\
   & = \mathcal{O}_x (\varphi_L \otimes \varphi_M)^{n(x)} \Fil^0(\Fr(\mathcal{O}_{x_0}) \otimes_L D') \\
   & = \mathcal{O}_x (\varphi_L \otimes \varphi_M)^{n(x)} (\mathcal{O}_{x_0} \otimes_{\mathcal{O}_L(\mathfrak{X})} M) \\
   & = \mathcal{O}_x \otimes_{\mathcal{O}_L(\mathfrak{X})} M \ .
\end{align*}
Here the second, resp.\ fourth, resp.\ last, identity uses that $\varphi_L$ is unramified (Lemma \ref{unramified}.ii), resp.\ the previous case of points in $Z_0$, resp.\ an iteration of Remark \ref{phi-M-bounded}.
\end{proof}

Consider any point $x \in Z_0$, which is the Galois orbit of a character $\chi$ of $o_L$. Since $\chi$ has values in the group $\mu_p$ of $p$th roots of unity, the residue class field $L_x$ of $\mathcal{O}_x$ is equal to $L_x = L(\mu_p)$, and we have $x = \Gal(L_x/L) \chi$. On the other hand, by Lemma \ref{torsion} we have $\mathfrak{X}(\pi_L)(\Cp) \cong k_L$ as $o_L$-modules, where $k_L$ denotes the residue field of $o_L$. We see that $\Gamma_L$ acts on $Z_0(\Cp)$ through its factor group $k_L^\times$. The stabilizer $\Gamma_x$ of $x$ satisfies $1 + \pi_L o_L \subseteq \Gamma_x \subseteq \Gamma_L$. Since the $\Gamma_L/1+\pi_L o_L = k_L^\times$-action on $Z_0(\Cp)$ is simply transitive there is a unique isomorphism $\sigma_x : \Gamma_x/1+\pi_L o_L \xrightarrow{\cong} \Gal(L_x/L)$ such that $\gamma^*(\chi) = \sigma_x(\gamma)(\chi)$ for any $\gamma \in \Gamma_x$. On the other hand the $\Gamma_L$-action on $\mathfrak{X}$ induces a $\Gamma_x$-action on $\mathcal{O}_x$ and hence on $L_x$. The identities
\begin{equation*}
  (\gamma_*(f))(\chi) = f(\gamma^*(\chi)) = f(\sigma_x(\gamma)(\chi)) = \sigma_x(\gamma)(f(\chi)) \qquad\text{for any $f \in \mathcal{O}_x$ and $\gamma \in \Gamma_x$}
\end{equation*}
show that this last action is given by the homomorphism $\sigma_x : \Gamma_x \longrightarrow \Gal(L_x/L)$.

\begin{lemma}\label{ED}
Let $x \in Z_0$, let $V$ be a finite dimensional $L$-vector space, and equip $\Fr(\mathcal{O}_x) \otimes_L V$ with the $\Gamma_x$-action through the first factor. Suppose given any $\Gamma_x$-invariant $\mathcal{O}_x$-lattice $\mathcal{L} \subseteq \Fr(\mathcal{O}_x) \otimes_L V$ and define $\Fil^i V := V \cap \mathfrak{m}_x^i \mathcal{L}$ for any $i \in \ZZ$. We then have
\begin{equation*}
  \mathcal{L} = \sum_{i \in \ZZ} \mathfrak{m}_x^i \otimes_L \Fil^{-i} V \ .
\end{equation*}
\end{lemma}
\begin{proof}
Let $\ell_x \in \mathcal{O}_x$ denote the germ of $\log_\mathfrak{X}$. By Lemma \ref{zeros} we have $\mathfrak{m}_x = \ell_x \mathcal{O}_x$ and $\gamma_*(\ell_x) = \gamma \cdot \ell_x$ for any $\gamma \in \Gamma_x$. We find integers $n \leq 0 \leq m$ such that
\begin{equation*}
  \ell_x^m \mathcal{O}_x \otimes_L V \subseteq \mathcal{L} \subseteq \ell_x^n \mathcal{O}_x \otimes_L V \ .
\end{equation*}
Let $\widehat{\mathcal{O}}_x = L_x [[\ell_x]]$ be the $\mathfrak{m}_x$-adic completion of $\mathcal{O}_x$. We then have the $\Gamma_x$-invariant decomposition
\begin{equation*}
  \ell_x^n \mathcal{O}_x / \ell_x^m \mathcal{O}_x = \ell_x^n \widehat{\mathcal{O}}_x / \ell_x^m \widehat{\mathcal{O}}_x = \ell_x^n L_x \oplus \ell_x^{n+1} L_x \oplus \ldots \oplus \ell_x^{m-1} L_x \ .
\end{equation*}
The $\Gamma_x$-action on $\ell_x^j L_x$ is given by $\gamma_*(\ell_x^j c) = \ell_x^j \gamma^j \cdot \sigma_x(\gamma)(c)$. Since $\Gal(L_x/L)$ acts semisimply on $L_x$ we see that the above decomposition exhibits the left hand side as a semisimple $\Gamma_x$-module with the summands on the right hand side having no simple constituents in common. The same then holds true for the decomposition
\begin{equation*}
  \ell_x^n \mathcal{O}_x \otimes_L V / \ell_x^m \mathcal{O}_x \otimes_L V = (\ell_x^n L_x \otimes_L V) \oplus \ldots \oplus (\ell_x^{m-1} L_x \otimes_L V) \ .
\end{equation*}
It therefore follows that
\begin{equation*}
  \mathcal{L} / \ell_x^m \mathcal{O}_x \otimes_L V = \ell_x^n (L_x \otimes_L V)_{-n} \oplus \ldots \oplus \ell_x^{m-1} (L_x \otimes_L V)_{-(m-1)}
\end{equation*}
with $\Gal(L_x/L)$-invariant $L_x$-vector subspaces $(L_x \otimes_L V)_{-j} \subseteq L_x \otimes_L V$. By Galois descent we have $(L_x \otimes_L V)_{-j} = L_x \otimes_L V_{-j}$ for a unique $L$-vector subspace $V_{-j} \subseteq V$. Hence
\begin{equation*}
  \mathcal{L} / \ell_x^m \mathcal{O}_x \otimes_L V = (\ell_x^n L_x \otimes_L V_{-n}) \oplus \ldots \oplus (\ell_x^{m-1} L_x \otimes_L V_{-(m-1)}) \ .
\end{equation*}
Multiplying this identity by $\ell_x$ shows that $V_{-j} \subseteq V_{-(j+1)}$. We deduce that
\begin{equation*}
  \mathcal{L} = (\ell_x^n \mathcal{O}_x \otimes_L V_{-n}) + \ldots + (\ell_x^{m-1} \mathcal{O}_x \otimes_L V_{-(m-1)}) + (\ell_x^m \mathcal{O}_x \otimes_L V) = \sum_{i \in \ZZ} \mathfrak{m}_x^i \otimes_L V_{-i}
\end{equation*}
with $V_{-i} := V$, resp.\ $:= \{0\}$, for $i \geq m$, resp.\ $i > n$. In particular, we obtain $\mathfrak{m}_x^i \otimes_L V_{-i} \subseteq \mathcal{L}$, hence $V_{-i} \subseteq \mathfrak{m}_x^{-i} \mathcal{L}$, and therefore $V_{-i} \subseteq \Fil^{-i} V$. We conclude that
\begin{equation*}
  \mathcal{L} \subseteq \sum_{i \in \ZZ} \mathfrak{m}_x^i \otimes_L \Fil^{-i} V \ .
\end{equation*}
The reverse inclusion is immediate from the definition of $\Fil^\bullet V$.
\end{proof}

\subsection{Crystalline $(\varphi_L,\Gamma_L)$-modules over $\mathscr{R}_L(\mathfrak{X})$}
\label{kisrobx}

We now consider the $\mathscr{R}_L(\mathfrak{X})$-module
\begin{equation*}
  \mathcal{M}_\mathscr{R}(D) := \mathscr{R}_L(\mathfrak{X}) \otimes_{\mathcal{O}_L(\mathfrak{X})} \mathcal{M}(D) \ .
\end{equation*}
Let $r_0 \in (0,1) \cap p^\QQ$ be such that $Z_0 \subseteq \mathfrak{X}(r_0)$. Then the inclusion of coherent ideal sheaves $\widetilde{I_{Z_0}} \subseteq \mathcal{O}$ is an isomorphism over
$\mathfrak{X} \setminus \mathfrak{X}(r_0)$. It follows, using Remark \ref{coherent}.iv, that $\mathcal{O}_L(\mathfrak{X} \setminus \mathfrak{X}(r_0)) \otimes_{\mathcal{O}_L(\mathfrak{X})} I_{Z_0} = \mathcal{O}_L(\mathfrak{X} \setminus \mathfrak{X}(r_0))$ and hence that $I_{Z_0}$ generates the unit ideal in $\mathcal{O}_L(\mathfrak{X} \setminus \mathfrak{X}(r_0))$. This implies that $\mathcal{O}_L(\mathfrak{X})[Z_0^{-1}] \subseteq \mathcal{O}_L(\mathfrak{X} \setminus \mathfrak{X}(r_0))$. Using Cor.\ \ref{iso-Z0invert}.ii we deduce that the $\mathcal{O}_L(\mathfrak{X} \setminus \mathfrak{X}(r_0))$-linear map
\begin{align}\label{f:tilde-varphi-iso}
    \widetilde{\varphi}_{\mathcal{M}(D)} : \mathcal{O}_L(\mathfrak{X} \setminus \mathfrak{X}(r_0)) \otimes_{\mathcal{O}_L(\mathfrak{X}),\varphi_L} \mathcal{M}(D) & \xrightarrow{\; \cong \;} \mathcal{O}_L(\mathfrak{X} \setminus \mathfrak{X}(r_0)) \otimes_{\mathcal{O}_L(\mathfrak{X})} \mathcal{M}(D) \\
    f \otimes s & \longmapsto f \varphi_{\mathcal{M}(D)} (s) \nonumber
\end{align}
is an isomorphism.

We now define the finitely generated projective module
\begin{equation*}
    \mathcal{M}_\mathscr{R}(D) := \mathscr{R}_L(\mathfrak{X}) \otimes_{\mathcal{O}_L(\mathfrak{X})} \mathcal{M}(D)
\end{equation*}
over the Robba ring. Its rank is equal to the dimension $d_D$ of $D$. The $\varphi_L$-linear endomorphisms $\varphi_{\mathcal{M}_\mathscr{R}(D)} := \varphi_L \otimes \varphi_{\mathcal{M}(D)}$ and $c_{\mathcal{M}_\mathscr{R}(D)} := c_* \otimes c_{\mathcal{M}(D)}$, for $c \in o_L^\times$, which by \eqref{f:commute} commute with $\varphi_{\mathcal{M}_\mathscr{R}(D)}$, together define a semilinear  action of the monoid $o_L \setminus \{0\}$ on $\mathcal{M}_\mathscr{R}(D)$ (which does not depend on the choice of $\pi_L$). As a consequence of \eqref{f:tilde-varphi-iso} the induced $\mathscr{R}_L(\mathfrak{X})$-linear map
\begin{equation*}
    \mathscr{R}_L(\mathfrak{X}) \otimes_{\mathscr{R}_L(\mathfrak{X}), \varphi_L} \mathcal{M}_\mathscr{R}(D) \xrightarrow[\id \otimes \varphi_{\mathcal{M}_\mathscr{R}(D)}]{\; \cong \;} \mathcal{M}_\mathscr{R}(D)
\end{equation*}
is an isomorphism.

\begin{corollary}
\label{mrispgm}
The module $\mathcal{M}_\mathscr{R}(D)$ is a $(\varphi_L,\Gamma_L)$-module over $\mathscr{R}_L(\mathfrak{X})$.
\end{corollary}

In this way we have constructed a functor
\begin{equation*}
    D \longmapsto \mathcal{M}_\mathscr{R}(D)
\end{equation*}
from the category of filtered $\varphi_L$-modules into the category $\Mod_L(\mathscr{R}_L(\mathfrak{X}))$, which is exact as can be seen from the description of the stalks in Lemmas \ref{fg-proj} and \ref{stalks}.
By Lemma \ref{action-MD}, $\mathcal{M}_\mathscr{R}(D)$ is $L$-analytic.

Those $(\varphi_L,\Gamma_L)$-modules that come by extension of scalars from a $(\varphi_L,\Gamma_L)$-module over $\mathcal{O}_L(\mathfrak{X})$ are said to be of finite height. The above constructions show that $\mathcal{M}_\mathscr{R}(D)$ is of finite height, at least when $\Fil^0 D = D$. We can get further examples by allowing an action of $\Gamma_L$ on $D$ that commutes with $\varphi_L$.

The construction of the equivalence of categories
\begin{equation*}
  \mathcal{M}_{\mathfrak{X}} := \mathcal{M} : \Ve(\Fil,\varphi_L)  \xrightarrow{\;\simeq\;} \Mod^{\varphi_L,\Gamma_L,\mathrm{an}}_{/\mathfrak{X}}
\end{equation*}
in Thm.\ \ref{Wach-equiv} works literally in the same way over $\mathfrak{B}$ producing an equivalence
\begin{equation*}
  \mathcal{M}_{\mathfrak{B}} : \Ve(\Fil,\varphi_L)  \xrightarrow{\;\simeq\;} \Mod^{\varphi_L,\Gamma_L,\mathrm{an}}_{/\mathfrak{B}}
\end{equation*}
with the property that $\mathscr{R}_L(\mathfrak{B}) \otimes_{\mathcal{O}_L(\mathfrak{B})} \mathcal{M}_{\mathfrak{B}}(D)$ is a  $(\varphi_L,\Gamma_L)$-module over $\mathscr{R}_L(\mathfrak{B})$. This case is due to \cite{KFC} and \cite{KR} \S 2.2.

\begin{lemma}\label{O-iso}
Under the identification $\mathfrak{B}_{/\Cp} = \mathfrak{X}_{/\Cp}$ via $\kappa$ we have $\mathcal{O}_{\Cp} (\mathfrak{X}) \otimes_{\mathcal{O}_L(\mathfrak{B})} \mathcal{M}_{\mathfrak{X}}(D) = \mathcal{O}_{\Cp} (\mathfrak{B}) \otimes_{\mathcal{O}_L(\mathfrak{X})} \mathcal{M}_{\mathfrak{B}}(D)$.
\end{lemma}
\begin{proof}
Let $\mathbf{S}$ be either $\mathfrak{X}$ or $\mathfrak{B}$ over $L$ with structure sheaf $\mathcal{O}_{\mathbf{S}}$. Let $t_{\mathbf{S}}$ denote $\log_\mathbf{S}$ and recall that $\log_\mathfrak{X} = \Omega_{t_0'} \cdot \log_{\mathfrak{B}}$ (compare the proof of Lemma \ref{zeros}.i). We consider the $\mathcal{O}_{\Cp} (\mathbf{S})$-module
\begin{multline*}
   \mathcal{M}_{\Cp}(D) := \\
   \{ s \in \mathcal{O}_{\Cp} (\mathbf{S}) [t_{\mathbf{S}}^{-1}] \otimes_L D :
   (\id \otimes \varphi_D^{-n(x)})(s) \in \Fil^0(\Fr(\mathcal{O}_{\mathbf{S}_{/\Cp},y}) \otimes_L D) \ \text{for all $y \in
     Z(\Cp)$} \}.
\end{multline*}
It is, in fact, independent of the choice of $\mathbf{S}$. Therefore it suffices to show that
\begin{equation*}
  \mathcal{O}_{\Cp} (\mathbf{S}) \otimes_{\mathcal{O}_L(\mathbf{S})} \mathcal{M}_{\mathbf{S}}(D) = \mathcal{M}_{\Cp}(D) \ .
\end{equation*}
The left hand side is obviously included in the right hand side. By redoing Lemmas \ref{fg-proj} and \ref{stalks} over $\mathfrak{X}_{/\Cp}$ we see that $\mathcal{M}_{\Cp}(D)$ is a finitely generated projective $\mathcal{O}_{\Cp} (\mathbf{S})$-module such that the stalks $\widetilde{\mathcal{M}}_{\Cp,y}(D)$, for $y \in \mathbf{S}(\Cp)$, of the corresponding coherent module sheaf on $\mathbf{S}_{/\Cp}$ satisfy
\begin{equation*}
  \widetilde{\mathcal{M}}_{\Cp,y}(D)
  \begin{cases}
  = \mathcal{O}_{\mathbf{S}_{/\Cp},y} \otimes_L D & \text{if $y \not\in Z(\Cp)$}, \\
  \xrightarrow[\id \otimes \varphi_D^{-n(y)}]{\cong} \Fil^0(\Fr(\mathcal{O}_{\mathbf{S}_{/\Cp},y}) \otimes_L D) & \text{if $y \in Z(\Cp)$}.
  \end{cases}
\end{equation*}
On the other hand we deduce directly, by base change, from these same lemmas that the stalks of the coherent module sheaf corresponding to the finitely projective $\mathcal{O}_{\Cp} (\mathbf{S})$-module $\mathcal{O}_{\Cp} (\mathbf{S}) \otimes_{\mathcal{O}_L(\mathbf{S})} \mathcal{M}_{\mathbf{S}}(D)$ satisfy the very same formula as above. This shows the asserted equality.
\end{proof}

\begin{theorem}\label{mdbxcomp}
The $(\varphi_L,\Gamma_L)$-modules $\mathscr{R}_L(\mathfrak{X}) \otimes_{\mathcal{O}_L(\mathfrak{X})} \mathcal{M}_{\mathfrak{X}}(D)$ and $\mathscr{R}_L(\mathfrak{B}) \otimes_{\mathcal{O}_L(\mathfrak{B})} \mathcal{M}_{\mathfrak{B}}(D)$ correspond to each other via the equivalence in Thm.\ \ref{equiv}.
\end{theorem}
\begin{proof}
Given the construction of our equivalence, it is enough to show that
\begin{equation*}
  \mathscr{R}_{\Cp} (\mathfrak{X}) \otimes_{\mathcal{O}_L(\mathfrak{X})} \mathcal{M}_{\mathfrak{X}}(D) = \mathscr{R}_{\Cp} (\mathfrak{B}) \otimes_{\mathcal{O}_L(\mathfrak{B})} \mathcal{M}_{\mathfrak{B}}(D) \ .
\end{equation*}
But this is immediate from Lemma \ref{O-iso}.
\end{proof}

Let $D$ be a $1$-dimensional filtered $\varphi_L$-module,
and let $v$ be a basis of $D$. Let $t_H(D)$ denote the unique integer
$i$ such that $\operatorname{gr}^i(D) \neq 0$, and let $t_N(D) = v_L(\alpha)$
where $\varphi_D(v) = \alpha v$. If $D$ is any filtered $\varphi_L$-module,
let $t_H(D) = t_H(\det D)$ and $t_N(D) = t_N(\det D)$.
We say that $D$ is admissible if $t_H(D) = t_N(D)$ and if $t_H(D') \leq t_N(D')$ for every sub filtered $\varphi_L$-module $D'$ of $D$.

\begin{proposition}
\label{colfononx}
If $D$ is a filtered $\varphi_L$-module, then $\deg(\mathcal{M}_\mathscr{R}(D)) = t_N(D) - t_H(D)$. In particular,  $\mathcal{M}_\mathscr{R}(D)$ is \'etale if and only if $D$ is admissible.
\end{proposition}
\begin{proof}
Given Thm.\ \ref{mdbxcomp} and Prop.\ \ref{mxet}, this is a consequence of the analogous claim over $\mathfrak{B}$, which is proved in \S 2.3 of \cite{KR}.
\end{proof}

Let $V$ be a crystalline $L$-analytic representation of $G_L$ and let
$D = (\mathbf{B}_{\mathrm{cris},L} \otimes_L V)^{G_L}$. The filtered
$\varphi_L$-module $D$ is admissible (see \S 3.1 of \cite{KR}; note that
$D$ is not $\mathrm{D}_{\mathrm{cris}}(V)$ but the two are related by
a simple recipe given in ibid.). By Prop.\ \ref{colfononx},
the $(\varphi_L,\Gamma_L)$-module $\mathcal{M}_\mathscr{R}(D)$ is \'etale.
This gives us a functor from the category of crystalline $L$-analytic
representations of $G_L$ to the category of \'etale $L$-analytic
$(\varphi_L,\Gamma_L)$-modules over $\mathscr{R}_L(\mathfrak{X})$.
By Thm.\ \ref{mdbxcomp}, this functor is compatible with the one
given in Cor.\ \ref{repequiv}.


\begin{thebibliography}{B-GAL}

\bibitem[Ax]{Ax}
Ax J.: \emph{Zeros of polynomials over local fields---{T}he {G}alois
  action}, J. Algebra \textbf{15} (1970), 417--428.

\bibitem[Ber08]{BEQ}
Berger L.: \emph{\'{E}quations diff{\'e}rentielles {$p$}-adiques et
  {$(\varphi,{N})$}-modules filtr{\'e}s}, Ast\'erisque (2008), no.~319, 13--38.

\bibitem[Ber11]{LBGL}
Berger L.: \emph{La correspondance de {L}anglands locale {$p$}-adique pour {${\mathrm{GL}}_ 2({\mathbf{Q}}_p)$}}, Ast\'erisque (2011), no.~339, Exp. No. 1017, viii,
  157--180, S{\'e}minaire Bourbaki. Vol. 2009/2010. Expos{\'e}s 1012--1026.

\bibitem[Ber15]{PGMLAV}
Berger L.: \emph{Multivariable {$(\varphi,\Gamma)$}-modules and locally analytic vectors}, preprint, 2015.

\bibitem[BC08]{BC}
Berger L., Colmez P.: \emph{Familles de repr\'esentations de de
  {R}ham et monodromie {$p$}-adique}, Ast\'erisque (2008), no.~319, 303--337,
  Repr{\'e}sentations $p$-adiques de groupes $p$-adiques. I.
  Repr{\'e}sentations galoisiennes et $(\varphi,\Gamma)$-modules.

\bibitem[BC15]{STLAV}
Berger L., Colmez P.:  \emph{Th{\'e}orie de {S}en et vecteurs localement analytiques}, Ann.
  Sci. \'Ecole Norm. Sup. (2015), to appear.

\bibitem[BGR]{BGR}
Bosch S., G\"untzer U., Remmert R.: \emph{Non-{A}rchimedean analysis},
  Grundlehren der Mathematischen Wissenschaften, vol. 261, Springer-Verlag, Berlin, 1984.

\bibitem[B-A]{B-A}
Bourbaki N.: \emph{Alg\`ebre I}. Diffusion C.C.L.S, Paris 1970

\bibitem[B-AC]{B-AC}
Bourbaki N.: \emph{Alg\`ebre commutative}. Springer 2006

\bibitem[B-GAL]{B-GAL}
Bourbaki N.: \emph{Groupes et alg\`ebres de Lie}. Chap.\ II -- III. Hermann, Paris 1972

\bibitem[B-TG]{B-TG}
Bourbaki N.: \emph{Topologie G\'en\'erale}. Chap.\ 1 -- 4, 5 -- 10. Springer 2007

\bibitem[Bre]{ICBM}
Breuil C.: \emph{The emerging {$p$}-adic {L}anglands programme},
  Proceedings of the {I}nternational {C}ongress of {M}athematicians. {V}olume
  {II}, Hindustan Book Agency, New Delhi, 2010, pp.~203--230.

\bibitem[CE]{CE}
Cartan H., Eilenberg S.: \emph{Homological algebra}, Princeton
  Landmarks in Mathematics, Princeton University Press, Princeton, NJ, 1999.

\bibitem[CC]{CC98}
Cherbonnier F., Colmez P.: \emph{Repr\'esentations {$p$}-adiques
  surconvergentes}, Invent. Math. \textbf{133} (1998), no.~3, 581--611.

\bibitem[Col02]{Col}
Colmez P.: \emph{Espaces de {B}anach de dimension finie}, J. Inst. Math.
  Jussieu \textbf{1} (2002), no.~3, 331--439.

\bibitem[Col10]{CGL2}
Colmez P.:  \emph{Repr\'esentations de {${\mathrm{GL}}_2(\mathbf{Q}_p)$} et
  {$(\varphi,\Gamma)$}-modules}, Ast\'erisque (2010), no.~330, 281--509.

\bibitem[Co1]{Co1}
Conrad B.: \emph{Irreducible components of rigid spaces}, Ann. Inst. Fourier
(Grenoble) \textbf{49} (1999), no.~2, 473--541.

\bibitem[Co2]{Co2}
Conrad B.: \emph{Moishezon spaces in rigid geometry}, preprint, 2010.

\bibitem[DG]{DG}
Demazure M., Gabiel P.: \emph{Groupes alg\'ebriques. {T}ome {I}:
  {G}\'eom\'etrie alg\'ebrique, g\'en\'eralit\'es, groupes commutatifs}, Masson
  \& Cie, \'Editeur, Paris; North-Holland Publishing Co., Amsterdam, 1970.

\bibitem[Eme]{Eme}
Emerton M.: \emph{Locally analytic vectors in representations of locally
  $p$-adic analytic groups}, Memoirs of the AMS, to appear, 2011.

\bibitem[Fea]{Fea}
F\'eaux de Lacroix C.T.: \emph{Einige Resultate \"uber die topologischen Darstellungen $p$-adischer Liegruppen auf unendlich dimensionalen Vektorr\"aumen \"uber einem $p$-adischen K\"orper}. Thesis, K\"oln 1997. Schriftenreihe Math. Inst. Univ. M\"unster 3. Ser., vol.~23, Univ. M\"unster, M\"unster, 1999, pp.~x+111.

\bibitem[Fie]{Fie}
Fieseler K.-H.: \emph{Zariski's {M}ain {T}heorem f\"ur affinoide
  {K}urven}, Math. Ann. \textbf{251} (1980), no.~2, 97--110.

\bibitem[Fon]{Fon}
Fontaine J.-M.: \emph{Repr\'esentations {$p$}-adiques des corps locaux.
  {I}}, The {G}rothendieck {F}estschrift, {V}ol.\ {II}, Progr. Math., vol.~87,
  Birkh\"auser Boston, Boston, MA, 1990, pp.~249--309.

\bibitem[FX]{FX}
Fourquaux L., Xie B.: \emph{Triangulable
  {$\mathcal{O}_F$}-analytic {$(\varphi_q,\Gamma)$}-modules of rank $2$},
  Algebra Number Theory \textbf{7} (2013), no.~10, 2545--2592.

\bibitem[Gru]{Gru}
Gruson L.: \emph{Fibr\'es vectoriels sur un polydisque ultram\'etrique},
  Ann. Sci. \'Ecole Norm. Sup. (4) \textbf{1} (1968), 45--89.

\bibitem[Haz]{Haz}
Hazewinkel M.: \emph{Formal groups and applications}, AMS Chelsea
  Publishing, Providence, RI, 2012, Corrected reprint of the 1978 original.

\bibitem[Hum]{Hum}
Humphreys J.E.: \emph{Linear algebraic groups}, Springer-Verlag, New
  York-Heidelberg, 1975, Graduate Texts in Mathematics, No. 21.

\bibitem[Kap]{Kap}
Kaplansky I.: \emph{Modules over {D}edekind rings and valuation rings},
  Trans. Amer. Math. Soc. \textbf{72} (1952), 327--340.

\bibitem[Ked05]{KSF}
Kedlaya K., \emph{Slope filtrations revisited}, Doc. Math. \textbf{10}
  (2005), 447--525 (electronic).

\bibitem[Ked08]{KedAst}
Kedlaya K.: \emph{Slope filtrations for relative {F}robenius}, Ast\'erisque
  (2008), no.~319, 259--301, Repr{\'e}sentations $p$-adiques de groupes
  $p$-adiques. I. Repr{\'e}sentations galoisiennes et $(\varphi,\Gamma)$-modules.

\bibitem[Ked10]{Ked}
Kedlaya K.: \emph{{$p$}-adic differential equations}, Cambridge Studies in
  Advanced Mathematics, vol. 125, Cambridge University Press, Cambridge, 2010.

\bibitem[Kis]{KFC}
Kisin M.: \emph{Crystalline representations and {$F$}-crystals}, Algebraic
  geometry and number theory, Progr. Math., vol. 253, Birkh\"auser Boston,
  Boston, MA, 2006, pp.~459--496.

\bibitem[KR]{KR}
Kisin M., Ren W.: \emph{Galois representations and {L}ubin-{T}ate
  groups}, Doc. Math. \textbf{14} (2009), 441--461.

\bibitem[Lam]{Lam}
Lam T.Y.: \emph{Lectures on modules and rings}, Graduate Texts in Mathematics,
  vol. 189, Springer-Verlag, New York, 1999.

\bibitem[Lan]{Lan}
Lang S.: \emph{Cyclotomic fields {I} and {II}}, second ed., Graduate Texts
  in Mathematics, vol. 121, Springer-Verlag, New York, 1990, With an appendix
  by Karl Rubin.

\bibitem[Laz]{Laz}
Lazard M.: \emph{Groupes analytiques {$p$}-adiques}, Inst. Hautes \'Etudes
  Sci. Publ. Math. (1965), no.~26, 389--603.

\bibitem[NSW]{NSW}
Neukirch J., Schmidt A., Wingberg K.: \emph{Cohomology of
  number fields}, second ed., Grundlehren der Mathematischen Wissenschaften,
  vol. 323, Springer-Verlag, Berlin, 2008.

\bibitem[PGS]{PGS}
Perez-Garcia  C., Schikhof W.H.: \emph{Locally convex spaces over
  non-{A}rchimedean valued fields}, Cambridge Studies in Advanced Mathematics,
  vol. 119, Cambridge University Press, Cambridge, 2010.

\bibitem[Schi]{Schi}
Schikhof W.H.: \emph{Ultrametric calculus}, Cambridge Studies in Advanced
  Mathematics, vol.~4, Cambridge University Press, Cambridge, 2006.

\bibitem[Sch]{Sch}
Schneider P.: \emph{Points of rigid analytic varieties}, J. Reine Angew.
  Math. \textbf{434} (1993), 127--157.

\bibitem[NFA]{NFA}
Schneider P.: \emph{Nonarchimedean functional analysis}, Springer Monographs in
  Mathematics, Springer-Verlag, Berlin, 2002.

\bibitem[TdA]{TdA}
Schneider P.: \emph{Die Theorie des Anstiegs}. Lecture Notes, M\"unster 20006, avalaible at \texttt{http://wwwmath.uni-muenster.de/u/schneider/publ/lectnotes/index.html}


\bibitem[ST01]{ST}
Schneider P., Teitelbaum J.: \emph{{$p$}-adic {F}ourier theory}, Doc. Math.
  \textbf{6} (2001), 447--481 (electronic).


\bibitem[ST02]{STla}
Schneider P., Teitelbaum J.: \emph{Locally analytic distributions and
  {$p$}-adic representation theory, with applications to {${\mathrm{GL}}_ 2$}}, J.
  Amer. Math. Soc. \textbf{15} (2002), no.~2, 443--468 (electronic).

\bibitem[ST03]{ST0}
Schneider P., Teitelbaum J.: \emph{Algebras of {$p$}-adic distributions and admissible
  representations}, Invent. Math. \textbf{153} (2003), no.~1, 145--196.

\bibitem[Ser]{Ser}
Serre J.-P.: \emph{Abelian {$l$}-adic representations and elliptic
  curves}, Research Notes in Mathematics, vol.~7, A K Peters, Ltd., Wellesley,
  MA, 1998, With the collaboration of Willem Kuyk and John Labute, Revised
  reprint of the 1968 original.

\bibitem[Spr]{Spr}
Springer T.A.: \emph{Linear algebraic groups}, second ed., Modern Birkh\"auser
  Classics, Birkh\"auser Boston, Inc., Boston, MA, 2009.

\bibitem[Tat]{Tat}
Tate J.: \emph{{$p$}-divisible groups}, Proc. {C}onf. {L}ocal {F}ields
  ({D}riebergen, 1966), Springer, Berlin, 1967, pp.~158--183.

\bibitem[vRo]{vRo}
van Rooij A.:  \emph{Non-{A}rchimedean functional analysis}, Monographs
  and Textbooks in Pure and Applied Math., vol.~51, Marcel Dekker, Inc., New
  York, 1978.

\end{thebibliography}

\end{document}